\documentclass[reqno,11pt]{amsart}

\usepackage[top=2.0cm,bottom=2.0cm,left=3cm,right=3cm]{geometry}
\usepackage{amsthm,amsmath,amssymb,dsfont}
\usepackage{mathrsfs,amsfonts,functan,extarrows,mathtools}
\usepackage[colorlinks]{hyperref}

\usepackage{marginnote}
\usepackage{xcolor}
\usepackage{stmaryrd}
\usepackage{esint}
\usepackage{graphicx}
\usepackage{bbm}

\usepackage{tikz}
\usetikzlibrary{positioning}

\newtheorem{theorem}{Theorem}[section]
\newtheorem{proposition}{Proposition}[section]

\newtheorem{remark}{Remark}[section]
\newtheorem{lemma}{Lemma}[section]

\numberwithin{equation}{section}
\allowdisplaybreaks

\def\p{\partial}
\def\d{\mathrm{d}}
\def\no{\nonumber}
\def\R{\mathbb{R}}
\def\eps{\varepsilon}
\def\div{\mathrm{div}}
\def\u{\mathbf{u}}

\def\l{\langle}
\def\r{\rangle}
\def\exp{\mathrm{exp}}
\def\M{\mathfrak{M}}
\def\A{\mathcal{A}}
\def\B{\mathcal{B}}

\def\u{\mathfrak{u}}
\def\L{\mathcal{L}}

\def\P{\mathcal{P}}
\def\I{\mathcal{I}}

\newcounter{wronumber}

\begin{document}
\title[Compressible Euler limit from Boltzmann equation]
			{Compressible Euler limit from Boltzmann equation with Maxwell reflection boundary condition in half-space}

\author[Ning Jiang]{Ning Jiang}
\address[Ning Jiang]{\newline School of Mathematics and Statistics, Wuhan University, Wuhan, 430072, P. R. China}
\email{njiang@whu.edu.cn}

\author[Yi-Long Luo]{Yi-Long Luo${}^*$}
\address[Yi-Long Luo]
{\newline School of Mathematics, South China University of Technology, Guangzhou, 510641, P. R. China}
\email{luoylmath@scut.edu.cn}

\author[Shaojun Tang]{Shaojun Tang}
\address[Shaojun Tang]
		{\newline School of Mathematical Sciences,
			University of Science and Technology of China,
			Hefei, 230026, P. R. China }
\email{sjtang@ustc.edu.cn}
\thanks{${}^*$ Corresponding author \quad \today}

\maketitle

\begin{abstract}
   Starting from the local-in-time classical solution to the compressible Euler system with impermeable boundary condition in half-space, by employing the coupled weak viscous layers (governed by linearized compressible Prandtl equations with Robin boundary condition) and linear kinetic boundary layers, and the analytical tools in \cite{Guo-Jang-Jiang-2010-CPAM} and some new boundary estimates both for Prandtl and Knudsen layers, we proved the local-in-time existence of Hilbert expansion type classical solutions to the scaled Boltzmann equation with Maxwell reflection boundary condition with accommodation coefficient $\alpha_\eps=O(\sqrt{\eps})$ when the Knudsen number $\eps$ small enough. As a consequence, this justifies the corresponding case of formal analysis in Sone's books \cite{Sone-2002book, Sone-2007-Book}. This also extends the results in \cite{GHW-2020} from specular to Maxwell reflection boundary condition. Both of this paper and \cite{GHW-2020} can be viewed as generalizations of Caflisch's classic work \cite{Caflish-1980-CPAM} to the cases with boundary. \\


\end{abstract}





\section{Introduction}

\subsection{Boltzmann equation with Maxwell boundary condition}
Hydrodynamic limits from the Boltzmann equation has been one of the central problems in fluid dynamics and kinetic theory since late 1970's. Among many references, we list some standard books for thorough introduction to the Boltzmann equation and fluid limits, for example, \cite{Cercignani-1988, CIP-1994, SRM-book}. Despite the tremendous progress in this fields (which we will review later), the most important problem basically remains open: the limit from the Boltzmann equation to the compressible Euler system.

More precisely, the Boltzmann equation in the Euler scaling is the following form:
\begin{equation}\label{BE}
\left\{
\begin{array}{l}
\p_t F_\eps + v \cdot \nabla_x F_\eps = \frac{1}{\eps} B (F_\eps, F_\eps) \quad \qquad {\text{on}}\  \R_+ \times \R^3_+ \times \R^3 \,,\\[1.5mm]
F_\eps (0, x, v) = F_\eps^{in} (x,v) \geq 0 \qquad \ \, \qquad {\text{on}}\  \R^3_+ \times \R^3 \,,
\end{array}
\right.
\end{equation}
where the dimensionless number $\eps>0$ is the Knudsen number which is the ratio of mean free path and macroscopic length scale, and $F_\eps (t,x,v) \geq 0$ is the density of particles of velocity $v \in \R^3$ and position $x \in \R^3_+ = \{ x \in \R^3; x_3 > 0 \}$, a half-space. Moreover, the equation \eqref{BE} is imposed with the Maxwellian reflection boundary condition
\begin{equation}\label{MBC}
\gamma_- F_\eps = (1-\alpha_\eps) L \gamma_+ F_\eps + \alpha_\eps K \gamma_+ F_\eps \quad \textrm{on } \R_+ \times \Sigma_- \,,
\end{equation}
For the derivation from the dimensional Boltzamnn equation to the scaled form \eqref{BE}, the readers can find the details in the aforementioned books, or papers \cite{BGL1, BGL2}.

 The Boltzmann collision operator is defined as
\begin{equation}\label{B-collision}
  \begin{aligned}
    B (F_1, F_2) (v) = \int_{\R^3} \int_{\mathbb{S}^2} |v-u|^{\gamma_0} F_1 (u') F_2 (v') b(\theta) \d \omega \d u \\
    - \int_{\R^3} \int_{\mathbb{S}^2} |v-u|^{\gamma_0} F_1 (u) F_2 (v) b(\theta) \d \omega \d u \,,
  \end{aligned}
\end{equation}
where $u' = u + [(v-u) \cdot \omega] \omega$, $v' = u + [(v-u) \cdot \omega] \omega$, $\cos \theta = (v-u) \cdot \omega / |v - u|$, $0 < b (\theta) \leq C |\cos \theta|$, and $0 \leq \gamma_0 \leq 1$ means the hard potential. Such collision operators cover the hard-sphere interaction with an angular cutoff. The accommodation coefficient $\alpha_\eps \in [0,1]$ describes how much the molecules accommodate to the state of the wall. The special case $\alpha_\eps=0$ corresponds to specular reflection, while $\alpha_\eps=1$ refers to complete diffusion. Usually this coefficient can be taken the form $\chi \eps^\beta$ with $\beta\geq 0$. In this paper, we consider the form $\alpha_\eps = \sqrt{2 \pi} \eps^{\frac{1}{2}}$. Analytically, the cases $\beta \neq \frac{1}{2}, 0$ will be similar to the current paper (but with more complicated expansion of boundary layers). For the simplicity of the presentation, we will write it in an separate forthcoming paper. For the case $\beta=0$, the formal analysis will be the same (all the details are in Sone's books \cite{Sone-2002book, Sone-2007-Book}), except the leading order involves {\em nonlinear} compressible Prandtl equation, whose well-posedness (even local) in Sobolev type space is completely open.

We denote $n = (0, 0, -1)$ by the outward normal of $\R^3_+$. Let $\Sigma : = \p \R^3_+ \times \R^3$ be the phase space boundary of $\R^3_+ \times \R^3$. The phase boundary $\Sigma$ can be split by outgoing boundary $\Sigma_+$, incoming boundary $\Sigma_-$, and grazing boundary $\Sigma_0$:
\begin{equation*}
  \begin{aligned}
    & \Sigma_+ = \{ (x, v) : x \in \p \R^3_+ \,, v \cdot n = - v_3 > 0 \} \,, \\
    & \Sigma_- = \{ (x, v) : x \in \p \R^3_+ \,, v \cdot n = - v_3 < 0 \} \,, \\
    & \Sigma_0 = \{ (x, v) : x \in \p \R^3_+ \,, v \cdot n = - v_3 = 0 \} \,.
  \end{aligned}
\end{equation*}
Let $\gamma_\pm F = \mathbbm{1}_{\Sigma_{\pm}} F$. The specular-reflection $L \gamma_+ F_\eps$ and the diffuse-reflection part $K \gamma_+ F_\eps$ in \eqref{MBC} are
\begin{align*}
L \gamma_+ F_\eps (t, x, v) = & F_\eps (t, x, R_x v) \,, \ R_x v = v - 2 (v \cdot n) n = (\bar{v}, - v_3) \,, \\
K \gamma_+ F_\eps (t, x, v) = & \sqrt{2 \pi} M_w (t, \bar{x}, v) \int_{v \cdot n >0} \gamma_+ F_\eps ( v \cdot n ) \d v \,,
\end{align*}
respectively, where $M_w (t, \bar{x}, v)$ is the local Maxwellian distribution function corresponding to the wall (boundary) with the form
\begin{equation}\label{M_w}
  \begin{aligned}
    M_w (t, \bar{x}, v) = \frac{\rho_w (t, \bar{x})}{[2 \pi T_w (t, \bar{x})]^{\frac{3}{2}}} \exp\Big\{ - \frac{|v- u_w (t, \bar{x})|^2}{2 T_w (t, \bar{x})} \Big\} \,.
  \end{aligned}
\end{equation}
The $\rho_w, u_w, T_w$ are, respectively, density, velocity and temperature of the boundary. We also assume
$$u_{w,3} = 0\,,$$
which denotes the boundary wall is fixed. We remark that the wall Maxwellian $M_w$ in \eqref{M_w} plays a key role in this paper. It matches with the leading local Maxwellian from interior, see \eqref{M_w-Perturbation}, which gives the compatibility condition between the initial and boundary data. It is one of the main difference or difficulty of the current paper with \cite{GHW-2020} in which the wall Maxwellian did not appear.

\subsection{History of compressible Euler limits}
There have been many significant progress in the limits to incompressible fluids, such as Navier-Stokes, Stokes, or even Euler. We only list some representative results. One group of results is in the framework of DiPerna-Lions renormalized solutions of the Boltzmann equation \cite{DiPerna-Lions}, i.e., the so-called BGL program, which aimed to justify the limit to Leray solutions of the incompressible Navier-Stokes equations. This was initialized by Bardos-Golse-Levermore \cite{BGL1, BGL2}, and finished by Golse and Saint-Raymond \cite{Golse-SRM-04, Golse-SRM-09}. The corresponding results in bounded domain were carried out in \cite{Masmoudi-SRM-CPAM, Jiang-Masmoudi-CPAM}. We also should mention the incompressible Euler limits of Saint-Raymond \cite{SRM-ARMA03, SRM-IHP09}. Another group of results is in the framework of classical solutions, which are based on nonlinear energy method, semi-group method or hypercoercivity (the latter two further rely on the spectral analysis of the linearized Boltzmann operator), see \cite{Bardos-Ukai1991, Briant-JDE2015, BMM-AA2019, Gallagher-Tristani, Guo-CPAM06, Jang-Kim, Jiang-Xu-Zhao}.

Comparing to the incompressible limits listed above, the limit to the compressible Euler system from the Boltzmann equation is much limited. This is mainly due to our still very poor understanding of the well-posedness of compressible Euler system for which we do not even know how to define weak solutions (at least for multiple spatial dimensions). The only available global-in-time solution is the calibrated BV solution of Glimm in 1965 \cite{Glimm-1965}. Regarding the higher dimensions, there are local-in-time classical solutions (see standard textbook \cite{Majda-1984}). The compressible Euler limit from the Boltzmann equations dates back to Japanese school for analytical data and Caflisch's Hilbert expansion approach which are based on the local well-posedness of compressible Euler system, see \cite{Caflish-1980-CPAM, Nishida1978}. Later improvement of Caflisch result employing the recent progress on the $L^2\mbox{-}L^\infty$ estimate was also used to help justifying the linearized acoustic limit \cite{Guo-Jang-Jiang-2009-KRM, Guo-Jang-Jiang-2010-CPAM}. We should also mention the compressible Euler limit in the context of 1-D Riemann problem away from the initial time in \cite{HWWY-SIMA2013}.

In the domain with boundary, the situation is much more complicated. As formally (and numerically) analyzed by Japanese school (summarized in Sone's books \cite{Sone-2002book, Sone-2007-Book}), to derive the equations of compressible Euler system from the Boltzmann equations with Maxwell reflection condition (with non-zero accommodation coefficient) using Hilbert type expansion, there needed two coupled boundary layers: viscous and kinetic layers. The formal is also called {\em Prandtl layer}, the latter is {\em Knudsen layer} in the physics literatures. The reason that these two types of boundary layers are needed can be explained as follows: as well-known, the leading order in {\em interior} is compressible Euler system whose natural boundary condition is impermeable condition $\u\!\cdot\!n=0$. However, the local Maxwellian governed by  Euler system with this condition does not satisfy the Maxwell reflection boundary condition, except the specular reflection case (i.e., $\alpha_\eps=0$). So, kinetic layers with thickness $\eps$ are needed. In these layers, each term satisfies the linear kinetic boundary layer equations which requires {\em four} solvability conditions (by theorem of Golse-Perthame-Sulem \cite{Golse-Perthame-Sulem-1988-ARMA}), but the number of boundary conditions for compressible Euler system is {\em one}. This mismatch indicates there should be another layer with thickness $\sqrt{\eps}$: Prandtl layer. This name comes from the fact that in this layer, each term satisfies the famous Prandtl equations of compressible type. Specifically, if $\alpha_\eps=O(1)$ the leading term satisfies the {\em nonlinear} compressible Prandtl equations, the higher order terms satisfied the {\em linearized} compressible Prandtl equations. If $\alpha_\eps=O(\eps^\beta)$ with $\beta>0$, the boundary layers are weak, thus all the boundary terms appear in higher order. Thus, nonlinear Prandtl equation does not appear. The current paper aims to give a rigorous justification for the case $\alpha_\eps= O(\sqrt{\eps})$.

In \cite{GHW-2020}, Guo-Huang-Wang started to justify the compressible Euler limit using Hilbert expansion approach, formally derived in Sone's books. They considered the simplest case $\alpha_\eps=0$, i.e., the specular reflection. We emphasize that in this case, the local Maxwellians governed by compressible Euler system with impermeable condition indeed satisfy the specular reflection condition. In this sense, boundary layers seem not needed. However, if the Hilbert expansion is used, at higher orders, the boundary conditions do not match. As a consequence, both fluid and kinetic layers are needed.

In the context of fluid equations, the application of the compressible Prandtl equations is not new. It has been investigated in the vanishing viscosity limit of compressible Naver-Stokes equations. This phenomena is consistent with paper \cite{GHW-2020} and the current paper, because as well-known, the compressible Navier-Stokes system with viscosity and thermal conductivity of order $O(\eps)$ can be derived from the Boltzmann equation (more precisely, by Chapman-Enskog expansion). This means \cite{GHW-2020} and current paper include this limit, but with different boundary conditions for compressible Navier-Stokes: paper \cite{GHW-2020} corresponds to Neumann boundary, the current paper corresponds to mixed Robin boundary. Indeed, at fluid equations level, Xin-Yanagisawa \cite{Xin-Yana-CPAM1999} proved the zero viscosity limit of linearized compressible Navier-Stokes system with Dirichlet condition, in which the linearized compressible Prandtl equations were also used. This result was extended to the case including temperature equation by Ding and the first author of the current paper in \cite{Ding-Jiang}.

We would like to mention that the analysis of coupled viscous and kinetic layers was already used in the work of the first author of this paper with Masmoudi \cite{Jiang-Masmoudi-CPAM} in slightly different form with \cite{GHW-2020} and this paper. The paper \cite{Jiang-Masmoudi-CPAM} is about the incompressible Navier-Stokes limit, which in fact does not need viscous layer. However, the major issue of \cite{Jiang-Masmoudi-CPAM} is to demonstrate the role played by acoustic waves, which is compressible. Thus, the incompressible Navier-Stokes limit (which happens in the time scale $O(\frac{1}{\eps})$) includes the acoustic limit in short time scale $O(1)$. In this sense, the two boundary layers in \cite{Jiang-Masmoudi-CPAM} are the same as the current papers and \cite{GHW-2020} (again, we emphasize that all the layers in \cite{GHW-2020} and here are linear because they appear in higher order terms). The main difference is that \cite{Jiang-Masmoudi-CPAM} works in the framework of renormalized solutions, which can not be expanded. So the coupled viscous and kinetic layers are considered in {\em dual} form in the sense that these layers appear in test functions. More specifically, the kinetic layers equations are the same as here and \cite{GHW-2020}, but with different boundary conditions: in this paper, for Maxwell reflection, the wall Maxwellian is local and nontrivial, while in \cite{Jiang-Masmoudi-CPAM}, the wall Maxwellian is global, i.e., $M(v)$. In \cite{GHW-2020}, it is specular reflection, so there is no wall Maxwellian. For the viscous layer parts, in \cite{Jiang-Masmoudi-CPAM}, the stationary version, i.e., degenerate heat operators are used, while in \cite{GHW-2020} and current paper, the linearized compressible Prandtl equations are used. The more detailed technical difference between \cite{GHW-2020} and the current paper will be explained in Subsection \ref{Sec_1.5}.

\subsection{Hilbert expansion}
Throughout this paper, we use the notation $\bar{U} = (U_1, U_2)$ for any vector $U = (U_1, U_2, U_3) \in \R^3$. Moreover, for simplicity of presentations, we use the notation $V^0 : = V |_{x_3 = 0}$ for any symbol $V = V(x)$, which may be a function, vector or operator. For any derivative operator $D_x$, we denote by $ D_x V^0 = (D_x V)^0 \,.$

For the functional spaces,  $H^s$ denotes the Sobolev space $W^{s,2} (\R^3_+)$ with norm $\| \cdot \|_{H^s}$,  $\| \cdot  \|_2$ and $\| \cdot \|_\infty$ are the $L^2$-norm and $L^\infty$-norm in both $(x,v) \in \R^3_+ \times \R^3$ variables, and $\l \cdot, \cdot \r$ is the $L^2$-inner product.

For any function $G = G(t, \bar{x}, y, v)$, with $(t, \bar{x}, y, v) \in \R_+ \times \R^2 \times \overline{\R}_+ \times \R^3$,  the Taylor expansion at $y = 0$ is
\begin{equation}
  G = G^0 + \sum_{1 \leq l \leq N} \tfrac{y^l}{l !} G^{(l)} + \tfrac{y^{N+1}}{(N+1) !} \widetilde{G}^{(N+1)} \,,
\end{equation}
where the symbols
\begin{equation*}
  \begin{aligned}
    G^{(l)} = (\p_y^l G ) (t, \bar{x}, 0, v) \,, \ \widetilde{G}^{(N+1)} = ( \p_y^{N+1} G ) (t, \bar{x}, \eta, v) \ \textrm{for some } \eta \in (0, y).
  \end{aligned}
\end{equation*}

In this paper, we take the Hilbert expansion approach to rigorously justify the asymptotic behaviors of \eqref{BE}-\eqref{MBC} as $\eps \to 0$. As we already mentioned above, its corresponding formal analysis has basically been down in Sone's books \cite{Sone-2002book, Sone-2007-Book}. Analytically, the key of this approach is the estimate on the {\em remainder}, after suitable truncation.

Due to the thickness of viscous boundary layer is $\sqrt{\eps}$, and the accommodation coefficient $\alpha_\eps=O(\sqrt{\eps})$, we expand $F_\eps (t, x, v)$ by order $\sqrt{\eps}$. We remark that if $\alpha_\eps=O(\eps^\beta)$ with $\beta\neq \frac{1}{2}, 0$, the expansion will be more involved. This will be discussed in a separate forthcoming paper.

\subsubsection{Interior expansion}
The expansion in interior has the form
\begin{equation*}
  \begin{aligned}
    F_\eps (t, x, v) \thicksim \sum_{k \geq 0} \sqrt{\eps}^k F_k (t, x, v) \,.
  \end{aligned}
\end{equation*}
Plugging into \eqref{BE} and collecting the same orders,
\begin{equation}\label{Order_Anal_Interior}
\begin{aligned}
\sqrt{\eps}^{-2}:& \quad 0 = B(F_0, F_0)\,,\\
\sqrt{\eps}^{-1}:& \quad 0 = B(F_0, F_1) + B(F_1, F_0)\,,\\
\sqrt{\eps}^0:& \quad (\p_t + v \cdot \nabla_x) F_0 = B(F_0, F_2) + B(F_2, F_0) + B(F_1, F_1)\,,\\
\sqrt{\eps}^1:& \quad (\p_t + v \cdot \nabla_x) F_1 = B(F_0, F_3) + B(F_3, F_0) + B(F_1, F_2) + B(F_2, F_1)\,,\\
\cdots \cdots &\\
\sqrt{\eps}^k:& \quad (\p_t + v \cdot \nabla_x) F_k = B(F_0, F_{k+2}) + B(F_{k+2}, F_0) + \sum_{\substack{i+j=k+2\,,\\ i, j\ge 1}} B(F_i, F_j)\,.
\end{aligned}
\end{equation}
Then {\em H}-theorem implies that $F_0$ must be a local Maxwellian:
\begin{equation}\label{Maxwellian}
F_0 (t, x, v) : = \M (t, x, v)
= \frac{\rho(t, x)}{(2\pi T(t, x))^{\frac{3}{2}}} \exp \Big(- \frac{|v- \u (t, x)|^2}{2 T(t, x)}\Big)\,.
\end{equation}
 Here $(\rho, \u, T)$ represent the macroscopic density, bulk velocity and temperature, respectively. It is well-known that (see \cite{Sone-2007-Book}, for instance), as $\eps \to 0$, the solutions $F_\eps$ of the Boltzmann equation \eqref{BE} converge to a local Maxwellian $\M$ whose parameters $(\rho, \u, T)$ satisfy the compressible Euler system
\begin{equation}\label{Compressible_Euler_Sys}
\left\{
\begin{aligned}
&\p_t \rho + \div_x (\rho \u) = 0\,,\\
&\p_t (\rho \u) + \div_x (\rho \u \otimes \u) + \nabla p = 0 \,, \\
&\p_t \big[\rho \big( \tfrac{3}{2} T + \tfrac{1}{2} |\u|^2 \big) \big] + \div_x \big[ \rho \u \big( \tfrac{3}{2} T + \tfrac{1}{2} |\u|^2 \big) \big] + \div_x (p \u) = 0
\end{aligned}
\right.
\end{equation}
over $(t, x) \in \R_+ \times \R^3_+$ with the slip boundary condition
\begin{equation}\label{BC-CEuler}
  \begin{aligned}
    \u \cdot n |_{x_3 =0} =  -\u_3 |_{x_3 = 0} = - \u_3^0 = 0 \,,
  \end{aligned}
\end{equation}
where $p= \rho T$ is the pressure. We further impose the initial data (this can be realized by setting special form of the initial data of the Boltzmann equation \eqref{BE}):
\begin{equation}\label{IC-CEuler}
  \begin{aligned}
    (\rho, \u, T) (0, x) = (\rho^{in}, \u^{in}, T^{in}) (x)
  \end{aligned}
\end{equation}
with compatibility condition
\begin{equation*}
  \begin{aligned}
    \u^{in} \cdot n |_{x_3} = 0 \,.
  \end{aligned}
\end{equation*}

By \cite{Schochet-1986-CMP} or \cite{Chen-FMC-2007}, we have the following proposition.

\begin{proposition}\label{Proposition_Compressible_Euler}
	Let $s_0 \ge 3$ and $\rho_\#, T_\# > 0$ with $\rho_\# \sqrt{T_\#} < \frac{1}{3\sqrt{3}}$. Assume $(\rho^{in} - \rho_\#, \u^{in}, T^{in} - T_\#) \in H^{s_0} (\R^3_+)$ satisfies $
	0 < \tfrac{3}{4} \rho_\# \le \rho^{in} (x) \le \tfrac{5}{4} \rho_\# \,,\quad
	0 < \tfrac{3}{4} T_\# \le T^{in} (x) \le \tfrac{5}{4} T_\#  $. Then there is a $\tau >0$ such that the compressible Euler system \eqref{Compressible_Euler_Sys}-\eqref{BC-CEuler}-\eqref{IC-CEuler} admits a unique solution $(\rho, \u, T)$ such that
	$$
	(\rho - \rho_\#, \u, T - T_\#) \in C \big( [0,\tau]; H^{s_0} (\R^3_+) \big) \cap C^1 \big( [0,\tau]; H^{s_0-1} (\R^3_+) \big)
	$$
	and
	$$
	0 < \tfrac{1}{2} \rho_\# \le \rho(t, x) \le \tfrac{3}{2} \rho_\#\,, \quad
	0 < \tfrac{1}{2} T_\# \le T(t, x) \le \tfrac{3}{2} T_\#
	$$
	hold for any $(t, x) \in [0, \tau] \times \R^3_+$. Moreover, the following estimate holds:
	\begin{align}\label{Compressible_Euler_Bound}
	\| (\rho - \rho_\#, \u, T - T_\#) \|_{C \big( [0,\tau]; H^{s_0} (\R^3_+) \big) \cap C^1 \big( [0,\tau]; H^{s_0-1} (\R^3_+) \big)} \le C_0 \,.
	\end{align}
	Here the constants $\tau$, $C_0 >0$ depend only on the $H^{s_0}$-norm of $(\rho^{in} - \rho_\#, \u^{in}, T^{in} - T_\#)$.
\end{proposition}

By using the local Maxwellian $\M (t,x,v)$, the linearized collision operator $\L$ is defined as
\begin{equation*}
  \begin{aligned}
    \L g = - \frac{1}{\sqrt{\M}} \Big\{ B (\M, \sqrt{\M} g) + B (\sqrt{\M} g, \M) \Big\} \,.
  \end{aligned}
\end{equation*}
The null space $\mathcal{N}$ of $\L$ is spanned by (see \cite{Caflish-1980-CPAM}, for instance)
\begin{equation*}
  \begin{aligned}
    \tfrac{1}{\sqrt{\rho}} \sqrt{\M} \,, \ \tfrac{v_i  - \u_i}{\sqrt{\rho T}} \sqrt{\M} \ (i=1,2,3) \,, \ \tfrac{1}{\sqrt{6 \rho}} \left\{ \tfrac{|v - \u|^2}{T} - 3 \right\} \sqrt{\M} \,.
  \end{aligned}
\end{equation*}
The weighted $L^2$-norm
\begin{equation*}
\begin{aligned}
\| g \|^2_\nu = \int_{\R^3_+} \int_{\R^3} |g (x,v)|^2 \nu (v) \d x \d v \,,
\end{aligned}
\end{equation*}
is defined by the collision frequency $\nu (v) \equiv \nu (\M) (v)$
\begin{equation*}
\begin{aligned}
\nu (\M) = \int_{\R^3} \int_{\mathbb{S}^2} b (\theta) |v - v'|^{\gamma_0} \M (v') \d v' \d \omega \,.
\end{aligned}
\end{equation*}
Note that for given $0 \leq \gamma_0 \leq 1$,
\begin{equation}\label{nu-phi}
\begin{aligned}
\nu (\M) \thicksim \rho \l v \r^{\gamma_0} \,,
\end{aligned}
\end{equation}
where $\l v \r = \sqrt{1 + |v|^2}$. Let $\mathcal{P} g$ be the $L^2_v$ projection with respect to $\mathcal{N}$. Then it is well-known (see for example \cite{Caflish-1980-CPAM}) that there exists a positive number $c_0 > 0$ such that
\begin{equation}\label{Hypocoercivity}
  \begin{aligned}
    \l \L g, g \r \geq c_0 \| (\mathcal{I} - \mathcal{P}) g \|^2_\nu
  \end{aligned}
\end{equation}
for $(\rho, \u, T)$ given in Proposition \ref{Proposition_Compressible_Euler}.

For each $k \geq 1$,  $\frac{F_k}{\sqrt{\M}}$ can be decomposed as the macroscopic and microscopic parts:
\begin{equation*}
  \begin{aligned}
    \tfrac{F_k}{\sqrt{\M}} = & \mathcal{P} ( \tfrac{F_k}{\sqrt{\M}} ) + (\mathcal{I} - \mathcal{P}) ( \tfrac{F_k}{\sqrt{\M}} ) \\
    \equiv & \Big\{ \tfrac{\rho_k}{\rho} + u_k \cdot \tfrac{v - \u}{T} + \tfrac{\theta_k}{6 T} ( \tfrac{|v - \u|^2}{T} - 3 ) \Big\} \sqrt{\M} + (\mathcal{I} - \mathcal{P}) ( \tfrac{F_k}{\sqrt{\M}} ) \,.
  \end{aligned}
\end{equation*}
Following \cite{Guo-Jang-Jiang-2009-KRM}, for $k \geq 1$,
\begin{equation}\label{Ker-Arth-Inte}
(\mathcal{I} - \mathcal{P} ) \big(\tfrac{F_k}{\sqrt{\M}}\big) = \mathcal{L}^{-1} \Big( - \tfrac{(\p_t + v \cdot \nabla_x ) F_{k-2} - \sum_{\substack{i+j = k\,,\\ i, j \ge 1}} B (F_i, F_j) }{\sqrt{\M}} \Big) \,,
\end{equation}
and the fluid variables $(\rho_k, u_k, \theta_k)$ obeys the following linear hyperbolic system
\begin{equation}\label{Linear-Hyperbolic-Syst}
  \left\{
    \begin{aligned}
      & \p_t \rho_k + \div_x (\rho u_k + \rho_k \u) = 0 \,, \\
      & \rho \big( \p_t u_k + u_k \cdot \nabla_x \u + \u \cdot \nabla_x u_k \big) - \tfrac{\nabla_x (\rho T)}{\rho} \rho_k + \nabla_x \Big( \tfrac{\rho \theta_k + 3 T \rho_k}{3}\Big) = \mathcal{F}^\bot_u (F_k) \,, \\
      & \rho \Big( \p_t \theta_k + \u \cdot \nabla_x \theta_k + \tfrac{2}{3} \big( \theta_k \div_x \u + 3 T \div_x u_k \big) + 3 u_k \cdot \nabla_x T \Big)
      = \mathcal{G}^\bot_\theta (F_k)\,,
    \end{aligned}
  \right.
\end{equation}
where $(\rho, \u, T)$ is the smooth solution of compressible Euler equations \eqref{Compressible_Euler_Sys}, the source terms $\mathcal{F}^\bot_u (F_k)$ and $\mathcal{G}^\bot_\theta (F_k)$ are defined as
\begin{equation}
  \begin{aligned}
    & \mathcal{F}^\bot_{u, i} (F_k) = - \sum_{j=1}^3 \p_{x_j} \int_{\R^3} T \A_{ij} \tfrac{F_k}{\sqrt{\M}} \d v \ (i = 1,2,3) \,, \\
    & \mathcal{G}^\bot_\theta (F_k) = - \div_x \Big( 2 T^{\frac{3}{2}} \int_{\R^3} \B \tfrac{F_k}{\sqrt{\M}} \d v + \sum_{j=1}^3 2 T \u \cdot \int_{\R^3} \A \tfrac{F_k}{\sqrt{\M}} \d v \Big) - 2 \u \cdot \mathcal{F}^\bot_u (F_k)\,.
  \end{aligned}
\end{equation}
Here $\A \in \R^{3 \times 3}$ and $\B \in \R^3$ are the Burnett functions with entries
\begin{equation}
  \begin{aligned}
    \A_{ij} = & \Big\{ \tfrac{(v_i - \u_i) (v_j - \u_j)}{T} - \delta_{ij} \tfrac{|v - \u|^2}{3 T} \Big\} \sqrt{\M} \quad (1 \leq i, j \leq 3) \,, \\
    \B_i = & \tfrac{v_i - \u_i}{2 \sqrt{T}} \Big( \tfrac{|v - \u|^2}{T} - 5 \Big) \sqrt{\M} \quad (1 \leq i \leq 3) \,.
  \end{aligned}
\end{equation}
Finally, the initial data of \eqref{Linear-Hyperbolic-Syst} are imposed on
\begin{equation}\label{IC-Linear-Hyperbolic}
\begin{aligned}
(\rho_k, u_k, \theta_k) (0, x) = (\rho_k^{in}, u_k^{in}, \theta_k^{in}) (x) \in \R \times \R^3 \times \R \,, \quad k = 1,2,3 \,, \cdots .
\end{aligned}
\end{equation}

\subsubsection{Viscous boundary layer expansion}

In viscous layer, the scaled normal coordinate is needed:
\begin{equation}
  \begin{aligned}
    \zeta = \tfrac{x_3}{\sqrt{\eps}} \,.
  \end{aligned}
\end{equation}
As shown in Sone's book \cite{Sone-2007-Book},  the viscous boundary layer expansion has the form
\begin{equation}
  \begin{aligned}
    F^b_\eps (t, \bar{x}, \zeta) \thicksim \sum_{k \geq 1} \sqrt{\eps}^k F^b_k (t, \bar{x}, \zeta, v) \,,
  \end{aligned}
\end{equation}
where, throughout our paper, the far field condition is always assumed:
\begin{equation}
  \begin{aligned}
    F^b_k (t, \bar{x}, \zeta , v) \to 0 \,, \quad \textrm{as } \zeta \to + \infty \,.
  \end{aligned}
\end{equation}
By similar calculation in \cite{GHW-2020}, plugging $F_\eps + F^b_\eps$ into the Boltzmann equation \eqref{BE} gives
\begin{align}\label{Order_Ana}
\no \sqrt{\eps}^{-1}:\quad & 0= B(\mathfrak{M}^{0}, F^b_1) + B(F^b_1, \mathfrak{M}^{0})\,,\\
\no \sqrt{\eps}^0:\quad & v_3 \cdot \p_\zeta F^b_1 = \big[ B(\mathfrak{M}^{0}, F^b_2) + B(F^b_2, \mathfrak{M}^{0}) \big] + [ B(F^0_1, F^b_1) + B(F^b_1, F^0_1)] \\
\no & \qquad \qquad + B(F^b_1, F^b_1) + \zeta \big[ B(\mathfrak{M}^{(1)}, F^b_1) +  B(F^b_1, \mathfrak{M}^{(1)})\big] \,,\\
\no \sqrt{\eps}: \quad & \p_t F^b_1 + \bar{v} \cdot \nabla_{\bar{x}} F^b_1 + v_3 \cdot \p_\zeta F^b_2 =  \big[ B(\mathfrak{M}^{0}, F^b_3) + B(F^b_3, \mathfrak{M}^{0}) \big] \\
\no & \qquad \quad + \tfrac{\zeta}{1!} \big[ B(\mathfrak{M}^{(1)}, F^b_2) + B(F^b_2, \mathfrak{M}^{(1)}) \big]  + \tfrac{\zeta^2}{2!} \big[ B(\mathfrak{M}^{(2)}, F^b_1) + B(F^b_1, \mathfrak{M}^{(2)}) \big]\\
\no & \qquad \quad + \big[ B(F^0_1, F^b_2) + B(F^b_2, F^0_1) \big] + \big[ B(F^0_2, F^b_1) + B(F^b_1, F^0_2) \big] \\
\no & \qquad \quad + \tfrac{\zeta}{1!} \big[ B(F^{(1)}_1, F^b_1) + B(F^b_1, F^{(1)}_1) \big] + \big[ B(F^b_1, F^b_2) + B(F^b_2, F^b_1) \big]\,,\\
\no & \cdots \cdots\\
\no \sqrt{\eps}^k: \quad & \p_t F^b_k + \bar{v} \cdot \nabla_{\bar{x}} F^b_k + v_3 \cdot \p_\zeta F^b_{k+1} \\
\no & \qquad = \sum_{\substack{i+j=k+2\,,\\ i,j \ge 1 }} B(F^b_i, F^b_j) + \Big[ B(\mathfrak{M}^0, F^b_{k+2}) + B(F^b_{k+2}, \mathfrak{M}^0) \Big]\\
\no & \qquad \quad + \sum_{\substack{l+j=k+2\,,\\ 1\le l \le N\,, j \ge 1}} \tfrac{\zeta^l}{l!} \Big[ B(\mathfrak{M}^{(l)}, F^b_j) + B(F^b_j, \mathfrak{M}^{(l)}) \Big] \\
\no & \qquad \quad + \sum_{\substack{i+j=k+2\,,\\ i,j \ge 1 }} \Big[ B(F^0_i, F^b_j) + B(F^b_j, F^0_i) \Big] \\
& \qquad \quad + \sum_{\substack{i+j+l = k+2\,, \\ 1\le l \le N\,, i, j \ge 1}} \tfrac{\zeta^l}{l!} \Big[ B(F^{(l)}_i, F^b_j) + B(F^b_j, F^{(l)}_i) \Big] \quad {\text{for}}\ k \ge 1 \,,
\end{align}
where the Taylor expansion at $x_3 = 0$ is used:
\begin{equation*}
  \begin{aligned}
    \M = \M^0 + \sum_{1 \leq l \leq N} \tfrac{\zeta^l}{l !} \M^{(l)} + \tfrac{\zeta^{N+1}}{(N+1) !} \widetilde{\M}^{(N+1)} \,, \\
    F_i = F_i^0 + \sum_{1 \leq l \leq N} \tfrac{\zeta^l}{l !} F_i^{(l)} + \tfrac{\zeta^{N+1}}{(N+1) !} \widetilde{F}_i^{(N+1)} \,.
  \end{aligned}
\end{equation*}
Here the number $N \in \mathbb{N}_+$ will be chosen later.

Let
\begin{equation}
  \begin{aligned}
    f^b_k = \tfrac{F^b_k}{\sqrt{\M^0}} \,,
  \end{aligned}
\end{equation}
which can be decomposed as
\begin{equation*}
  \begin{aligned}
    f^b_k = \P^0 f^b_k + (\I - \P^0) f^b_k = \left\{ \tfrac{\rho^b_k}{\rho^0} + u^b_k \cdot \tfrac{v - \u^0}{T^0} + \tfrac{\theta^b_k}{6 T^0} ( \tfrac{|v - \u^0|^2}{T^0} - 3 ) \right\} \sqrt{\M^0} + (\I - \P^0) f^b_k \,.
  \end{aligned}
\end{equation*}
Here $u^b_k = (u^b_{k,1}, u^b_{k,2}, u^b_{k,3}) \in \R^3$. Furthermore, let
\begin{equation}\label{pb_k}
  \begin{aligned}
    p^b_k = \tfrac{\rho^0 \theta^b_k + 3 T^0 \rho^b_k}{3} \,.
  \end{aligned}
\end{equation}
Following the Theorem 1.1 of \cite{GHW-2020}, the following relations are obyed
\begin{equation}\label{ub_13-pb_1}
  \begin{aligned}
    u^b_{1,3} (t, \bar{x}, \zeta) \equiv 0 \,, \ p^b_1 (t, \bar{x}, \zeta) \equiv 0 \,, \ \forall (t, \bar{x}, \zeta) \in [0, \tau] \times \R^2 \times \R_+ \,,
  \end{aligned}
\end{equation}
and $(u^b_{k,1}, u^b_{k,2}, \theta^b_k)$ $(k \geq 1)$ satisfy the following {\em linear compressible Prandtl-type} equations\footnote{In \cite{GHW-2020}, the same system was named as a linear parabolic system.}
\begin{equation}\label{Linear-Prandtl}
  \left\{
    \begin{aligned}
      \rho^0 (\p_t + \bar{\u}^0 \cdot \nabla_{\bar{x}}) u^b_{k,i} + \rho^0 (\p_{x_3} \u^0_3 \zeta + u^0_{1,3}) \p_{\zeta} u^b_{k,i} + & \rho^0 \bar{u}^b_k \cdot \nabla_{\bar{x}} \u^0_i + \tfrac{\p_{x_3} p^0}{3 T^0} \theta^b_k \\
      = & \mu (T^0) \p_{\zeta}^2 u^b_{k,i} + \mathsf{f}^b_{k-1, i} \ (i=1,2) \,, \\
      \rho^0 \p_t \theta^b_k + \rho^0 \bar{\u}^0 \cdot \nabla_{\bar{x}} \theta^b_k + \rho^0 \big( \p_{x_3} \u^0 \zeta + \u^0_{1,3} \big) \p_\zeta \theta^b_k + & \tfrac{2}{3} \rho^0 \div_{x} \u^0 \theta^b_k \\
      = & \tfrac{3}{5} \kappa(T^0) \p_{\zeta\zeta} \theta^b_k + \mathsf{g}^b_{k-1} \,, \\
      \lim_{\zeta \to \infty} (\bar{u}^b_k, \theta^b_k) (t, \bar{x}, \zeta) = & 0 \,,
    \end{aligned}
  \right.
\end{equation}
and $(\I - \P^0) f^b_{k+1}$, $u^b_{k+1, 3}$, $p^b_{k+1}$ are determined by the equations
\begin{equation}\label{f_b_k+1-Kernal-Ortho}
  \begin{aligned}
    (\I - \P^0) & f^b_{k+1} = (\mathcal{L}^0)^{-1} \Big\{ - (\I - \P^0) (v_3 \p_\zeta \P^0 f^b_k) + \tfrac{\zeta}{\sqrt{\M^0}} \big[ B(\M^{(1)}, \sqrt{\M^0} \P^0 f^b_k) \\
   & + B(\sqrt{\M^0} \P^0 f^b_k, \M^{(1)})\big] + \tfrac{1}{\sqrt{\M^0}} \big[ B(F^0_1, \sqrt{\M^0} \P^0 f^b_k) + B(\sqrt{\M^0} \P^0 f^b_k, F^0_1) \big] \\
   & + \tfrac{1}{\sqrt{\M^0}} \big[ B(\sqrt{\M^0} f^b_1, \sqrt{\M^0} \P^0 f^b_k) + B(\sqrt{\M^0} \P^0 f^b_k, \sqrt{\M^0} f^b_1) \big] \Big\} + J^b_{k-1}\,,
  \end{aligned}
\end{equation}
and
\begin{equation}\label{u_b_k+1-derivative}
    \begin{aligned}
    \p_\zeta u^b_{k+1, 3} = - \frac{1}{\rho^0} ( \p_t \rho^b_k + \div_{\bar{x}} ( \rho^0 \bar{u}^b_k + \rho^b_k \bar{\u}^0 ) )\,, \ \lim_{\zeta \to \infty} u^b_{k+1,3} (t, \bar{x}, \zeta) = 0 \,,
    \end{aligned}
\end{equation}
and
\begin{equation}\label{p_b_k+1-derivative}
  \left\{
    \begin{aligned}
      \p_\zeta p^b_{k+1}
      = - \rho^0 \p_t u^b_{k,3} - \rho^0 \bar{\u}^0 \!\cdot\! \nabla_{\bar{x}} u^b_{k,3} + \rho^0 \p_{x_3} \u^0_3 u^b_{k,3} + \tfrac{4}{3} \mu(T^0) \p_{\zeta\zeta} u^b_{k,3} \\
      - \tfrac{4}{3} \rho^0 \p_\zeta \Big[ \big(\p_{x_3} \u^0 \zeta + u^0_{1,3} \big) u^b_{k,3} \Big] - \p_\zeta \l T^0 \A^0_{33}, J^b_{k-1} \r + W^b_{k-1,3} \,, \\
      \lim_{\zeta \to \infty} \theta^b_{k+1} (t, \bar{x}, \zeta) = 0 \,,
    \end{aligned}
  \right.
\end{equation}
where the source terms $\mathsf{f}^b_{k-1, i} \ (i = 1, 2)$ and $\mathsf{g}^b_{k-1}$ are
\begin{equation}
  \begin{aligned}
    \mathsf{f}^b_{k-1, i} = -\rho^0 \p_\zeta [ ( \p_{x_3} \u^0_i \zeta + u^0_{1,i} + u^b_{1,i} ) u^b_{k,3} ] - ( \p_{x_i} - \tfrac{\p_{x_i} p^0}{p^0} ) p^b_k \\
    + W^b_{k-1,i} - T^0 \p_\zeta \l J^b_{k-1}, \A^0_{3i} \r \,, \\
    \mathsf{g}^b_{k-1} = - \rho^0 \p_\zeta \Big[ \big( 3 \p_\zeta T^0 \zeta + \theta^0_1 + \theta^b_1 \big) u^b_{k,3}\Big] + \tfrac{3}{5} H^b_{k-1} - \tfrac{6}{5} (T^0)^{\frac{3}{2}} \p_\zeta \l J^b_{k-1}, \B^0_3 \r\\
    + \tfrac{3}{5} \big\{ 2 \p_t + 2 \bar{\u} ^0 \cdot \nabla_{\bar{x}} + \tfrac{10}{3} \div_{x} \u^0 \big\} p^b_k \,,
  \end{aligned}
\end{equation}
and
\begin{align}\label{W_b_k-1}
  W^b_{k-1,i} = - \sum_{j=1}^2 \p_{x_j} \l T^0 (\I- \P^0) f^b_k, \A^0_{ij} \r \,, \qquad {\text{for}}\ i =1,2,3\,,
\end{align}
\begin{align}\label{H_b_k-1}
  H^b_{k-1} = - \sum_{j=1}^2 \p_{x_j} \l (T^0)^{\frac{3}{2}} \B^0_j + \sum_{l=1}^2 2 T^0 \u^0_l \A^0_{jl},  (\mathcal{I} - \P^0) f^b_k \r - 2 \bar{\u}^0 \cdot \bar{W}^b_{k-1}\,,
\end{align}
and
  \begin{align}\label{J_b_k-1}
    \no J^b_{k-1} = & (\mathcal{L}^0)^{-1} \Big\{ - (\mathcal{I}-\mathcal{P}^0) \Big( \tfrac{1}{\sqrt{\M^0}} \big\{ \p_t + \bar{v} \cdot \nabla_{\bar{x}} \big\} F^b_{k-1} \Big) - (\mathcal{I}-\mathcal{P}^0) \big( v_3 \p_\zeta ( \mathcal{I} - \mathcal{P}^0 \big) f^b_k \big) \\
    \no & + \sum_{\substack{l+j=k+1\,,\\2 \le l \le N, j\ge 1}} \tfrac{\zeta^l}{l!} \tfrac{1}{\sqrt{\M^0}} \big[ B(\M^{(l)}, \sqrt{\M^0} f^b_j) + B(\sqrt{\M^0} f^b_j, \M^{(l)}) \big] \\
    \no & + \sum_{\substack{l+j=k+1\,,\\i\ge 2, j\ge 1}} \tfrac{1}{\sqrt{\M^0}} \big[ B(F^0_i, \sqrt{\M^0} f^b_j) + B(\sqrt{\M^0} f^b_j, F^0_i) \big] \\
    \no & + \sum_{\substack{l+j=k+1\,,\\ i,j\ge 2}} \tfrac{1}{\sqrt{\M^0}} \big[ B(\sqrt{\M^0} f^b_i, \sqrt{\M^0} f^b_j) + B(\sqrt{\M^0} f^b_j, \sqrt{\M^0} f^b_i) \big] \\
    \no & + \sum_{\substack{i+j+k = k+1\,,\\1\le l \le N, i,j \ge 1}} \tfrac{1}{\sqrt{\M^0}} \tfrac{\zeta^l}{l!} \big[ B(F^{(l)}_i, \sqrt{\M^0} f^b_j) + B(\sqrt{\M^0} f^b_j, F^{(l)}_i)\big] \\
    \no & + \tfrac{\zeta}{\sqrt{\M^0}} \big[ B(\M^{(1)}, \sqrt{\M^0} (\mathcal{I} - \mathcal{P}^0) f^b_k) + B(\sqrt{\M^0} (\mathcal{I} - \mathcal{P}^0) f^b_k, \M^{(1)}) \big] \\
    \no & + \tfrac{1}{\sqrt{\M^0}} \big[ B(F^0_1, \sqrt{\M^0} (\mathcal{I} - \mathcal{P}^0) f^b_k) + B(\sqrt{\M^0} (\mathcal{I} - \mathcal{P}^0) f^b_k, F^0_1) \big] \\
    & + \tfrac{1}{\sqrt{\M^0}} \big[ B(\sqrt{\M^0} f^b_1, \sqrt{\M^0} (\mathcal{I} - \mathcal{P}^0) f^b_k) + B(\sqrt{\M^0} (\mathcal{I} - \mathcal{P}^0) f^b_k, \sqrt{\M^0} f^b_1) \big] \Big\} \,.
  \end{align}
We emphasize that $W^b_{k-1}$, $H^b_{k-1}$ and $J^b_{k-1}$ depend on $f^b_j$ $(1 \le j \le k-1)$. Moreover, when $k=1$,  $J^b_0 = W^b_0 = H^b_0 = \mathsf{f}^b_0 = \mathsf{g}^b_0 = 0$. Actually, by \eqref{BC-u_13} below, $u_{1,3}^0 = 0$ in \eqref{Linear-Prandtl}. Finally, the initial conditions of \eqref{Linear-Prandtl} are imposed on
\begin{equation}\label{IC-Prandtl}
\begin{aligned}
(\bar{u}^b_k, \theta^b_k) (0, \bar{x}, \zeta) = (\bar{u}^{b, in}_k , \theta^{b, in}_k ) (\bar{x}, \zeta) \in \R^2 \times \R \,, \quad k = 1,2,3 \,, \cdots
\end{aligned}
\end{equation}
with $\lim_{\zeta \to \infty} (\bar{u}^{b, in}_k , \theta^{b, in}_k) (\bar{x}, \zeta) = 0$.

\subsubsection{Knudsen boundary layer expansion}

In Knudsen layer, the new scaled normal coordinate is introduced:
\begin{equation}
  \begin{aligned}
    \xi = \tfrac{x_3}{\eps} \,.
  \end{aligned}
\end{equation}
Then the Knudsen boundary expansion is defined as
\begin{equation}
  \begin{aligned}
    F^{bb}_\eps (t, \bar{x}, \xi, v) \thicksim \sum_{k \geq 1} \sqrt{\eps}^k F^{bb}_k (t, \bar{x}, \xi, v) \,.
  \end{aligned}
\end{equation}
From plugging $F_\eps + F^b_\eps + F^{bb}_\eps$ in \eqref{BE},
\begin{equation}\label{Order_Anal_Knudsen}
\begin{aligned}
\sqrt{\eps}^{-1} :\qquad & v_3 \cdot \p_\xi F^{bb}_1 = B(\M^0, F^{bb}_1) + B(F^{bb}_1, \M^0)\\
\sqrt{\eps}^0 : \qquad & v_3 \cdot \p_\xi F^{bb}_2 - \big[ B(\M^0, F^{bb}_2) + B(F^{bb}_2, \M^0) \big] \\  = & B(F^0_1 + F^{b,0}_1, F^{bb}_1) + B(F^{bb}_1, F^0_1 + F^{b,0}_1) + B(F^{bb}_1, F^{bb}_1)\\
......&\\
\sqrt{\eps}^k :\qquad & v_3 \cdot \p_\xi F^{bb}_{k+2} - \big[ B(\M^0, F^{bb}_{k+2}) + B(F^{bb}_{k+2}, \M^0) \big] \\
= & - \big\{ \p_t + \bar{v} \cdot \nabla_{\bar{x}}
\big\} F^{bb}_k + \sum_{\substack{j+2l=k+2\,,\\ 1 \le l \le N, j\ge 1}} \frac{\xi^l}{l!} \big[ B(\M^{(l)}, F^{bb}_j) + B(F^{bb}_j, \M^{(l)})\big] \\
& + \sum_{\substack{i+j=k+2\,,\\ i, j\ge 1}} \big[ B(F^0_i + F^{b,0}_i, F^{bb}_j) + B(F^{bb}_j, F^0_i + F^{b,0}_i) + B(F^{bb}_i, F^{bb}_j) \big] \\
& + \sum_{\substack{i+2l+j=k+2\,,\\1\le l \le N, i,j\ge 1}} \frac{\xi^l}{l!} \big[ B(F^{(l)}_i, F^{bb}_j) + B(F^{bb}_j, F^{(l)}_i)\big] \\
& + \sum_{\substack{i+l+j=k+2\,,\\1\le l \le N, i,j\ge 1}} \frac{\xi^l}{l!} \big[ B(F^{b, (l)}_i, F^{bb}_j) + B(F^{bb}_j, F^{b, (l)}_i)\big]\,,
\end{aligned}
\end{equation}
where the Taylor expansion of $F^b_i$ at $\zeta = 0$ is utilized:
\begin{equation*}
  \begin{aligned}
    F_i^b = F_i^{b,0} + \sum_{1 \leq l \leq N} \tfrac{\xi^l}{l!} F_i^{b, (l)} + \tfrac{\xi^{N+1}}{(N+1)!} \widetilde{F}_i^{b, (N+1)} \,.
  \end{aligned}
\end{equation*}

Similar in viscous layer, let $f^{bb}_k = \tfrac{F^{bb}_k}{\sqrt{\M^0}}$, then \eqref{Order_Anal_Knudsen} can be rewritten as
\begin{equation}\label{fbb_k}
  \begin{aligned}
    v_3 \p_\xi f^{bb}_k + \L^0 f^{bb}_k = S^{bb}_k \,, \quad k \geq 1 \,,
  \end{aligned}
\end{equation}
where $S^{bb}_k = S^{bb}_{k,1} + S^{bb}_{k,2}$ with
\begin{align}\label{Sbb_k}
\no S^{bb}_{k,1} =& - \mathcal{P}^0 \Big\{ \frac{(\p_t + \bar{v} \cdot \nabla_{\bar{x}}) F^{bb}_{k-2}}{\sqrt{\M^0}} \Big\} \in \mathcal{N}^0 \,, \\
\no S^{bb}_{k,2} = & \sum_{\substack{j+2l=k\,,\\ 1 \le l \le N, j\ge 1}} \frac{\xi^l}{l!} \frac{1}{\sqrt{\M^0}} \big[ B(\M^{(l)}, \sqrt{\M^0} f^{bb}_j) + B(\sqrt{\M^0} f^{bb}_j, \M^{(l)})\big] \\
\no & + \sum_{\substack{i+j=k\,,\\ i, j\ge 1}} \frac{1}{\sqrt{\M^0}} \big[ B(F^0_i + F^{b,0}_i, \sqrt{\M^0} f^{bb}_j) + B(\sqrt{\M^0} f^{bb}_j, F^0_i + F^{b,0}_i)\big] \\
\no & + \sum_{\substack{i+2l+j=k\,,\\1\le l \le N, i,j\ge 1}} \frac{\xi^l}{l!} \frac{1}{\sqrt{\M^0}} \big[ B(F^{(l)}_i, \sqrt{\M^0} f^{bb}_j) + B(\sqrt{\M^0} f^{bb}_j, F^{(l)}_i)\big] \\
& + \sum_{\substack{i+l+j=k\,,\\1\le l \le N, i,j\ge 1}} \frac{\xi^l}{l!} \frac{1}{\sqrt{\M^0}} \big[ B(F^{b, (l)}_i, \sqrt{\M^0} f^{bb}_j) + B(\sqrt{\M^0} f^{bb}_j, F^{b, (l)}_i)\big] \\
\no & + \sum_{\substack{i+j=k\,,\\ i,j\ge 1}} \frac{1}{\sqrt{\M^0}} B(\sqrt{\M^0} f^{bb}_i, \sqrt{\M^0} f^{bb}_j) - (\mathcal{I} - \mathcal{P}^0) \Big\{ \frac{\{\p_t + \bar{v} \cdot \nabla_{\bar{x}} \} F^{bb}_{k-2}}{\sqrt{\M^0}} \Big\} \in (\mathcal{N}^0)^\bot \,.
\end{align}
Here  the notation $F^{bb}_{-1} = F^{bb}_0=0 $ are used, and
\begin{align*}
  S^{bb}_1 = S^{bb}_{1,1} = S^{bb}_{1,2} =0\,,\ S^{bb}_{2,1} = \mathcal{P}^0 S^{bb}_2 =0 \,.
\end{align*}

\begin{lemma}[\cite{Bardos-Caflisch-Nicolaenko-1986-CPAM}]\label{Lmm-fbb-k1}
	We assume that
	\begin{equation}\label{S_bb_k1}
	  \begin{aligned}
	    S^{bb}_{k,1}= \big\{ a_k + b_k \cdot (v- \u^0) + c_k |v-\u^0|^2 \big\} \sqrt{\M^0}
	  \end{aligned}
	\end{equation}
	satisfy
	\begin{align*}
	  \lim_{\xi \to \infty} e^{\eta \xi} |(a_k, b_k, c_k) (t, \bar{x}, \xi)| = 0
	\end{align*}
	for some positive constant $\eta > 0$. Then there exists a function
	\begin{equation*}
	f^{bb}_{k,1} = \big\{ \Psi_k v_3 + \Phi_{k,1} v_3 (v_1 - \u^0_1) + \Phi_{k,2} v_3 (v_2 - \u^0_2) + \Phi_{k,3} + \Theta_k v_3 |v- \u^0|^2 \big\} \sqrt{\M^0} 	
	\end{equation*}
	such that $v_3 \p_\xi f^{bb}_{k,1} - S^{bb}_{k,1} \in (\mathcal{N}^0)^\bot$, where
	\begin{equation}\label{f_bb_k1_Coef}
	\begin{array}{l}
	\Psi_k (t, \bar{x}, \xi)
	= - \int_\xi^{+\infty} \big( \tfrac{2}{T^0} a_k + 3 c_k \big) (t, \bar{x}, s) \d s\,, \\ [1.5mm]
	\Phi_{k,i} (t, \bar{x}, \xi)
	= - \int_\xi^{+\infty} \tfrac{1}{T^0} b_{k,i} (t, \bar{x}, s) \d s\,,\ i=1,2\,, \\ [1.5mm]
	\Phi_{k,3} (t, \bar{x}, \xi)
	= - \int_\xi^{+\infty} b_{k,3} (t, \bar{x}, s) \d s\,, \\ [1.5mm]
	\Theta_k (t, \bar{x}, \xi) = \tfrac{1}{5(T^0)^2} \int_\xi^{+\infty} a_k (t, \bar{x}, s) \d s\,.
	\end{array}
	\end{equation}
	Moreover, there holds
	\begin{equation*}
	  \begin{aligned}
	    & |v_3 \p_\xi f^{bb}_{k,1} - S^{bb}_{k,1}| \leq C |(a_k, b_k, c_k) (t, \bar{x}, \xi)| \l v \r^4 \sqrt{\M^0} \,, \\
	    & |f^{bb}_{k,1} (t, \bar{x}, \xi, v)| \leq C \l v \r^3 \sqrt{\M^0} \int_\xi^\infty |(a_k, b_k, c_k)| \to 0 \textrm{ as } \xi \to \infty \,.
	  \end{aligned}
	\end{equation*}
\end{lemma}
It is easy to know that $a_1 = a_2 =0 $, $ b_1 =b_2=0$, $c_1 =c_2 =0$, $\Psi_1=\Psi_2 =0\,,\ \Phi_{1,i} = \Phi_{2,i} =0 \ (i = 1,2,3)$ and $ \Theta_1 = \Theta_2=0$. Denote by $f^{bb}_{k,2} = f^{bb}_k - f^{bb}_{k,1}$. We thereby see that
\begin{equation}\label{KBL-fbb_k2}
  \begin{aligned}
    v_3 \p_\xi f^{bb}_{k,2} + \L^0 f^{bb}_{k,2} = S^{bb}_{k,2} - (v_3 \p_\xi f^{bb}_{k,1} - S^{bb}_{k,1} ) \in (\mathcal{N}^0)^\bot \,, \ \lim_{\xi \to \infty} f^{bb}_{k,2} (t, \bar{x}, \xi, v) = 0 \,.
  \end{aligned}
\end{equation}
Once we impose the following boundary condition on \eqref{KBL-fbb_k2}:
\begin{equation}\label{BC-KBL-fbb_k2}
  \begin{aligned}
    f^{bb}_{k,2} (t, \bar{x}, 0, \bar{v}, v_3) |_{v_3 > 0} = f^{bb}_{k,2} (t, \bar{x}, 0, \bar{v}, - v_3) + \mathbbm{f}_k (t, \bar{x}, \bar{v}, - v_3)
  \end{aligned}
\end{equation}
for some function $\mathbbm{f}_k (t, \bar{x}, \bar{v}, v_3)$ only defined for $v_3 < 0$ and extended to be 0 for $v_3 > 0$, Golse, Perthame and Sulem \cite{Golse-Perthame-Sulem-1988-ARMA} proved that the solvability conditions of \eqref{KBL-fbb_k2}-\eqref{BC-KBL-fbb_k2} were
\begin{equation}\label{Solva_Cond}
  \int_{\R^3}
    \left(
      \begin{array}{c}
        1 \\
        \bar{v} - \bar{\u}^0 \\
        |v - \u^0|^2
      \end{array}
    \right)
    v_3 \mathbbm{f}_k (t, \bar{x}, v) \sqrt{\M^0} \d v \equiv 0 \,.
\end{equation}

\subsubsection{Expansions of the Maxwell reflection boundary condition \eqref{MBC}}

In order to give suitable boundary conditions so that the interior expansions, viscous and Knudsen boundary layers are all well-posed,  the expansions of $F_\eps + F^b_\eps + F^{bb}_\eps$ will be plugged into the Maxwell reflection boundary condition \eqref{MBC} and then utilize the solvability conditions \eqref{Solva_Cond} of the Knudsen boundary layer problem \eqref{KBL-fbb_k2}-\eqref{BC-KBL-fbb_k2}. More precisely, on $\Sigma_-$,
\begin{equation}\label{MBC-Expand-1}
  \begin{aligned}
    \sqrt{\eps}^0 : & \quad L^R \M = 0 \,,  \\
    \sqrt{\eps}^1 : & \quad L^R (F_1 + F_1^b + F_1^{bb}) = L^D \M \,, \\
    \cdots & \cdots \\
    \sqrt{\eps}^k : & \quad L^R (F_k + F^b_k + F^{bb}_k ) = L^D (F_{k-1} + F^b_{k-1} + F^{bb}_{k-1}) \ (k \geq 2) \,,
  \end{aligned}
\end{equation}
where the operators $L^R$ and $L^D$ are defined as
\begin{equation}
  \begin{aligned}
    L^R F = (\gamma_- - L \gamma_+ ) F \,, \quad L^D F = \sqrt{2 \pi} (K \gamma_+ - L \gamma_+ ) F \,.
  \end{aligned}
\end{equation}

Actually, $\sqrt{\eps}^0$-order of \eqref{MBC-Expand-1} can imply the slip boundary condition \eqref{BC-CEuler} of the compressible Euler system \eqref{Compressible_Euler_Sys}. Recalling the definitions
\begin{equation*}
  \begin{aligned}
    f_k = \tfrac{F_k}{\sqrt{\M}} \,, \ f^b_k = \tfrac{F^b_k}{\sqrt{\M^0}} \,, \ f^{bb}_k = \tfrac{F^{bb}_k}{\sqrt{\M^0}} \,, \ f^{bb}_{k,2} = f^{bb}_k - f^{bb}_{k,1}
  \end{aligned}
\end{equation*}
for $k \geq 1$, where the functions $f^{bb}_{k,1}$ ($k \geq 1$) are given in Lemma \ref{Lmm-fbb-k1}, we thereby obtain by direct calculation and \eqref{MBC-Expand-1} that the functions $\mathbbm{f}_k (t, \bar{x}, \bar{v}, v_3)$ ($k \geq 1$) in \eqref{BC-KBL-fbb_k2} are
\begin{equation}\label{fk-KBL}
  \begin{aligned}
    \mathbbm{f}_k (t, \bar{x}, \bar{v}, v_3) = \left\{
    \begin{array}{l}
    0\,, \qquad \textrm{if } v_3 >0\,, \\ [1.5mm]
    (f_k + f^b_k + f^{bb}_{k,1})(t, \bar{x}, 0, \bar{v}, v_3) - (f_k + f^b_k + f^{bb}_{k,1}) (t, \bar{x}, 0, \bar{v}, -v_3) \\
    \qquad + \sqrt{2\pi} \big\{ \big[ \l \gamma_+ (f_{k-1} + f^b_{k-1} + f^{bb}_{k-1}) \r_{\p \R^3_+} \sqrt{\M^0} \big] \\
    \qquad \qquad \qquad \quad - (f_{k-1} + f^b_{k-1} + f^{bb}_{k-1}) \big\} (t, \bar{x}, 0, \bar{v}, v_3)\,,\  \textrm{if } v_3 <0\,,
    \end{array}
    \right.
  \end{aligned}
\end{equation}
where the symbol $\l \gamma_+ f \r_{\p \R^3_+}$ means
\begin{equation*}
  \begin{aligned}
    \l \gamma_+ f \r_{\p \R^3_+}
    = \sqrt{2\pi} \tfrac{M_w (v)}{\M^0} \int_{v \cdot n(x)>0} v \cdot n(x) (\gamma_+ f)\sqrt{\M^0} \d v\,,
  \end{aligned}
\end{equation*}
and  the notations $f_0 = \sqrt{\M}$, $f^b_0 = f^{bb}_0 = 0$ are used, for simplicity of presentation.

Therefore, the following lemma holds:
\begin{lemma}\label{Lmm-Robin-BC}
	Let the local Maxwellian of the boundary $M_w = \M^0$ in \eqref{MBC} and $\mathbbm{f}_k (t, \bar{x}, \bar{v}, v_3)$ be given in \eqref{fk-KBL}. Then the solvability conditions \eqref{Solva_Cond} of the Knudsen boundary layer problem \eqref{KBL-fbb_k2}-\eqref{BC-KBL-fbb_k2} imply that for $k \geq 1$, the linear hyperbolic system \eqref{Linear-Hyperbolic-Syst} has the following slip boundary condition
	\begin{equation}\label{BC-u_k3}
	  \begin{aligned}
	    & u_{k,3} (t, \bar{x}, 0) = - u^b_{k,3} (t, \bar{x}, 0) - T^0 (\Psi_k + 5T^0 \Theta_k) (t, \bar{x}, 0) \\
	    & \ \ \ + \tfrac{(\rho^0 \sqrt{T^0} + 1)}{\rho^0} \sqrt{2 \pi} \int_{\R^2} \int_{-\infty}^0 v_3 (f_{k-1} + f^b_{k-1} + f^{bb}_{k-1}) (t, \bar{x}, 0, \bar{v}, v_3) \sqrt{\M^0} \d \bar{v} \d v_3 \\
	    & = - \int_0^{+\infty} \tfrac{1}{\rho^0} \big[ \p_t \rho^b_{k-1} + \div_{\bar{x}} (\rho^0 \bar{u}^b_{k-1} + \rho^b_{k-1} \bar{\u}^0) \big] (t, \bar{x}, \zeta) \d \zeta - T^0 (\Psi_k + 5T^0 \Theta_k) (t, \bar{x}, 0) \\
	    & \ \ \ + \tfrac{(\rho^0 \sqrt{T^0} + 1)}{\rho^0} \sqrt{2 \pi} \int_{\R^2} \int_{-\infty}^0 v_3 (f_{k-1} + f^b_{k-1} + f^{bb}_{k-1}) (t, \bar{x}, 0, \bar{v}, v_3) \sqrt{\M^0} \d \bar{v} \d v_3\,,
	  \end{aligned}
	\end{equation}
	and for $k \geq 2$, the linear Prandtl-type equations \eqref{Linear-Prandtl} are of the Robin-type boundary conditions
	\begin{equation}\label{Boudary_Equa}
	\left\{
	\begin{array}{l}
	\big( \p_\zeta u^b_{k-1, i} -  \tfrac{ \rho^0 \sqrt{T^0} (2 + \rho^0 \sqrt{T^0})}{\mu(T^0)} u^b_{k-1,i} \big) (t, \bar{x}, 0) = \Lambda^b_{k-1, i} (t, \bar{x}), i=1,2\,, \\[1.5mm]
	\big( \p_\zeta \theta^b_{k-1} - \tfrac{\rho^0 \sqrt{T^0}}{\kappa (T^0)}  ( 2 \rho^0 \sqrt{T^0} + \tfrac{\sqrt{2 \pi}}{3} \rho^0 + \tfrac{2}{3} ) \theta^b_{k-1} \big) (t, \bar{x}, 0) = \Lambda^b_{k-1,\theta} (t, \bar{x})\,,
	\end{array}
	\right.
	\end{equation}
	where $\Psi_k$ and $\Theta_k$ are given in \eqref{f_bb_k1_Coef}, and
	\begin{equation*}
	\begin{aligned}
	& \Lambda_{k-1, i}^b (t, \bar{x}) = \tfrac{\rho^0 \sqrt{T^0}}{\mu(T^0)} u_{k-1, i} (t, \bar{x}, 0) \\
	& + \tfrac{1}{\mu(T^0)} \Big\{\rho^0 (T^0)^2 \Phi_{k,i} + \rho^0 \big[ (u_{1,i} + u^b_{1,i}) u^b_{k-1,3} \big] + T^0 \l \A^0_{3i}, J^b_{k-2} + (\mathcal{I} - \mathcal{P}^0) f_k \r \Big\} (t, \bar{x}, 0) \\
	& - \tfrac{\sqrt{2\pi} }{\mu(T^0)} \int_{\R^2} \int_{-\infty}^0 (v_i - \u^0_i ) v_3 \big[ (\mathcal{I} - \mathcal{P}^0) (f_{k-1} + f^b_{k-1}) + f^{bb}_{k-1} \big] (t, \bar{x}, 0, \bar{v}, v_3) \sqrt{\M^0} \d \bar{v} \d v_3\,,
	\end{aligned}
	\end{equation*}
	and
	\begin{equation*}
	\begin{aligned}
	\Lambda_{k-1, \theta}^b & (t, \bar{x})
	= \tfrac{2 (T^0)^{\frac{3}{2}}}{\kappa(T^0)} \l \B^0_3, (\mathcal{I}-\mathcal{P}^0) f_k + J^b_{k-2} \r (t, \bar{x}, 0) + \tfrac{5}{3\kappa(T^0)} \rho^0 \big[ (\theta_1 + \theta^b_1) u^b_{k-1,3} \big] (t, \bar{x}, 0)\\
	& + \tfrac{\rho^0 T^0}{\kappa (T^0)} \big\{ 10 (T^0)^2 \Theta_k - ( \tfrac{1}{2} + \tfrac{3}{2} T^0 - 2 \rho^0 (T^0)^\frac{3}{2} ) (\Psi_{k-1} + 5T^0 \Theta_{k-1}) \big\} (t, \bar{x}, 0) \\
	& + \tfrac{\rho^0 \sqrt{T^0}}{\kappa(T^0)} \big\{ (\tfrac{2}{3} \rho^0 \sqrt{T^0} + \tfrac{\sqrt{2 \pi}}{2} \rho^0 + 2) \theta_{k-1} - 4 (\sqrt{T^0} - \tfrac{1}{\rho^0}) (T^0 \rho_{k-1} + p^b_{k-1}) \big\} (t, \bar{x}, 0) \\
	& - \tfrac{\sqrt{2\pi}}{\kappa(T^0)} \int_{\R^2} \int_{-\infty}^0 v_3 (|v - \u^0|^2 - 4 \rho^0 (T^0)^\frac{3}{2} ) \\
	& \qquad \qquad \times \big[ (\mathcal{I}-\mathcal{P}^0) (f_{k-1} + f^b_{k-1}) + f^{bb}_{k-1} \big](t, \bar{x}, 0, \bar{v}, v_3)  \sqrt{\M^0} \d \bar{v} \d v_3 \\
	& + \tfrac{\sqrt{2\pi}}{\kappa(T^0)} (\rho^0 \sqrt{T^0} + 1) ( \tfrac{1}{2} + \tfrac{3}{2} T^0 - 2 \rho^0 (T^0)^\frac{3}{2} ) \\
	& \qquad \qquad \times \int_{\R^2} \int_{-\infty}^0 v_3 (f_{k-2} + f^b_{k-2} + f^{bb}_{k-2} )(t, \bar{x}, 0, \bar{v}, v_3)  \sqrt{\M^0} \d \bar{v} \d v_3 \,.
	\end{aligned}
	\end{equation*}
	In particular, if $k = 1$ in \eqref{BC-u_k3}, we have
	\begin{equation}\label{BC-u_13}
	  \begin{aligned}
	    u^0_{1,3} = u_{1, 3} (t, \bar{x}, 0) = \sqrt{T^0} (\rho^0 \sqrt{T^0} + 1) \,.
	  \end{aligned}
	\end{equation}
\end{lemma}
The proof of the lemma will be given in Section \ref{Sec_Robin-BC}. We remark that the Robin-type boundary condition \eqref{Boudary_Equa} will be Neumann-type boundary values and $u_{1,3}^0$ will vanish, if the Maxwell reflection boundary condition \eqref{MBC} is replaced by the specular reflection boundary condition, i.e., by letting $\alpha_\eps = 0$, see \cite{GHW-2020}.

\subsubsection{Truncations of the Hilbert expansion}\label{subsubsec115}

Our goal is to prove the compressible Euler limit from the scaled Boltzmann equation by above Hilbert type expansion. The key point is to prove that the remainders of expansion will go to zero as the Knudsen number $\eps \to 0$. Mathematically, this means we search for a special class of solutions of the original scaled Boltzmann equation for sufficiently small Knudsen number $\eps$. Since the more terms are expanded, the more special the solutions are. We hope that the terms in the expansion are as less as possible.

Our truncated Hilbert expansion takes the following form:
\begin{equation}\label{Hilbert-Expnd-Form}
  \begin{aligned}
     F_\eps (t, x , v) = \M (t,x,v) + \sum_{k=1}^5 \sqrt{\eps}^k \big\{ F_k (t,x,v) + F_k^b (t, \bar{x}, \tfrac{x_3}{\sqrt{\eps}}, v) + F_k^{bb} (t, \bar{x}, \tfrac{x_3}{\eps}, v) \big\} \\
     + \sqrt{\eps}^4 F_{R, \eps} (t, x, v) \geq 0 \,.
  \end{aligned}
\end{equation}
Here $\M$ is governed by the compressible Euler system \eqref{Compressible_Euler_Bound}-\eqref{BC-CEuler}-\eqref{IC-CEuler}.

First, the interior expansions $F_1 = \P (\tfrac{F_1}{\sqrt{\M}}) \sqrt{\M}$ and $F_k = \big\{ \P (\tfrac{F_k}{\sqrt{\M}}) + (\I - \P) (\tfrac{F_k}{\sqrt{\M}}) \big\} \sqrt{\M}$ ($k = 2, 3$), where $(\I - \P) (\tfrac{F_k}{\sqrt{\M}})$ are given by \eqref{Ker-Arth-Inte}, and the fluid variables $(\rho_k, u_k, \theta_k)$ $(1 \leq k \leq 3)$ associated with $\P (\tfrac{F_k}{\sqrt{\M}})$ are solutions of the linear hyperbolic system \eqref{Linear-Hyperbolic-Syst} with initial data \eqref{IC-Linear-Hyperbolic} and slip boundary condition \eqref{BC-u_k3}. In this sense, $F_k$ ($1 \leq k \leq 3$) are completely specified. For $k=4,5$, we take $F_k$ such that their kernel (fluid) parts vanish, i.e., $F_k = [ (\I - \P) \tfrac{F_k}{\sqrt{\M}} ] \sqrt{\M} $, which are defined by \eqref{Ker-Arth-Inte}. They thereby are not completely known terms.

Second, the viscous boundary layer expansions $F_1^b = (\P^0 f_1^b ) \sqrt{\M^0}$ and $F_k^b = \big\{ \P^0 f_k^b + (\I - \P^0) f_k^b \big\} \sqrt{\M^0}$ ($k = 2, 3$), where the kinetic part $(\I - \P^0) f_k^b$ ($k = 2, 3$) are given by \eqref{f_b_k+1-Kernal-Ortho}, and $\P^0 f_k^b$ ($1 \leq k \leq 3$) correspond to the fluid variables $(\rho^b_k, u^b_k, \theta^b_k)$. $(\bar{u}^b_k, \theta^b_k)$ are solutions of the linear Prandtl-type system \eqref{Linear-Prandtl} with initial data \eqref{IC-Prandtl} and Robin-type boundary conditions \eqref{Boudary_Equa}. The functions $(\rho^b_k, u^b_{k,3}) $ $(1 \leq k \leq 3)$ are expressed through equations \eqref{pb_k}, \eqref{ub_13-pb_1}, \eqref{u_b_k+1-derivative} and \eqref{p_b_k+1-derivative}, where the subscript $k+1$ is replaced by $k$. Thus $F_k$ ($1 \leq k \leq 3$) are all known terms, so are $F^b_k$ ($1 \leq k \leq 3$). For $k = 4,5$, we choose $F^b_k$ such that their fluid variables vanish, hence, $F^b_k = [ (\I - \P^0) f^b_k ] \sqrt{\M^0}$, which are expressed by \eqref{f_b_k+1-Kernal-Ortho}, which mean that they are not completely specified.

Finally, we observe from \eqref{fbb_k}-\eqref{Sbb_k} and the boundary values \eqref{BC-KBL-fbb_k2}-\eqref{fk-KBL} that the Knudsen boundary layer expansions $F^{bb}_k$ depends on $F^{bb}_i$ ($i \leq k-1$) and $F^b_j$, $F_j$ ($j \leq k$). Here $F^b_i = F^{bb}_i = F_{i-1} = 0$ for subscript $i \leq 0$ and $F_0 = \M$. Since $\M$, $F_k$ and $F^b_k$ ($k=1,2,3$) are already known, we easily see that $F^{bb}_k$ ($k=1,2,3$) are totally solved. Similarly, $F^{bb}_4$ and $F^{bb}_5$ are not completely solved. Consequently, the corresponding remainder shall be $\sqrt{\eps}^4 F_{R, \eps}$ as in \eqref{Hilbert-Expnd-Form}. We also emphasize that the number $N \in \mathbb{N}_+$ appeared in the Taylor expansions before will be chosen by $N = 4$ for the Hilbert expansion \eqref{Hilbert-Expnd-Form}.

Consequently, from plugging \eqref{Hilbert-Expnd-Form} into \eqref{BE}-\eqref{MBC}, we obtain the remainder equation
\begin{equation}\label{Remainder-F}
  \begin{aligned}
    \p_t F_{R, \eps} + & v \cdot \nabla_x F_{R, \eps} - \tfrac{1}{\eps} [ B(\M, F_{R, \eps}) + B (F_{R, \eps}, \M) ] \\
    = & \sqrt{\eps}^2 B(F_{R, \eps}, F_{R, \eps}) + R_\eps + R^b_\eps + R^{bb}_\eps \\
    & + \sum_{i=1}^5 \sqrt{\eps}^{i-2} [ B(F_i + F^b_i + F^{bb}_i , F_{R,\eps}) + B( F_{R,\eps}, F_i + F^b_i + F^{bb}_i ) ]
  \end{aligned}
\end{equation}
with Maxwell reflection type boundary condition
\begin{equation}\label{BC-Remainder-F}
  \begin{aligned}
    \gamma_- F_{R, \eps} = (1 - \alpha_\eps) L \gamma_+ F_{R, \eps} + \alpha_\eps K \gamma_+ F_{R, \eps} + \sqrt{\eps}^2 \Gamma_\eps \quad \textrm{ on } \Sigma_- \,,
  \end{aligned}
\end{equation}
where
\begin{equation}\label{Gamma_eps}
  \begin{aligned}
    \Gamma_\eps = \tfrac{\alpha_\eps}{\sqrt{\eps}} ( K \gamma_+ - L \gamma_+ ) (F_5 + F^b_5 + F^{bb}_5) \,,
  \end{aligned}
\end{equation}
and
\begin{equation}\label{R_eps}
  \begin{aligned}
    R_\eps = - (\p_t + v \cdot \nabla_x) (F_4 + \sqrt{\eps} F_5) + \sum_{\substack{ i+j \geq 6 \\ 1 \leq i, j \leq 5 }} \sqrt{\eps}^{i+j-6} B(F_i, F_j) \,,
  \end{aligned}
\end{equation}
and
\begin{equation}\label{R_eps_b}
  \begin{aligned}
    R^b_\eps = & - (\p_t + \bar{v} \cdot \nabla_{\bar{x}}) (F^b_4 + \sqrt{\eps} F^b_5) - v_3 \p_\zeta F^b_5 \\
    & + \sum_{\substack{ j+l \geq 6 \\ 1 \leq j \leq 5, 1 \leq l \leq 4 }} \sqrt{\eps}^{j+l-6} \tfrac{\zeta^l}{l !} [B (\M^{(l)}, F^b_j) + B(F^b_j, \M^{(l)})] \\
    & + \sum_{\substack{ i+j \geq 6 \\ 1 \leq i, j \leq 5 }} \sqrt{\eps}^{i+j-6} [ B(F^0_i, F^b_j) + B(F^b_j, F^0_i) + B(F^b_i, F^b_j) ] \\
    & + \sum_{\substack{ i+j + l \geq 6 \\ 1 \leq i, j \leq 5, 1 \leq l \leq 4 }} \sqrt{\eps}^{i+j+l-6} \tfrac{\zeta^l}{l !} [B(F_i^{(l)}, F^b_j) + B(F^b_j, F_i^{(l)})] \\
    & + \tfrac{\zeta^5}{5!} \sum_{j=1}^5 \sqrt{\eps}^{j-1} [ B (\widetilde{\M}^{(5)} + \sum_{i=1}^5 \sqrt{\eps}^i \widetilde{F}^{(5)}_i , F^b_j) + B (F^b_j, \widetilde{\M}^{(5)} + \sum_{i=1}^5 \sqrt{\eps}^i \widetilde{F}^{(5)}_i) ] \,,
  \end{aligned}
\end{equation}
and
  \begin{align}\label{R_eps_bb}
    \no R^{bb}_\eps = & - (\p_t + \bar{v} \cdot \nabla_{\bar{x}}) (F^{bb}_4 + \sqrt{\eps} F^{bb}_5) \\
    \no & + \sum_{\substack{ j+2l \geq 6 \\ 1 \leq j \leq 5, 1 \leq l \leq 4 }} \sqrt{\eps}^{j+2l-6} \tfrac{\xi^l}{l!} [ B(\M^{(l)}, F^{bb}_j) + B(F^{bb}_j, \M^{(l)}) ] \\
    \no & + \sum_{\substack{ i+j \geq 6 \\ 1 \leq i, j \leq 5 }} \sqrt{\eps}^{i+j-6} [B(F^0_i + F^{b,0}_i, F^{bb}_j) + B(F^{bb}_j, F^0_i + F^{b,0}_i) + B( F^{bb}_i, F^{bb}_j ) ] \\
    \no & + \sum_{\substack{ i+j+2l \geq 6 \\ 1 \leq i, j \leq 5, 1 \leq l \leq 4 }} \sqrt{\eps}^{i+j+2l-6} \tfrac{\xi^l}{l!} [ B(\widetilde{F}^{(l)}_i, F^{bb}_j) + B (F^{bb}_j, \widetilde{F}^{(l)}_i) ] \\
    \no & + \sum_{\substack{ i+j+l \geq 6 \\ 1 \leq i, j \leq 5, 1 \leq l \leq 4 }} \sqrt{\eps}^{i+j+l-6} \tfrac{\xi^l}{l!} [ B(\widetilde{F}^{b, (l)}_i, F^{bb}_j) + B (F^{bb}_j, \widetilde{F}^{b, (l)}_i) ] \\
    \no & + \tfrac{\xi^5}{5!} \sum_{j=1}^5 \sqrt{\eps}^{j-1} [ B  (\sqrt{\eps}^5 \widetilde{\M}^{(5)} + \sum_{i=1}^5 \sqrt{\eps}^{i+5} \widetilde{F}^{(5)}_i + \sqrt{\eps}^i \widetilde{F}^{b, (5)}_i , F^{bb}_j) \\
    & \qquad \qquad \qquad \qquad + B (F^{bb}_j, \sqrt{\eps}^5 \widetilde{\M}^{(5)} + \sum_{i=1}^5 \sqrt{\eps}^{i+5} \widetilde{F}^{(5)}_i + \sqrt{\eps}^i \widetilde{F}^{b, (5)}_i) ] \,.
  \end{align}
  Furthermore, the following initial data are imposed on the remainder equation \eqref{Remainder-F}:
  \begin{equation}\label{IC-Remainder-F}
    \begin{aligned}
      F_{R, \eps} (0, x, v) = F_{R, \eps}^{in} (x,v) \,,
    \end{aligned}
  \end{equation}
  which satisfies the compatibility condition on $\Sigma_-$
  \begin{equation*}
    \begin{aligned}
      \gamma_- F_{R, \eps}^{in} = (1 - \alpha_\eps) L \gamma_+ F_{R, \eps}^{in} + \alpha_\eps K \gamma_+ F_{R, \eps}^{in} + \sqrt{\eps}^2 \Gamma_\eps |_{t = 0} \,.
    \end{aligned}
  \end{equation*}

For the remainder $F_{R, \eps}$, let
\begin{equation}\label{f_Reps-h_Reps}
  \begin{aligned}
    f_{R, \eps} = \tfrac{F_{R, \eps}}{\sqrt{\M}} \,, \quad h_{R, \eps}^\ell = \l v \r^\ell \tfrac{F_{R,\eps}}{\sqrt{\M_M}}
  \end{aligned}
\end{equation}
for $\ell \geq 9 - 2 \gamma_0$, where the global Maxwellian $\M_M = \M_M (v)$ is introduced by \cite{Caflish-1980-CPAM}
\begin{align}\label{Global_Maxwellian_M}
\M_M = \tfrac{1}{(2\pi T_M)^{\frac{3}{2}}} \exp \Big\{ - \tfrac{|v|^2}{2 T_M} \Big\} \,.
\end{align}
Here the constant $T_M$ satisfies
\begin{align} \label{T_M}
T_M < \max_{t \in [0, \tau], x \in \R^3_+} T (t, x) < 2 T_M \,.
\end{align}
Then there exists constants $C_1$, $C_2$ such that for some $\frac{1}{2} < z <1$ and for each $(t, x, v) \in [0, \tau] \times \R^3_+ \times \R^3$, the following inequality holds:
\begin{align}\label{M-Bound}
C_1 \M_M \le \M \le C_2 ( \M_M )^z\,.
\end{align}

\subsection{Main results}

We first introduce some notations for convenience for stating our main theorem. For multi-indexes $\alpha = (\alpha_1, \alpha_2, \cdots , \alpha_m), \alpha' = (\alpha_1' , \alpha_2', \cdots , \alpha_m') \in \mathbb{N}^m$, the symbol $\alpha \leq \alpha'$ means $\alpha_i \leq \alpha_i' (i = 1, 2, \cdots, m)$ and $|\alpha| = \alpha_1 + \alpha_2 + \cdots + \alpha_m$.

In order to quantitatively describe the linear hyperbolic system \eqref{Linear-Hyperbolic-Syst} in the half-space $\R^3_+$, let
\begin{equation*}
\p^\alpha_{t, \bar{x}} = \p_t^{\alpha_0} \p_{x_1}^{\alpha_1} \p_{x_2}^{\alpha_2} \,,
\end{equation*}
where $\alpha = (\alpha_0, \alpha_1, \alpha_2) \in \mathbb{N}^3$, and let
\begin{equation}\label{Notation_H_k}
\begin{aligned}
\|f(t)\|^2_{\mathcal{H}^k (\R^3_+)} = \sum_{|\alpha|+i \le k} \|\p_{t, \bar{x}}^\alpha \p^i_{x_3} f(t)\|^2_{L^2(\R^3_+)}\,, \ \|g(t)\|^2_{\mathcal{H}^k(\R^2)} = \sum_{|\alpha|\le k} \|\p_{t, \bar{x}}^\alpha g(t)\|^2_{L^2(\R^2)}
\end{aligned}
\end{equation}
for functions $f (t) = f(t, \bar{x}, x_3)$ and $g(t) = g (t, \bar{x})$. We remark that, for a function $f = f(\bar{x}, x_3)$ independent of the variable $t$, $\| f \|_{\mathcal{H}^k (\R^3_+)}$ is equivalent to the standard Sobolev norm $\| f \|_{H^k (\R^3_+)}$.

While characterizing quantitatively the linear Prandtl-type system \eqref{Linear-Prandtl} associated with the macroscopic parts of the viscous boundary layer, some new norms are required to be introduced. For $l \geq 0$,  the weighted norm is defined as
\begin{equation}\label{L2l}
\begin{aligned}
\| f \|^2_{L^2_l} = \int_{\R^2} \int_{\R_+} (1 + \zeta)^l | f (\bar{x}, \zeta) |^2 \d \bar{x} \d \zeta \,.
\end{aligned}
\end{equation}
We further introduce a weighted Sobolev space $\mathbb{H}_l^r (\R^3_+)$ for any $r, l \geq 0$. Denote by the 2D multi-index $\beta = (\beta_1, \beta_2) \in \mathbb{N}^2$. For any $l, r \geq 0$, let
\begin{equation}\label{Def-l_j}
\begin{aligned}
l_j = l + 2 (r - j) \,, \ 0 \leq j \leq r \,.
\end{aligned}
\end{equation}
We then introduce the norms
\begin{equation}\label{Hrl-t-3D}
\begin{aligned}
& \| f (t) \|^2_{l,r,n} = \sum_{2\gamma + |\beta| =r -n} \| \partial_t^\gamma \partial_{\bar{x}}^\beta \p_{\zeta}^n f (t) \|^2_{L^2_{l_r}} \ (0 \leq n \leq r) \,, \\
& \| f (t) \|^2_{l,r} = \sum_{n=0}^r \| f (t) \|^2_{l,r,n} = \sum_{2\gamma + |\beta| + n =r} \| \partial_t^\gamma \partial_{\bar{x}}^\beta \partial_\zeta^n f (t) \|^2_{L^2_{l_r}} \,, \\
& \| f (t) \|^2_{\mathbb{H}^r_{l, n} (\R^3_+)} = \sum_{j=0}^r \| f (t) \|^2_{l, j, n} \ (n = 0, 1, \cdots, r) \,, \\
& \| f (t) \|^2_{\mathbb{H}^r_l (\R^3_+)} = \sum_{j=0}^r \| f (t) \|^2_{l, j} = \sum_{n=0}^r \| f (t) \|^2_{\mathbb{H}^r_{l, n} (\R^3_+)}
\end{aligned}
\end{equation}
for function $f = f (t, \bar{x}, \zeta)$. For $g = g(\bar{x}, \zeta)$, let
\begin{equation}\label{Hrl-3D}
\begin{aligned}
\| g \|^2_{\mathbb{H}^r_l (\R^3_+)} = \sum_{j=0}^r \sum_{|\beta| + n = j} \| \p_{\bar{x}}^\beta \p_{\zeta}^n g \|^2_{L^2_{l_j}} \,.
\end{aligned}
\end{equation}
Similarly, for function $h = h (t, \bar{x})$, let
\begin{equation}\label{Hr-2D}
\begin{aligned}
\| h (t) \|^2_{\mathbb{H}^r (\R^2)} = \sum_{j=0}^r \| h (t) \|^2_{\Gamma, j} = \sum_{j=0}^r \sum_{2 \gamma + |\beta| = j} \| \p_t^\gamma \p_{\bar{x}}^\beta h (t) \|^2_{L^2 (\R^2)} \,.
\end{aligned}
\end{equation}
In the above norms, one order time derivative is equivalent to two orders space derivative.

We now clarify the initial data of the scaled Boltzmann equation \eqref{BE}. Let
\begin{equation}
  \begin{aligned}
    \M^{in} (x, v) = \tfrac{\rho^{in} (x)}{[2 \pi T^{in} (x)]^\frac{3}{2}} \exp \Big\{ - \tfrac{|v - \u^{in} (x)|^2}{2 T^{in} (x)} \Big\} \,.
  \end{aligned}
\end{equation}
For $1 \leq k \leq 5$,  $F_k^{in} (x,v)$, $F_k^{b, in} (\bar{x}, \tfrac{x_3}{\sqrt{\eps}}, v)$ and $F^{bb, in}_k (\bar{x}, \tfrac{x_3}{\eps}, v)$ can be constructed by the same ways of constructing the expansions $F_k (t, x, v)$, $F^b_k (t, \bar{x}, \tfrac{x_3}{\sqrt{\eps}}, v)$ and $F^{bb}_k (t, \bar{x}, \tfrac{x_3}{\eps}, v)$ in Subsection \ref{subsubsec115}, respectively. More precisely, for $k = 1,2,3$, it just replaces $\M$, $(\rho_k, u_k, \theta_k)$ and $(\bar{u}^b_k, \theta^b_k)$ by $\M^{in}$, $(\rho_k^{in}, u_k^{in}, \theta_k^{in})$ and $(\bar{u}^{b,in}_k, \theta^{b, in}_k)$, respectively. Here we further assume the initial data $(\rho_k^{in}, u_k^{in}, \theta_k^{in})$ and $(\bar{u}^{b,in}_k, \theta^{b, in}_k)$ ($1 \leq k \leq 3$) compatibly satisfy the conditions \eqref{BC-u_k3} and \eqref{Boudary_Equa}, respectively. We impose the well-prepared initial data on the scaled Boltzmann equation \eqref{BE}
\begin{equation}\label{IC-BE-wellprepared}
  \begin{aligned}
    F_\eps (0, x, v) = \M^{in} (x,v) + \sum_{k=1}^5 \sqrt{\eps}^k \big\{ F_k^{in} (x,v) + F_k^{b, in} (\bar{x}, \tfrac{x_3}{\sqrt{\eps}}, v) + F^{bb, in}_k (\bar{x}, \tfrac{x_3}{\eps}, v) \big\} \\
    + \sqrt{\eps}^4 F_{R, \eps}^{in} (x,v) \geq 0 \,.
  \end{aligned}
\end{equation}

Furthermore, for the local Maxwellian distribution $M_w (t,x)$ of the boundary given in \eqref{M_w}, we take the special form
\begin{align}\label{M_w-Perturbation}
M_w (t, \bar{x}, v) = \M^0 (t, \bar{x}, v) \,.
\end{align}

We now state our main theorem.

\begin{theorem}\label{Main-Thm}
	Consider the hard potential interaction $(0 \leq \gamma_0 \leq 1)$ Boltzmann collision kernel $B$ with an angular cutoff $($see \eqref{B-collision}$)$. Let $\ell \geq 9 - 2 \gamma_0$, and integers $s_0$, $s_k$, $s_k^b$, $s_k^{bb}$, $l^b_k$ $(1 \leq k \leq 3)$ be described as in Proposition \ref{Pro_Regularities-Coefficients}. Assume that $\| (\rho^{in}, \u^{in}, T^{in}) \|_{H^{s_0} (\R^3_+)} < \infty$ and
	\begin{equation}\label{Ein}
	  \begin{aligned}
	    \mathcal{E}^{in} : = \sum_{k=1}^3 \Big\{ \| (\rho_k^{in}, u_k^{in}, \theta_k^{in}) \|_{\mathcal{H}^{s_k} (\R^3_+)} + \| (\bar{u}^{b, in}_k, \theta_k^{b, in}) \|_{\mathbb{H}^{s_k^b}_{l^b_k} (\R^3_+)} \Big\} < \infty \,.
	  \end{aligned}
	\end{equation}
	Let $(\rho, \u, T)$ be the solution to the compressible Euler equations \eqref{Compressible_Euler_Sys} over the time interval $t \in [0, \tau]$ constructed in Proposition \ref{Proposition_Compressible_Euler}, which determines the local Maxwellian $\M$ defined in \eqref{Maxwellian}. For $1 \leq k \leq 5$, let $F_k (t,x,v)$, $F_k^b (t, \bar{x}, \tfrac{x_3}{\sqrt{\eps}}, v)$ and $F_k^{bb} (t, \bar{x}, \tfrac{x_3}{\eps}, v)$ be constructed in Proposition \ref{Pro_Regularities-Coefficients}. The local Maxwellian $M_w (t, \bar{x}, v)$ of the boundary is assumed as in \eqref{M_w-Perturbation}. There is a small constant $\eps_0 > 0$ such that if for $\eps \in (0, \eps_0)$
	\begin{equation*}
	  \begin{aligned}
	    \mathcal{E}_R^{in} : = \sup_{\eps \in (0, \eps_0)} \{ \| \tfrac{F_{R,\eps}^{in}}{\sqrt{\M^{in}}} \|_2 + \sqrt{\eps}^3 \| \l v \r^\ell \tfrac{F_{R,\eps}^{in}}{\sqrt{\M_M}} \|_\infty \} < \infty \,,
	  \end{aligned}
	\end{equation*}
	then the scaled Boltzmann equation \eqref{BE} with Maxwell reflection boundary condition \eqref{MBC} and well-prepared initial data \eqref{IC-BE-wellprepared} admits a unique solution for $\eps \in (0, \eps_0)$ over the time interval $t \in [0, \tau]$ with the expanded form \eqref{Hilbert-Expnd-Form}, i.e.,
	\begin{equation*}
	  \begin{aligned}
	    F_\eps (t, x , v) = \M (t,x,v) + \sum_{k=1}^5 \sqrt{\eps}^k \big\{ F_k (t,x,v) + F_k^b (t, \bar{x}, \tfrac{x_3}{\sqrt{\eps}}, v) + F_k^{bb} (t, \bar{x}, \tfrac{x_3}{\eps}, v) \big\} \\
	    + \sqrt{\eps}^4 F_{R, \eps} (t, x, v) \geq 0 \,,
	  \end{aligned}
	\end{equation*}
	where the remainder $F_{R, \eps} (t, x, v)$ satisfies
	\begin{equation*}
	  \begin{aligned}
	    \sup_{t \in [0, \tau]} \Big\{ \| \tfrac{F_{R, \eps} (t)}{\sqrt{\M}} \|_2 + \sqrt{\eps}^3 & \| \l v \r^\ell \tfrac{F_{R,\eps} (t)}{\sqrt{\M_M}} \|_\infty \Big\} \\
	    & \leq C(\tau, \| (\rho^{in} - \rho_\#, \u^{in}, T^{in} - T_\#) \|_{H^{s_0} (\R^3_+)}, \mathcal{E}^{in}, \mathcal{E}_R^{in}) < \infty \,.
	  \end{aligned}
	\end{equation*}
\end{theorem}

\begin{remark}
	Together with Proposition \ref{Pro_Regularities-Coefficients}, Theorem \ref{Main-Thm} shows that as $\eps \to 0$,
	\begin{equation*}
	  \begin{aligned}
	    \sup_{t \in [0, \tau]} \Big\{ \| (\tfrac{F_\eps - \M}{\sqrt{\M}}) (t) \|_2 + \| \l v \r^\ell ( \tfrac{F_\eps - \M}{\sqrt{\M_M}} ) (t) \|_\infty \Big\} \leq C \sqrt{\eps} \to 0 \,.
	  \end{aligned}
	\end{equation*}
	Therefore, while imposed on the well-prepared initial data and accommodation coefficients $\alpha_\eps = \mathcal{O} (\sqrt{\eps})$ as $\eps \to 0$, we have justified the hydrodynamic limit from the Boltzmann equation with Maxwell reflection boundary condition to the compressible Euler system with slip boundary condition for the half-space problem.
\end{remark}

\subsection{Sketch of ideas and novelties}\label{Sec_1.5}

Contrast with the work \cite{GHW-2020}, where the specular reflection boundary condition (in short, SRBC) for the scaled Boltzmann equation \eqref{BE} was considered, we consider the Maxwell reflection boundary condition (briefly, MRBC). We first point out the formal differences between \cite{GHW-2020} and our work. For the MRBC case, the slip boundary value $u_{1, 3} (t, \bar{x}, 0) = \sqrt{T^0} (\rho^0 \sqrt{T^0} + 1) > 0$ (see \eqref{BC-u_13}) for the linear hyperbolic system $(\rho_1, u_1, \theta_1)$ (see \eqref{Linaer_Hyper_Sys}) does not vanish, and for the SRBC case, $u_{1,3} (t, \bar{x}, 0) \equiv 0$. The types of boundary conditions for the linear compressible Prandtl-type equations \eqref{Linear-Prandtl} are also different. For the MRBC case, we obtain Robin-type boundary conditions (see in \eqref{Boudary_Equa}), while the Neumann-type  was derived in the SRBC case. We emphasize that the structures of boundary energy will be more subtle in the Robin-type than that in the Neumann-type while proving the existence of the system \eqref{Linear-Prandtl}. Furthermore, the way to deal with the boundary integral $- \tfrac{1}{2} \iint_{\p \R^3_+ \times \R^3} v_3 |f_{R, \eps} (t, \bar{x}, 0, v)|^2 \d v \d \bar{x}$ in $L^2$ estimates on the remainder $f_{R,\eps}$ is even much more different. For the SRBC case, the boundary integral is zero. However, for the MRBC case, a key boundary energy dissipative rate $\iint_{\Sigma_-} |v \cdot n(x)| |L \gamma_+ f_{R,\eps}|^2 \d \sigma_{\Sigma_-}$ is implied by the boundary integral, provided that the wall Maxwellian $M_w (t, \bar{x}, v)$ in \eqref{M_w} is assumed to be $\M^0$. For details see Lemma \ref{Lemma-L2-Estimates}. There are also many other small and not so essential differences, we will not list them all here.

In the Hilbert expansion approach, we hope the number of expanded terms as small as possible. In this paper, we want to employ $L^2$-$L^\infty$ framework (see \cite{Guo-2010-ARMA,Guo-Jang-Jiang-2010-CPAM}, for instance) to prove that the remainder $\sqrt{\eps}^4 F_{R, \eps} (t, x, v)$ is a higher order infinitesimal as $\eps \to 0$. We claim that the expansion \eqref{Hilbert-Expnd-Form} is {\em optimal} in this framework. The uniform $L^2$-$L^\infty$ estimates rely on an interplay between $L^2$ and $L^\infty$ estimates for the Boltzmann equation. On one hand, following \cite{Guo-Jang-Jiang-2010-CPAM}, $L^2$-estimates will be such that the norm $\| f_{R, \eps} \|_2$ can be bounded by $\sqrt{\eps}^4 \| h_{R, \eps}^\ell \|_\infty$. On the other hand, the norm $\sqrt{\eps}^3 \| h_{R, \eps}^\ell \|_\infty$ will be dominated by $\sqrt{\eps}^3 \| h_{R, \eps}^\ell (0) \|_\infty + \| f_{R, \eps} \|_2 $. For details see Lemma \ref{Lemma-L2-Estimates} and Lemma \ref{Lemma-L-infty-Estimate} below. Based on the previous interplay, we can obtain the uniform bound of the quantity
\begin{equation}\label{Uniform_L2Linfty}
\begin{aligned}
\sup_{t \in [0, \tau]} \Big\{ \| f_{R, \eps} (t) \|_2 + \| \sqrt{\eps}^3 h_{R, \eps}^\ell (t) \|_\infty \Big\} \leq C \,,
\end{aligned}
\end{equation}
provided that $\| f_{R, \eps} (0) \|_2 + \| \sqrt{\eps}^3 h_{R, \eps}^\ell (0) \|_\infty \leq C$ initially holds uniformly for small $\eps > 0$. Consequently, we see that $\sqrt{\eps}^4 F_{R, \eps}$ is bounded by $\mathcal{O} (\sqrt{\eps})$ in $L^2$-$L^\infty$ framework.

Moreover, the order $\sqrt{\eps}^3$ occurred in the uniform bound \eqref{Uniform_L2Linfty} is independent of the number of expanded terms. Indeed, no matter how many terms are expanded, the remainder equations will contain a common linear structure
\begin{equation}\label{Line-Struct}
\begin{aligned}
\p_t f_{R, \eps} + v \cdot \nabla_x f_{R, \eps} + \tfrac{1}{\eps} \L f_{R, \eps} = \textrm{ some other terms} \,.
\end{aligned}
\end{equation}
In $L^2$ estimates, the hypocoercivity of the operator $\L$ in \eqref{Hypocoercivity} produces a dissipative rate $\tfrac{c_0}{\eps} \| (\I - \P) f_{R, \eps} \|^2_\nu$, which is the key point to deal with the singular terms. However, in $L^\infty$ estimates, after transforming $f_{R, \eps}$ to $h_{R,\eps}^\ell$ (by using \eqref{f_Reps-h_Reps}), the equation \eqref{Line-Struct} will be integrated along the trajectory $[X_{cl} (s; t, x, v), V_{cl} (s; t, x , v)]$ (see \eqref{t-V-X-Formula} below). We then have the pointwise estimate
\begin{align*}
| h_{R, \eps}^\ell (t, x, v) | \leq & C \| h_{R, \eps}^\ell (0) \|_\infty + \tfrac{1}{\eps^2} \int_0^t \exp\Big\{ -\tfrac{1}{\eps} \int_s^t \nu (\phi) \d \phi \Big\} \\
& \qquad \times \int_{\R^3 \times \R^3} \big| \Bbbk_{\ell} \big( V_{cl} (s), v^\prime \big) \Bbbk_{\ell} (v^\prime, v^{\prime\prime}) \big| \int_0^s \exp\Big\{ -\tfrac{1}{\eps} \int_{s_1}^s \nu (v^\prime) (\phi) \d \phi \Big\} \\
& \qquad \qquad \qquad \times \big| h^\ell_{R, \eps} \big( s_1, X_{cl} (s_1; s, X_{cl} (s), v^\prime), v^{\prime\prime} \big) \big| \d v^{\prime\prime} \d v^\prime \d s_1 \d s \\
+ & \textrm{ some other controllable quantities} \,,
\end{align*}
where $\Bbbk_\ell (v, v')$ satisfies \eqref{K_c-Kernal-Bound}. As shown in Case 3b of the proof for Lemma \ref{Lemma-L-infty-Estimate}, the second term in RHS of the previous inequality will be bounded by
\begin{equation*}
\begin{aligned}
\tfrac{C}{N_0} \sup_{s \in [0, t]} \| h_{R, \eps}^\ell (s) \|_\infty + \tfrac{C}{\sqrt{\eps}^3} \sup_{s \in [0, t]} \| f_{R, \eps} (s) \|_2
\end{aligned}
\end{equation*}
for $N_0 > 0$ large enough. Comparing to the $L^2$ estimates, we fail to seek a similar hypocoercivity \eqref{Hypocoercivity} in $L^\infty$ estimates, so that the singular term $\tfrac{1}{\eps} \L f_{R, \eps}$ is roughly treated as a source term. Therefore, there is a $\sqrt{\eps}^3$-order disparity between $\sup_{s \in [0, t]} \| h_{R, \eps}^\ell (s) \|_\infty$ and $\sup_{s \in [0, t]} \| f_{R, \eps} (s) \|_2$. Once we expand less number of the terms than that in \eqref{Hilbert-Expnd-Form}, the corresponding remainder $\sqrt{\eps}^m F_{R, \eps}$ ($m \leq 3$) will only be bounded by $\mathcal{O} (\sqrt{\eps}^{m-3})$, which is failed to achieve our goal. Nevertheless, this disparity maybe not intrinsic, and perhaps be covered in some other frameworks. So, we may expand less terms in the Hilbert expansion when proving the compressible Euler limit by employing the Hilbert expansion approach.

When constructing the viscous boundary layers $F^b_k$, one of key points is to focus on the existence of the smooth solutions to the linear compressible Prandtl-type equations \eqref{u-theta_Eq} (or \eqref{Linear-Prandtl}) with Robin-type boundary conditions \eqref{u-theta_BC} (or \eqref{Boudary_Equa}), which derives from the Maxwell reflection boundary condition \eqref{MBC}. Note that the linear compressible Prandtl-type system \eqref{u-theta_Eq} is a degenerated parabolic system with nontrivial Robin-type boundary values \eqref{u-theta_BC}. For the linear system, we can introduce some explicit functions to zeroize the inhomogeneous boundary values (see \eqref{theta_BC}), which actually converts the boundary values into the form of external force terms $(\tilde{f}, \tilde{g})$, see \eqref{tilde-f-g} below. We therefore obtain the system \eqref{Theta_Equa} with zero boundary values. Due to \eqref{Theta_Equa} is a degenerated parabolic system, we can construct a linear parabolic approximate system \eqref{Theta_Appro} while proving the existence of \eqref{Theta_Equa}. When deriving the uniform bounds of the approximate system in the space $L^\infty(0, \tau; \mathbb{H}^k_l(\R^3_+))$, the {\em key point} is to deal with the boundary values of the higher order normal $\zeta$-derivatives on $\{ \zeta = 0 \}$. The idea is employing the structures of equations to convert the higher normal $\zeta$-derivatives to the values of tangential $\bar{x}$-derivatives and time derivatives on $\{ \zeta = 0 \}$, see Lemma \ref{Lemma_Theta_BC_infty} below. Then we can find a key boundary energy
$$\sup_{t \in [0, \tau]} \| B_c (u, \theta) (t) \|^2_{\mathbb{H}^{k-1}(\R^2)} + \int_0^\tau \| B_c (u, \theta) (s) \|^2_{\mathbb{H}^k (\R^2)} \d s \,,$$
where $B_c (u, \theta) = (u, \theta) |_{\zeta=0} - (\tfrac{1}{R_u} b, \tfrac{1}{R_\theta} a)$ (see Lemma \ref{Prop-Prandtl}).

Furthermore, the solution $(u, \theta) (t, \bar{x}, \zeta)$ to the linear compressible Prandtl-type equations \eqref{u-theta_Eq} will produce a loss of derivatives with respect to the boundary values and source terms. More precisely, as shown in Lemma \ref{Prop-Prandtl}, if $(u, \theta) (t, \bar{x}, \zeta)$ is in $L^\infty(0, \tau; \mathbb{H}^k_l(\R^3_+))$, the source terms $(f, g) (t, \bar{x}, \zeta)$ and the boundary values $(b,a) (t, \bar{x})$ should be in $L^\infty (0, \tau; \mathbb{H}^{k+1}_l (\R^3_+))$ and $L^\infty (0, \tau; \mathbb{H}^{k+3}(\R^2))$, respectively. The reasons are as follows. For the system \eqref{Theta_Equa} zeroed boundary values, we shall use the structures of the equations to dominate the boundary values of higher order $\zeta$-derivatives, which involve the new source term $(\tilde{f}, \tilde{g})$ restricted on $\{\zeta=0\}$. As shown in Lemma \ref{Lemma_IV7B}, in order to control the boundary values in the space $L^\infty(0, \tau; \mathbb{H}^k (\R^2))$, the $(\tilde{f}, \tilde{g}) |_{\zeta=0}$ should be also in $L^\infty(0, \tau; \mathbb{H}^k (\R^2) )$. Combining the trace inequalities in Lemma \ref{Lemma_Trace}, the new source terms should be in $ \mathbb{H}^{k+1}_l (\R^3_+)$, so should be the source terms $(f,g)$. Noticing that $\p_t (u_b, \theta_a) \thicksim \p_t (b, a)$ occur in the new source terms $(\tilde{f}, \tilde{g})$ and a first order time derivative is equivalent to a second order tangential $\bar{x}$-derivative in the space $\mathbb{H}^k (\R^2)$, we therefore see that the boundary values $(b,a) (t, \bar{x})$ should be in $L^\infty (0, \tau; \mathbb{H}^{k+3}(\R^2))$.

\subsection{Organization of this paper}

In next section,  the uniform bounds to the expansions $F_k$, $F_k^b$ and $F_k^{bb}$ ($1 \leq k \leq 5$) appeared in \eqref{Hilbert-Expnd-Form} is derived. Section \ref{Sec_Remainder-Bnds} devotes to prove Theorem \ref{Main-Thm} by using the $L^2$-$L^\infty$ arguments. More precisely,  the $L^2$ and $L^\infty$ estimates are given. In Section \ref{Sec_Robin-BC}, it is formally derived the boundary conditions for the linear hyperbolic system \eqref{Linaer_Hyper_Sys} and linear compressible Prandtl-type equations \eqref{Linear-Prandtl}, namely, prove Lemma \ref{Lmm-Robin-BC}. In Section \ref{Sec_Prandtl}, the proof for the existence of the linear compressible Prandtl-type system \eqref{u-theta_Eq} with Robin-type boundary values \eqref{u-theta_BC} is given. Hence, we prove Lemma \ref{Prop-Prandtl}.

\section{Uniform bounds to the solutions of expansions}\label{Sec_Expansion}

In this section, we aim at deriving the uniform bounds to the expansions $F_k$, $F_k^b$ and $F_k^{bb}$ ($1 \leq k \leq 5$).

\subsection{Some auxiliary systems}

In the expansions \eqref{Hilbert-Expnd-Form}, there are some common types systems, which are only distinguished by some known coefficients. For $1 \leq k \leq 3$, the fluid parts of $F_k$ satisfy a linear hyperbolic system, and that of $F_k^b$ obey a linear compressible Prandtl-type systems with Robin-type boundary conditions. Moreover, $F_k^{bb}$ ($1 \leq k \leq 5$) subject to the well-known Knudsen boundary layer equations. We summarize quantitatively these system as follows.

We alway assume that $(\rho, \u, T)$ over the time interval $t \in [0, \tau]$ is the smooth solution to the compressible Euler system \eqref{BE} given in Proposition \ref{Proposition_Compressible_Euler} and $p = \rho T$. We also employ the notation
\begin{equation}\label{Ek}
  \begin{aligned}
    E_k := \sup_{t\in [0, \tau]} \| (\rho - \rho_\#, \u, T - T_\#) (t)\|_{H^k(\R^3_+)}
  \end{aligned}
\end{equation}
for any integer $k \geq 3$.

Let $(\tilde{\rho}, \tilde{u}, \tilde{\theta}) (t, x)$ satisfy the following linear hyperbolic system:
\begin{equation}\label{Linaer_Hyper_Sys}
\begin{aligned}
\left\{
\begin{array}{l}
\p_t \tilde{\rho} + \div_x (\rho \tilde{u} + \tilde{\rho} \u) =0\,, \\ [1.5mm]
\rho (\p_t \tilde{u} + \tilde{u} \cdot \nabla_x \u + \u \cdot \nabla_x \tilde{u}) - \frac{\nabla_x p}{\rho} \tilde{\rho} + \nabla \big( \tfrac{\rho \tilde{\theta} + 3 T \tilde{\rho}}{3} \big) = f\,,\\ [1.5mm]
\rho \big[ \p_t \tilde{\theta} + \tfrac{2}{3} \big( \tilde{\theta} \div_x \u + 3 T \div_x \tilde{u} \big) + \u \cdot \nabla_x \tilde{\theta} + 3 \tilde{u} \cdot \nabla_x T \big] =g\,,
\end{array}
\right.
\end{aligned}
\end{equation}
with $(t, x) \in (0, \tau) \times \R^3_+$. We impose the boundary condition
\begin{equation}\label{BC_LHS}
\tilde{u}_3 (t, \bar{x}, 0) = d (t,x)\,, \forall (t, \bar{x}) \in (0, \tau) \times \R^2
\end{equation}
and initial condition
\begin{equation}\label{Initial_LHS}
(\tilde{\rho}, \tilde{u}, \tilde{\theta}) (0, x) = (\tilde{\rho}_0, \tilde{u}_0, \tilde{\theta}_0) (x)\,.
\end{equation}

Recalling the definitions of the space $\mathcal{H}^k (\R^3_+)$ and $\mathcal{H}^k (\R^2)$ in \eqref{Notation_H_k}, we then have the following lemma for the local existence of smooth solution for the linear hyperbolic system \eqref{Linaer_Hyper_Sys}.

\begin{lemma}[\cite{GHW-2020}] \label{Lemma_Local_Sol_Hyperbolic_Sys}
	Assume that
	\begin{equation*}
	\mathbb{E}_0 : = \|(\tilde{\rho}^0, \tilde{u}^0, \tilde{\theta}^0)\|^2_{\mathcal{H}^k (\R^3_+)} + \sup_{t \in (0, \tau)} \big[\|(f, g)(t)\|^2_{\mathcal{H}^{k+1} (\R^3_+)} + \|d(t)\|_{\mathcal{H}^{k+2} (\R^2)}^2 \big] < +\infty
	\end{equation*}
	with $k \ge 3$, and the compatibility condition is satisfied for the initial data. Then there exists a unique smooth solution to \eqref{Linaer_Hyper_Sys}-\eqref{Initial_LHS} with boundary condition \eqref{BC_LHS} for $t \in [0, \tau]$, such that
	\begin{equation}\label{Energy_Inequ_LHS}
	\begin{aligned}
	\sup_{t \in [0, \tau]} {\|(\tilde{\rho}, \tilde{u}, \tilde{\theta}) (t)\|_{\mathcal{H}^k (\R^3_+)}^2} \le C (\tau, E_{k+2}) \mathbb{E}_0 \,.
	\end{aligned}
	\end{equation}
\end{lemma}

As shown in \eqref{Linear-Prandtl} and Lemma \ref{Lmm-Robin-BC}, to construct the solution of viscous boundary layer, the following linear compressible Prandtl type system of $(u, \theta) = (u_1, u_2, \theta) (t, \bar{x}, \zeta)$ should be considered: for $i = 1, 2$,
\begin{equation}\label{u-theta_Eq}
\left\{
\begin{aligned}
& \rho^0 \partial_t u_i + \rho^0 \bar{\u}^0 \cdot \nabla_{\bar{x}} u_i + \rho^0 ( \partial_{x_3} \u^0_3 \zeta + u^0_{1,3} ) \partial_\zeta u_i + \rho^0 u \cdot \nabla_{\bar{x}} \u^0_i + \tfrac{\partial_{x_i} p^0}{3 T^0} \theta = \mu (T^0) \partial_\zeta^2 u_i + f_i \,, \\
& \rho^0 \p_t \theta + \rho^0 \bar{\u}^0 \cdot \nabla_{\bar{x}} \theta + \rho^0 ( \p_{x_3} \u^0_3 \zeta + u^0_{1.3} ) \p_\zeta \theta + \tfrac{2}{3} \rho^0 \div_x \u^0 \theta = \tfrac{3}{5} \kappa(T^0) \p_\zeta^2 \theta + g \,,
\end{aligned}
\right.
\end{equation}
where $(t, \bar{x}, \zeta) \in [0, \tau] \times \R^2 \times \R_+$, $u^0_{1,3} = \sqrt{T^0} (1 + \rho^0 \sqrt{T^0})$ given in \eqref{BC-u_13}, and $(f_1, f_2, g) (t, \bar{x}, \zeta)$ are the known source functions. By the properties of compressible Euler system, we can assume the viscosity and heat conductivity, respectively,
\begin{equation}\label{mu0-kappa0}
\begin{aligned}
\mu (T^0) \geq \mu_0 > 0 \,, \ \tfrac{3}{5} \kappa (T^0) \geq \kappa_0 > 0
\end{aligned}
\end{equation}
for some constants $\mu_0, \kappa_0 > 0$. The system \eqref{u-theta_Eq} should be imposed on the non-homogeneous Robin boundary conditions, namely,
\begin{equation}\label{u-theta_BC}
\left\{
\begin{aligned}
\big( \partial_\zeta u_i - R_u u_i \big) \big|_{\zeta = 0} = b_i (t, \bar{x}) \,, \ \big( \partial_\zeta \theta - R_\theta \theta \big) \big|_{\zeta = 0} = a (t, \bar{x} ) \,, \\
\lim_{\zeta \to \infty} (u, \theta) (t, \bar{x}, \zeta) = 0 \,,
\end{aligned}
\right.
\end{equation}
where the known functions $(R_u, R_\theta) = (R_u, R_\theta) (\rho^0, u^0, T^0)$, and
\begin{equation}\label{LowBnd-1}
\begin{aligned}
R_u \geq R^0_u > 0 \,, R_\theta \geq R^0_\theta > 0
\end{aligned}
\end{equation}
for some constants $R_u^0$ and $R_\theta^0$. We impose \eqref{u-theta_Eq} with initial data
\begin{equation}\label{u-theta_IC}
\begin{aligned}
u (t, \bar{x}, \zeta) |_{t = 0} = u_0 (\bar{x}, \zeta) \,, \ \theta (t, \bar{x}, \zeta) |_{t = 0} = \theta_0 (\bar{x}, \zeta) \,,
\end{aligned}
\end{equation}
which satisfy the corresponding compatibility conditions. We remark that the $\theta$-equation in \eqref{u-theta_Eq}-\eqref{u-theta_IC} is independent of $u$-equation. As a result, we can first solve $\theta$, and then solve $u$, where $\theta$ involved in $u$-equation can be regarded as the known source function.

Recalling the definitions of $\mathbb{H}^k_l (\R^3_+)$ and $\mathbb{H}^k(\R^2)$ in \eqref{Hrl-3D} and \eqref{Hr-2D}, we give the following results.

\begin{lemma}\label{Prop-Prandtl}
	Let $k \geq 3$, $l \geq 0$, and the compatibility conditions for the initial data \eqref{u-theta_IC} be satisfied. Assume
	\begin{equation}\label{IC_norm_theta}
	\begin{aligned}
	\mathbb{E}_1 : = \| u_0 \|^2_{\mathbb{H}^k_l (\R^3_+)} & + \sup_{t\in [0, \tau]} \big( \| b (t) \|^2_{\mathbb{H}^{k+3} (\R^2)} + \| f (t) \|^2_{\mathbb{H}^{k+1}_l (\R^3_+)} \big) \\
	& + \| \theta_0 \|^2_{\mathbb{H}^k_l (\R^3_+)} + \sup_{t\in [0, \tau]} \big( \| a (t) \|^2_{\mathbb{H}^{k+3} (\R^2)} + \| g (t) \|^2_{\mathbb{H}^{k+1}_l (\R^3_+)} \big) < \infty \,.
	\end{aligned}
	\end{equation}
	Then there exists a unique smooth solution $( u, \theta) (t, \bar{x}, \zeta)$ to \eqref{u-theta_Eq}-\eqref{u-theta_IC} over $t \in [0, \tau]$ satisfying
	\begin{equation}\label{theta-bnd}
	  \begin{aligned}
	    \sup_{t\in [0, \tau]} \Big( \| (u, \theta) & (t) \|^2_{\mathbb{H}^k_l (\R^3_+)} + \| B_c (u, \theta) (t) \|^2_{\mathbb{H}^{k-1}(\R^2)} \\
	    & + \int_0^t \| \p_{\zeta} (u, \theta ) (s) \|^2_{\mathbb{H}^k_l (\R^3_+)} + \| B_c (u, \theta) (s) \|^2_{\mathbb{H}^k(\R^2)} \d s \Big) \leq C (\tau, E_{k+1}) \mathbb{E}_1 \,,
	  \end{aligned}
	\end{equation}
	where $B_c (u, \theta) = (u, \theta) |_{\zeta=0} - (\tfrac{1}{R_u} b, \tfrac{1}{R_\theta} a)$.
\end{lemma}

The proof of Lemma \ref{Prop-Prandtl} will be given in Section \ref{Sec_Prandtl}.

We then introduce a result on the existence of solutions of the Knudsen boundary layer problem. Consider the following half-space linear equation for $f(t, \bar{x}, \xi, v)$ over $(t, \bar{x}, \xi, v) \in [0, \tau] \times \R^2 \times \R_+ \times \R^3$:
\begin{equation}\label{Knudsen_Boundary_Layer}
\left\{
\begin{aligned}
& v_3 \p_{\xi} f + \mathcal{L}^0 f = S (t, \bar{x}, \xi, v) \,, \\
& f (t, \bar{x}, 0, \bar{v}, v_3) |_{v_3 >0} = f (t, \bar{x}, 0, \bar{v}, -v_3) + \mathbbm{f}_k (t, \bar{x}, \bar{v}, -v_3) \,, \\
& \lim_{\xi \to \infty} f (t, \bar{x}, \xi, v) = 0\,.
\end{aligned}
\right.
\end{equation}
For the above equation, Golse, Perthame and Sulem \cite{Golse-Perthame-Sulem-1988-ARMA} has proved an existence result in the norm $\int_{\R_+ \times \R^3} \l v \r e^{2 \eta \xi} f^2 \d v \d \xi + \int_{\R^3} \| e^{\eta \xi} f \|^2_{L^\infty_\xi} \d v$. Due to the higher order derivatives are required, Guo, Huang and Wang \cite{GHW-2020} gave a more smooth version of existence result, hence the following lemma.

\begin{lemma}\label{Propisition_Knudsen_Layer}
	Let $0 < a < \frac{1}{2}$, $r \geq 0$ and $l \ge 3$. For each $(t, \bar{x}) \in [0, \tau] \times \R^2$, we assume that $S \in (\mathcal{N}^0)^\bot$, $\mathbbm{f}_k (t, \bar{x}, v)$ satisfies the solvability condition \eqref{Solva_Cond}, and
	\begin{equation*}
	\begin{aligned}
	   \mathbb{E}_2 : = \sup_{t \in [0, \tau]} \! \sum_{|\alpha| \le r} \Big\{ \| \l v \r^l (\M^0)^{-a} \p_{t, \bar{x}}^\alpha & \mathbbm{f}_k (t) \|_{L^\infty_{\bar{x},v} \cap L^2_{\bar{x}} L^\infty_v } \\
	   & + \| \l v \r^l (\M^0)^{-a} e^{\lambda_0 \xi} \p_{t, \bar{x}}^\alpha S (t) \|_{L^\infty_{\bar{x}, \xi, v} \cap L^2_{\bar{x}} L^\infty_{\xi, v}}  \Big\} < + \infty
	\end{aligned}
	\end{equation*}
	for some positive constant $\lambda_0 >0$. Then the Knudsen boundary layer equation \eqref{Knudsen_Boundary_Layer} has a unique solution $f(t, \bar{x}, \xi, v)$ satisfying
	\begin{align}\label{Knudsen_Boundary_Layer_Bound}
	\no \sum_{|\alpha| \le r} \sup_{t \in [0, \tau]} \Big\{ \| \l v \r^l (\M^0)^{-a} \p_{t, \bar{x}}^\alpha & f (t, \cdot, 0, \cdot) \|_{L^\infty_{\bar{x},v} \cap L^2_{\bar{x}} L^\infty_v } \\
	& + \| \l v \r^l (\M^0)^{-a} e^{\lambda \xi} \p_{t, \bar{x}}^\alpha f (t) \|_{L^\infty_{\bar{x}, \xi, v} \cap L^2_{\bar{x}} L^\infty_{\xi, v}} \Big\} \le \tfrac{C}{\lambda_0 - \lambda} \mathbb{E}_2
	\end{align}
	for all $\lambda \in (0, \lambda_0)$, where $C > 0$ is independence of $(t, \bar{x})$. Moreover, if $S$ is in $C \big( [0, \tau] \times \R^2 \times \R_+ \times \R^3 \big)$ and $\mathbbm{f}_k (t, \bar{x}, \bar{v}, -v_3)$ is in $C \big( [0, \tau] \times \R^2 \times \R^2 \times \R_+ \big)$, then the solution $f(t, \bar{x}, \xi, v)$ is continuous away from the grazing set $[0, \tau] \times \Sigma_0$.
\end{lemma}
We remark that it is hard to obtain the normal derivatives estimates for the boundary value problem \eqref{Knudsen_Boundary_Layer}. Lemma \ref{Propisition_Knudsen_Layer} merely give the tangential and time derivatives estimates.

\subsection{Control the expansions $F_k$, $F_k^b$ and $F_k^{bb}$ ($1 \leq k \leq 5$)}

Based on results in Lemma \ref{Lemma_Local_Sol_Hyperbolic_Sys}, Lemma \ref{Prop-Prandtl} and \ref{Propisition_Knudsen_Layer}, following the analogously arguments in Proposition 5.1 of \cite{GHW-2020}, we obtain the uniform bounds of the expansions $F_k$, $F_k^b$ and $F_k^{bb}$ ($1 \leq k \leq 5$) as follows. For simplicity, we omit the details of proof here.

\begin{proposition}\label{Pro_Regularities-Coefficients}
	Let $0 < \tfrac{1}{2 z} (1-z) < a < \frac{1}{2}$, where $z \in (\frac{1}{2}, 1)$ is given in \eqref{M-Bound}, and $s_0 \in \mathbb{N}_+$ be in Proposition \ref{Proposition_Compressible_Euler}. There are $s_k, s_k^b, s_k^{bb}, l_k^b \in \mathbb{N}_+$, $p_k, p_k^b, p_k^{bb} \in \R_+$ $(1 \leq k \leq 5)$ satisfying
	\begin{equation*}
	  \begin{aligned}
	    & s_0 \geq s_1 + 10 \,, \ s_1 = s_1^b = s_1^{bb} \,, \ s_{k-1}^{bb} \gg s_k \gg s_k^b \gg s_k^{bb} \gg 1 \ (2 \leq k \leq 5) \,, \\
	    & p_k \gg p_k^b \gg p_k^{bb} \gg p_{k+1} \gg 1 \ (1 \leq k \leq 4) \,,
	  \end{aligned}
	\end{equation*}
	and $l_j^k = l_k^b + 2 (s_k^b - j)$ $(0 \leq j \leq s_k^b, 1 \leq k \leq 5)$ with
	\begin{equation*}
	  \begin{aligned}
	    l^5_j \geq 8 \,, \ l_j^k \geq 2 l_j^{k+1} + 26 \ (1 \leq k \leq 4) \,,
	  \end{aligned}
	\end{equation*}
	such that if the initial data $(\rho_k^{in}, u_k^{in}, \theta_k^{in})$ $(1 \leq k \leq 3)$ in \eqref{IC-Linear-Hyperbolic} and $(\bar{u}^{b, in}_k, \theta_k^{b, in})$ $(1 \leq k \leq 3)$ in \eqref{IC-Prandtl} satisfy \eqref{Ein}, i.e., $\mathcal{E}^{in} < \infty$, then there are solutions $F_k = \sqrt{\M} f_k$, $F^b_k = \sqrt{\M^0} f^b_k$ and $F_k^{bb} = \sqrt{\M^0} f_k^{bb}$ $(1 \leq k \leq 5)$ constructed in Subsection \ref{subsubsec115} over the time interval $t \in [0, \tau]$ subjecting to the uniform bounds
	  \begin{align}
	    \no \sup_{t \in [0, \tau]} \sum_{k=1}^5 \Big\{ & \sum_{\gamma + |\beta| \leq s_k} \| \l v \r^{p_k} \M^{-a} \p_t^\gamma \p_x^\beta f_k (t) \|_{L^2_x L^\infty_v} \\
	    \no & + \sum_{j=0}^{s_k^b} \sum_{2 \gamma + |\bar{\beta}| = j} \| \l v \r^{p_k^b} (\M^0)^{-a} \p_t^\gamma \p_{\bar{x}}^{\bar{\beta}} f^b_k (t) \|_{L^2_{l^k_j} L^\infty_v} \\
	    \no & + \sum_{\gamma + |\bar{\beta}| \leq s_k^{bb}} \| e^{ \frac{\xi}{2^{k-1}}} \l v \r^{p_k^{bb}} (\M^0)^{-a} \p_t^\gamma \p_{\bar{x}}^{\bar{\beta}} f_k^{bb} (t) \|_{L^\infty_{\bar{x}, \xi, v} \cap L^2_{\bar{x}} L^\infty_{\xi, v}} \Big\} \\
	    \leq & C \Big(\tau, \| (\rho^{in} - \rho_\#, \u^{in}, T^{in} - T_\#) \|_{H^{s_0} (\R^3_+)} + \mathcal{E}^{in} \Big) \,.
	  \end{align}
\end{proposition}

We remark that, for $1 \leq k \leq 5$, the viscous boundary layers $F_k^b$ decay algebraically associated with $\zeta > 0$, and the Knudsen boundary layers $F_k^{bb}$ decay exponentially associated with $\xi > 0$. These suffice to dominate the quantities $R_\eps^b$ and $R_\eps^{bb}$ (in \eqref{R_eps_b} and \eqref{R_eps_bb}, respectively) while deriving the uniform $L^2$-$L^\infty$ bounds for the remainder $F_{R,\eps}$.

\section{Uniform bounds for remainder $F_{R,\eps}$: Proof of Theorem \ref{Main-Thm}}\label{Sec_Remainder-Bnds}

In this section, we aim at proving our main theorem by applying the $L^2$-$L^\infty$ arguments \cite{Guo-2010-ARMA,Guo-Jang-Jiang-2009-KRM}, which is sufficient to estimate $\| f_{R, \eps} (t) \|_2$ and $\| h^{\ell}_{R, \eps} (t) \|_\infty$ associated with the remainder $F_{R,\eps}$ in \eqref{Remainder-F}. Here $f_{R,\eps}$ and $h_{R,\eps}^\ell$ are defined in \eqref{f_Reps-h_Reps}. The proof relies on an interplay between $L^2$ and $L^\infty$ estimates for the Boltzmann equation. The $L^2$ norm of $f_{R, \eps}$ is controlled by the $L^\infty$ norm of the high-velocity part and vice versa.

Note that the remainder equation \eqref{Remainder-F} with Maxwell-type reflection boundary condition contains the coefficients $F_k$, $F_k^b$, $F_k^{bb}$ ($1 \leq k \leq 5$) and $R_\eps$, $R_\eps^b$, $R_\eps^{bb}$ composed of $F_k$, $F_k^b$, $F_k^{bb}$ ($1 \leq k \leq 5$). Thanks to Proposition \ref{Pro_Regularities-Coefficients}, following the similar arguments in Section 6 of \cite{GHW-2020}, we can bound the all norms associated with them appeared in estimates of $L^2$-$L^\infty$ arguments by some constants independent of $\eps$. For simplicity of presentation, the corresponding quantities will be directly bounded by a constant $C > 0$ without details in what follows.

We now state the following two key lemmas.

\begin{lemma}[$L^2$ Estimates]\label{Lemma-L2-Estimates}
	Under the same assumptions in Proposition \ref{Pro_Regularities-Coefficients} and $\ell \geq 9 - 2 \gamma_0$, let $(\rho, \u, T)$ be a smooth solution to the Euler equations over $t \in [0, \tau]$ obtained in Proposition \ref{Proposition_Compressible_Euler}. Let $c_0 > 0$ be mentioned in \eqref{Hypocoercivity}, and the local Maxwellian $M_w (t, \bar{x}, v)$ of the boundary be assumed as in \eqref{M_w-Perturbation}. Then there are constants $\eps_0' > 0$ and $
	C = C \big( \M, F_k, F^b_k, F^{bb}_k; 1 \leq k \leq 5 \big) >0 $ such that for all $0 < \eps < \eps_0'$,
	\begin{align}\label{L2-Estimate}
	\no \tfrac{\d }{\d t} \| f_{R, \eps} \|_2^2 + \tfrac{c_0}{\eps} \| (\mathcal{I} - \mathcal{P}) f_{R,\eps} \|_\nu^2 + c_1 \iint_{\Sigma_-} | v \cdot n(x) | | L \gamma_+ f_{R, \eps} |^2 \d \sigma_{\Sigma_-} \\
	\le C \big(1 + \eps^2 \| h^{\ell}_{R,\eps} \|_\infty \big) \big( \| f_{R,\eps} \|_2^2 + \| f_{R,\eps} \|_2 \big) + C \sqrt{\eps}^2
	\end{align}
	over $t \in [0, \tau]$, where $c_1 = 1- 3 \rho_{\#} \sqrt{3 T_\#} > 0$ and $\rho_\#, T_\# >0$ are given in Proposition \ref{Proposition_Compressible_Euler}.
\end{lemma}

\begin{lemma}[$L^\infty$ Estimate] \label{Lemma-L-infty-Estimate}
	Under the same assumptions in Lemma \ref{Lemma-L2-Estimates}, there are constants $\eps_0'' > 0$ and $
	C = C \big( \M, F_k, F^b_k, F^{bb}_k; 1 \leq k \leq 5 \big) >0 $ such that for all $0 < \eps < \eps_0''$,
	\begin{align} \label{L-infty-Estimate}
	\sup_{s \in [0, \tau]} \| \sqrt{\eps}^3 h^{\ell}_{R, \eps} (s) \|_\infty
	\le C \Big( \| \sqrt{\eps}^3 h^{\ell}_{R, \eps} (0) \|_\infty + \sup_{0 \le s \le t} \| f_{R, \eps} (s) \|_2 + \sqrt{\eps}^2 \Big) \,.
	\end{align}
\end{lemma}

\begin{proof}[Proof of Theorem \ref{Main-Thm}]
	Based on Lemma \ref{Lemma-L2-Estimates} and \eqref{Lemma-L-infty-Estimate}, one has
	\begin{equation*}
	  \begin{aligned}
	    \tfrac{\d}{\d t} \| f_{R,\eps} \|_2^2 + \tfrac{c_0}{\eps} \| (\mathcal{I} - \mathcal{P}) f_{R,\eps} \|_\nu^2 + c_1 \iint_{\Sigma_-} | v \cdot n(x) | | L \gamma_+ f_{R, \eps} |^2 \d \sigma_{\Sigma_-} \\
	    \leq C \big[ 1 + \sqrt{\eps} ( \| \sqrt{\eps}^3 h^{\ell}_{R, \eps} (0) \|_\infty + \sup_{0 \le s \le t} \| f_{R, \eps} (s) \|_2 + \sqrt{\eps}^2 ) \big] (\| f_{R,\eps} \|^2_2 + \| f_{R,\eps} \|_2) + C \sqrt{\eps}^2 \,.
	  \end{aligned}
	\end{equation*}
	The Gr\"onwall inequality yields
	\begin{equation*}
	  \begin{aligned}
	    \| f_{R,\eps} (t) \|_2^2 + 1 \leq (\| f_{R,\eps} (0) \|_2^2 + 1) e^{C t \big\{ 1 + \sqrt{\eps}^3 + \sqrt{\eps} ( \| \sqrt{\eps}^3 \| h_{R,\eps} (0) \|_\infty + \sup_{t \in [0, \tau]} \| f_{R,\eps} (t) \|_2 ) \big\}} \,.
	  \end{aligned}
	\end{equation*}
	For bounded $\sup_{t \in [0, \tau]} \| f_{R,\eps} (t) \|_2$ and small $\eps \in (0, \min\{ \eps_0', \eps_0'' \} )$, utilizing the Taylor expansion of the exponential function in the above inequality, one has
	\begin{equation*}
	  \begin{aligned}
	    \| f_{R,\eps} (t) \|_2^2 + 1 \leq C_1 (\| f_{R,\eps} (0) \|_2^2 + 1) \big[ 1 + \sqrt{\eps} ( \| \sqrt{\eps}^3 \| h_{R,\eps} (0) \|_\infty + \sup_{t \in [0, \tau]} \| f_{R,\eps} (t) \|_2 ) \big] \,.
	  \end{aligned}
	\end{equation*}
	For $t \leq \tau$, by using the Young inequality, we conclude the proof of Theorem \ref{Main-Thm} by the bound
	\begin{equation*}
	  \begin{aligned}
	    \sup_{t \in [0, \tau]} \| f_{R,\eps} (t) \|_2 \leq C_\tau \big( 1 + \| f_{R, \eps} (0) \|_2 + \| \sqrt{\eps}^3 h_{R,\eps}^\ell (0) \|_\infty \big) \,.
	  \end{aligned}
	\end{equation*}
\end{proof}

\subsection{$L^2$-estimate: Proof of Lemma \ref{Lemma-L2-Estimates}}

From \eqref{Remainder-F} and \eqref{f_Reps-h_Reps}, we see that $f_{R,\eps}$ satisfies
\begin{equation}\label{f_R_epsilon}
\begin{aligned}
\p_t f_{R, \varepsilon} + v \cdot \nabla_{x} f_{R, \varepsilon} +  \tfrac{1}{\varepsilon} L f_{R, \varepsilon} = - \tfrac{ \{\p_t + v \cdot \nabla_x \} \sqrt{\M} }{\sqrt{\M}} f_{R, \eps} +  \tfrac{\sqrt{\eps}^2}{\sqrt{\M}} B(\sqrt{\M} f_{R, \eps}, \sqrt{\M} f_{R,\eps}) \\
+ \sum_{k=1}^5 \sqrt{\eps}^{k-2} \tfrac{1}{\sqrt{\M}} \Big\{ B(F_k + F_k^b + F_k^{bb}, \sqrt{\M} f_{R, \eps}) + B(\sqrt{\M} f_{R, \eps}, F_k + F_k^b + F_k^{bb}) \Big\} \\
+ \tfrac{1}{\sqrt{\M}} ( R_\eps + R_\eps^b + R_\eps^{bb} )
\end{aligned}
\end{equation}
with boundary condition
\begin{equation}\label{f_R-BC}
\begin{aligned}
  f_{R, \eps} (t, \bar{x}, 0, \bar{v}, v_3) |_{v_3 >0} = f_{R, \eps} (t, \bar{x}, 0, \bar{v}, -v_3) + \sqrt{\eps} \tilde{\Gamma}_\eps \,,
\end{aligned}
\end{equation}
where
\begin{equation}\label{tilde-Gamma}
\begin{aligned}
\tilde{\Gamma}_\eps =  \sqrt{\eps} \tfrac{1}{\sqrt{\M^0}} \Gamma_\eps - \tfrac{\alpha_\eps}{\sqrt{\eps}} \Big\{ & f_{R, \eps} (t, \bar{x}, 0, \bar{v}, -v_3) \\
& - \sqrt{2\pi} \tfrac{M_w (v)}{\sqrt{\M^0}} \int_{\R^2} \int_0^{+\infty} v_3 f_{R, \eps} (t, \bar{x}, 0, \bar{v}, -v_3) \sqrt{\M^0} \d \bar{v} \d v_3 \Big\} \,.
\end{aligned}
\end{equation}
Here $\Gamma_\eps$ is given in \eqref{Gamma_eps}.

Multiplying the equation \eqref{f_R_epsilon} by $f_{R,\eps}$, integrating the resultant equation over $\R^3_+ \times \R^3$ and using \eqref{Hypocoercivity}, one obtains
\begin{align}\label{Remiander_Eq-L2-esimate-1}
\no & \tfrac{1}{2} \tfrac{\d}{\d t} \| f_{R, \eps} \|_2^2 + \tfrac{c_0}{\eps} \| (\mathcal{I} - \mathcal{P}) f_{R, \eps} \|_\nu^2 - \tfrac{1}{2} \iint_{\p \R^3_+ \times \R^3} v_3 | f_{R, \eps} (t, \bar{x}, 0, v) |^2 \d v \d \bar{x} \\
\no = & - \iint_{\R^3_+ \times \R^3} \tfrac{(\p_t + v \cdot \nabla_x) \sqrt{\M}}{\sqrt{\M}} |f_{R, \eps}|^2 \d v \d x + \sqrt{\eps}^2 \iint_{\R^3_+ \times \R^3} \tfrac{1}{\sqrt{\M}} B (\sqrt{\M} f_{R,\eps}, \sqrt{\M} f_{R, \eps}) f_{R, \eps} \d v \d x \\
\no + & \sum_{k=1}^5 \sqrt{\eps}^{k-2} \iint_{\R^3_+ \times \R^3} \tfrac{1}{\sqrt{\M}} \big\{ B(F_k + F^b_k + F^{bb}_k, \sqrt{\M} f_{R, \eps}) \\
\no & \qquad \qquad \qquad \qquad \qquad \qquad \qquad \qquad + B(\sqrt{\M} f_{R, \eps}, F_k + F^b_k + F^{bb}_k) \big\} f_{R, \eps} \d v \d x \\
+ & \iint_{\R^3_+ \times \R^3} \tfrac{1}{\sqrt{\M}} \big( R_\eps + R^b_\eps + R^{bb}_\eps \big) \d v \d x \,.
\end{align}
Following the similar arguments in Section 2.1 of \cite{Guo-Jang-Jiang-2010-CPAM},
\begin{align}\label{RE-L2-RB}
 \text{RHS of \eqref{Remiander_Eq-L2-esimate-1}}
\le C_\lambda \eps^2 \| h^{\ell}_{R, \eps} \|_\infty \| f_{R, \eps} \|_2 + C \| f_{R, \eps} \|_2 + C \| f_{R, \eps} \|_2^2 + \tfrac{C \lambda^{3 - \gamma_0}}{\eps} \| (\mathcal{I} - \mathcal{P}) f_{R, \eps} \|_\nu^2
\end{align}
for some small $\lambda >0$ to be determined.

We next focus on estimating the boundary integral term $ - \tfrac{1}{2} \iint_{\p \R^3_+ \times \R^3} v_3 | f_{R, \eps} (t, \bar{x}, 0, v) |^2 \d v \d \bar{x}$. Recalling the boundary condition \eqref{f_R-BC}, one has
\begin{align}\label{L2-Boundary-Term}
\no & \tfrac{1}{2} \iint_{\p \R^3_+ \times \R^3} v_3 |f_{R, \eps} (t, \bar{x}, 0, v)|^2 \d v \d \bar{x} \\
\no = & \tfrac{1}{2} \iint_{\R^2 \times \R^2} \Big\{ \int_{-\infty}^0 + \int_0^{+ \infty} \Big\} v_3 |f_{R, \eps} (t, \bar{x}, 0, \bar{v}, v_3)|^2 \d v_3 \d \bar{v} \d \bar{x} \\
\no = & \tfrac{1}{2} \iint_{\R^2 \times \R^2} \int_{-\infty}^0 v_3 |f_{R, \eps} (t, \bar{x}, 0, \bar{v}, v_3)|^2 \d v_3 \d \bar{v} \d \bar{x} \\
\no & \qquad \qquad \qquad \qquad + \tfrac{1}{2} \iint_{\R^2 \times \R^2} \int_0^{+\infty} v_3 | f_{R, \eps} (t, \bar{x}, 0, \bar{v}, -v_3) + \sqrt{\eps} \widetilde{\Gamma}_\eps|^2 \d v_3 \d \bar{v} \d \bar{x} \\
= & \underbrace{ \sqrt{\eps} \iint\limits_{\R^2 \times \R^2} \int_0^{+\infty} v_3 \widetilde{\Gamma}_\eps f_{R, \eps} (t, \bar{x}, 0, \bar{v}, -v_3) \d v_3 \d \bar{v} \d \bar{x} }_{I_1} + \underbrace{\tfrac{\sqrt{\eps}^2}{2}  \iint\limits_{\R^2 \times \R^2} \int_0^{+\infty} v_3 |\widetilde{\Gamma}_\eps|^2 \d v_3 \d \bar{v} \d \bar{x} }_{I_2} \,.
\end{align}
By using the definition of $\widetilde{\Gamma}_\eps$ in \eqref{tilde-Gamma}, one has
\begin{align}\label{L2-Boundary-Term-I1}
\no I_1 = & - \alpha_\eps \iint\limits_{\R^2 \times \R^2} \int_0^{+\infty} v_3 | f_{R, \eps} (t, \bar{x}, 0, \bar{v}, -v_3)|^2 \d \bar{v} \d v_3 \d \bar{x} \\
\no & + \alpha_\eps \sqrt{2 \pi} \iint\limits_{\R^2 \times \R^2} \int_0^{+\infty} v_3 f_{R, \eps} (t, \bar{x}, 0, \bar{v}, -v_3) \cdot \tfrac{M_w (v)}{\sqrt{\M^0}} \\
\no & \qquad \times \Big( \int_{\R^2} \int_0^{+\infty} v_3^\prime f_{R, \eps} (t, \bar{x}, 0, \bar{v}, -v_3^\prime) \sqrt{\M^0}  \d \bar{v} \d v_3^\prime \Big) \d v_3 \d \bar{v} \d \bar{x} \\
& + 2 \sqrt{\eps}^3 \iint\limits_{\R^2 \times \R^2} \int_0^{+\infty} v_3  f_{R, \eps} (t, \bar{x}, 0, \bar{v}, -v_3) \Gamma_\eps \d v_3 \d \bar{v} \d \bar{x} =: I_{11} + I_{12} + I_{13} \,.
\end{align}
Obviously, the term $I_{11}$ is negative, thus it is a  good term in the $L^2$-estimates as a boundary decay rate. We then control the quantity $I_{12}$. Recalling the assumption \eqref{M_w-Perturbation} of the local Maxwellian $M_w (t, \bar{x}, v)$ of the boundary, we derive from direct calculation that
\begin{align*}
  I_{12} = \sqrt{2 \pi} \alpha_\eps \int_{\R^2} \Big( \int_{\R^2} \int_0^{+\infty} v_3 f_{R, \eps} (t, \bar{x}, 0, \bar{v}, -v_3) \sqrt{\M^0} \d \bar{v} \d v_3 \Big)^2 \d \bar{x} \,.
\end{align*}
By H\"older inequality and the fact $\int_{\R^2} \int_0^{+\infty} v_3 \M^0 \d \bar{v} \d v_3 = \frac{\rho^0 \sqrt{T^0}}{\sqrt{2 \pi}}$, we further infer that
\begin{align}\label{L2-Boundary-Term-I-12-Bound}
\no |I_{12}| \le & \alpha_\eps \int_{\R^2} \rho^0 \sqrt{T^0} \Big( \int_{\R^2} \int_0^{+\infty} v_3 | f_{R, \eps} (t, \bar{x}, 0, \bar{v}, -v_3) |^2 \d \bar{v} \d v_3 \Big) \d \bar{x} \\
\no \le & \| \rho^0 \sqrt{T^0} \|_{L^\infty ([0, \tau] \times \R^2)} \cdot \alpha_\eps \iint_{\R^2 \times \R^2} \int_0^{+\infty} v_3 | f_{R, \eps} (t, \bar{x}, 0, \bar{v}, -v_3) |^2 \d \bar{v} \d v_3 \d \bar{x} \\
\le & - 3 \sqrt{3} \rho_\# \sqrt{T_\#} I_{11} \,,
\end{align}
where the last inequality is derived from the Proposition \ref{Proposition_Compressible_Euler}. Here the constant $3 \sqrt{3} \rho_\# \sqrt{T_\#} \in (0, 1)$ and the symbol $I_{11}$ is given in \eqref{L2-Boundary-Term-I1}.

Moreover, it is deduced from the H\"older inequality and Proposition \ref{Pro_Regularities-Coefficients} that
\begin{align}\label{L2-Boundary-Term-I13}
\no |I_{13}|
\le & \sqrt{\eps}^2 \alpha_\eps \iint_{\R^2 \times \R^2} \int_0^{+\infty} v_3 | f_{R, \eps} (t, \bar{x}, 0, \bar{v}, -v_3) |^2 \d \bar{v} \d v_3 \d \bar{x} \\
& + \frac{1}{4 \sqrt{2 \pi}} \sqrt{\eps}^3 \iint_{\R^2 \times \R^2} \int_0^{+\infty} v_3 | \Gamma_\eps (t, \bar{x}, v) |^2 \d \bar{v} \d v_3 \d \bar{x} \le - \sqrt{\eps}^2 I_{11} + C \sqrt{\eps}^3 \,.
\end{align}
Together with the bounds \eqref{L2-Boundary-Term-I1}, \eqref{L2-Boundary-Term-I-12-Bound} and \eqref{L2-Boundary-Term-I13}, one sees that
\begin{align}\label{L2-Boundary-Term-I1-Bound}
I_1
\le & \big(c_1 - \eps \big) I_{11} + C \sqrt{\eps}^3 \,.	
\end{align}
where $c_1 = 1 - 3 \sqrt{3} \rho_\# \sqrt{T_\#} > 0$

Similarly in estimating of \eqref{L2-Boundary-Term-I1-Bound}, the quantity $|I_2|$ can be bounded by
\begin{align}\label{L2-Boundary-Term-I2-Bound}
|I_2| \le - C_\Gamma \sqrt{\eps} I_{11} + C_\Gamma \sqrt{\eps}^2
\end{align}
for some constant $C_\Gamma >0$. We consequently derive from substituting \eqref{L2-Boundary-Term-I1-Bound} and \eqref{L2-Boundary-Term-I2-Bound} into \eqref{L2-Boundary-Term} that
\begin{align*}
& - \tfrac{1}{2} \iint_{\R^2 \times \R^2} \int_0^{+\infty} v_3 | f_{R, \eps} (t, \bar{x}, 0, v) |^2 \d v \d \bar{x} = -(I_1 + I_2) \\
\ge & \big(c_1 - \eps - C_\Gamma \sqrt{\eps} \big) \alpha_\eps \iint_{\R^2 \times \R^2} \int_0^{+\infty} v_3 | f_{R, \eps} (t, \bar{x}, 0, \bar{v}, -v_3) |^2 \d \bar{v} \d v_3 \d \bar{x} - C \sqrt{\eps}^3 - C_\Gamma \sqrt{\eps}^2 \,,
\end{align*}
We then choose
\begin{align*}
  \eps_0' = c_1 \min\big\{ \tfrac{1}{8}, \tfrac{1}{8 C_\Gamma} \big\} >0 \,.
\end{align*}
Therefore, for any $\eps \in (0, \eps_0')$,
\begin{align}
\no & - \tfrac{1}{2} \iint_{\p \R^3_+ \times \R^3} v_3 | f_{R, \eps} (t, \bar{x}, 0, v) |^2 \d v \d \bar{x} \\
\ge & \tfrac{1}{2} c_1 \iint_{\R^2 \times \R^2} \int_0^{+\infty} v_3 | f_{R, \eps} (t, \bar{x}, 0, \bar{v}, - v_3) |^2 \d \bar{v} \d v_3 \d \bar{x} - C \sqrt{\eps}^2
\end{align}
for some constant $ C > 0$. Together with  \eqref{Remiander_Eq-L2-esimate-1} and \eqref{RE-L2-RB}, we have
\begin{align}\label{Remainder_Eq-L2}
\no & \tfrac{1}{2} \tfrac{\d}{\d t} \| f_{R, \eps} \|_2^2 + \tfrac{c_0 - C \lambda^{3- \gamma_0}}{\eps} \| (\mathcal{I} - \mathcal{P}) f_{R, \eps} \|_\nu^2  \\
& + \tfrac{1}{2} c_1 \iint_{\R^2 \times \R^2} \int_0^{+\infty} v_3 | f_{R, \eps} (t, \bar{x}, 0, \bar{v}, - v_3) |^2 \d v_3 \d \bar{v} \d \bar{x} \\
\no \le & C_\lambda \eps^2 \| h^{\ell}_{R, \eps} \|_\infty \| f_{R, \eps} \|_2 + C \| f_{R, \eps} \|_2 + C \| f_{R, \eps} \|_2^2 + C \sqrt{\eps}^2
\end{align}
for some small $\lambda > 0$ to be determined. We further choose small $\lambda > 0$ such that $c_0 - C \lambda^{3-\gamma_0} \ge \tfrac{1}{2} c_0 >0$, and then finish the proof of Lemma \ref{Lemma-L2-Estimates}.

\subsection{$L^\infty$ estimate for $h^{\ell}_{R, \eps}$: Proof of Lemma \ref{Lemma-L-infty-Estimate}}
As in \cite{Guo-2010-ARMA} or \cite{GHW-2020}, we introduce the following backward $b$-characteristics of Boltzmann equation. Given $(t, x, v)$, let $[X(s), V(s)]$ be determined by
\begin{equation*}
\left\{
\begin{aligned}
&\tfrac{\d }{\d s} X(s) = V(s)\,, \ \tfrac{\d }{\d s} V(s) = 0\,, \\
& [X(t), V(t)] = [x, v]\,.
\end{aligned}
\right.
\end{equation*}
The solution is then given by
\begin{align*}
[X(s), V(s)] = [X (s; t, x, v), V (s; t, x, v)] = [x - (t-s) v, v] \,.
\end{align*}
For any fixed $(x, v)$ with $x \in \overline{\R}^3_+$ and $v_3 \ne 0$, we define its backward exit time $t_b (x,v) \ge 0$ to be the last moment at which the back-time straight line $[X (s; 0, x, v), V (s; 0, x, v)]$ remains in $x \in \overline{\R}^3_+$. Hence, it holds that
\begin{align*}
t_b (x,v) = \sup \big\{ \tau_0 \ge 0: x - \tau_0 v \in \R^3_+ \big\}\,,
\end{align*}
which means $x - t_b v \in \p \R^3_+$, i.e., $x_3 - t_b v_3 = 0$. We also define
\begin{align*}
x_b (x ,v) = x (t_b) = x- t_b v \in \p \R^3_+ \,.
\end{align*}
In the half space, for the case $v_3 <0$, the back-time trajectory is a straight line and does not hit the boundary, see Figure \ref{tb-NOTexist}. For the case $v_3 >0$, the back-time cycle will hit the boundary for one time, see Figure \ref{tb-exist}.
\begin{figure}[h]
	\centering
	\begin{tikzpicture}[node distance=2cm]
		\draw[-, very thick] node[left,scale=1]{$\p \R^3_+$}(0,0)--(3,0);
		\draw[->, thick] (2,-0.1)--(2,-0.6) node[below]{$n$};
		\draw[-,thick, dashed] (1,0)--(2,1);
		\draw[->,thick] (2,1)--(2.7,1.7) node[right]{$v$};
		\draw (2,1) node[below right, scale=1]{$x$};
		\draw (3, 0.1) node[above left]{$t_b (x,v)$};
		\draw[->,dashed, color=red] (2,1)--(3.8,1) node[right, color=black]{time $t$};
		\draw[->,dashed, color=red] (1,0)--(3.8,0) node[right, color=black]{time $t_1 = t - t_b (x,v)$};
		\draw (1,0) node[below]{$x_b (x,v)$};
		\draw[->, thick, color=blue] (0, 0.5)--(1,0);
		\draw (-0.3,0.8) node[below]{\color{brown}$\alpha_\eps$\color{blue}$v'$};
		\draw[->, thick, color=red] (0, 0.9)--(1,0);
		\draw (0,1.5) node[below]{\color{brown}$(1 - \alpha_\eps)$\color{red}$R_{x_b}v$};
	\end{tikzpicture}
	\caption{The case $v \cdot n < 0$: $t_b (x,v)$ exists.}\label{tb-exist}
\end{figure}
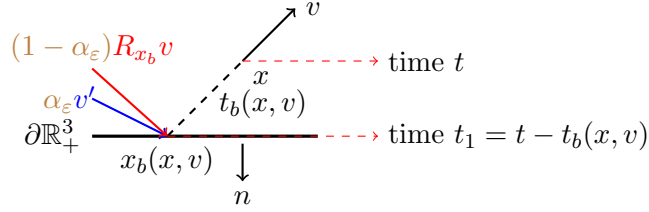
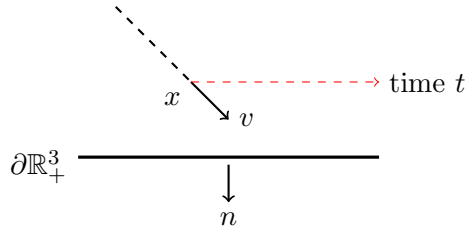
\begin{figure}[h]
	\centering
	\begin{tikzpicture}[node distance=2cm]
		\draw[-, very thick] node[above left,scale=1]{$\p \R^3_+$}(0,0.5)--(4,0.5);
		\draw[->, thick] (2,0.4)--(2,-0.1) node[below]{$n$};
		\draw[-,thick, dashed] (0.5,2.5)--(1.5,1.5);
		\draw[->,thick] (1.5,1.5)--(2,1) node[right]{$v$};
		\draw (1.5,1.5) node[below left, scale=1]{$x$};
		\draw[->,dashed, color=red] (1.5,1.5)--(4,1.5) node[right, color=black]{time $t$};
	\end{tikzpicture}
	\caption{The case $v \cdot n \geq 0$: $t_b (x,v)$ does not exist.}\label{tb-NOTexist}
\end{figure}
Now let $x \in \overline{\R}^3_+$, $(x ,v) \notin \Sigma_0 \cup \Sigma_-$ and $(t_0, x_0, v_0) = (t, x, v)$. By the Maxwell reflection boundary condition \eqref{MBC}, the back-time cycle is defined as
\begin{align}\label{t-V-X-Formula}
\no X_{cl} (s; t, x, v) = & \mathbbm{1}_{[t_1, t)} (s) \big\{ x - (t-s) v \big\} + \mathbbm{1}_{(- \infty, t_1)} (s) \big\{ x - [ (1 - \alpha_\eps)R_{x_b} v + \alpha_\eps v' ] (t-s) \big\} \,, \\
V_{cl} (s; t, x, v) =& \mathbbm{1}_{[t_1, t)} (s) v + \mathbbm{1}_{(- \infty, t_1)} (s) [ (1 - \alpha_\eps)R_{x_b} v + \alpha_\eps v' ] \,,
\end{align}
where $v^\prime \in \mathcal{V} := \{v^\prime: v^\prime \cdot n >0\}$.

According to the trajectories showed in Figure \ref{tb-exist} and \ref{tb-NOTexist}, the points set $[0, \tau] \times \R^3_+ \times \R^3$ can be clarified by following two parts:
\begin{equation}\label{UV}
	\begin{aligned}
		\mathscr{U} = & \{ (t,x,v) \in [0, \tau] \times \R^3_+ \times \R^3 : t \in [0, \tau], v \cdot n \geq 0 \} \\
		& \qquad \cup \{ (t,x,v) \in [0, \tau] \times \R^3_+ \times \R^3 : v \cdot n < 0 \textrm{ and } t_1 = t - t_b (x,v) \leq 0 \} \,, \\
		\mathscr{V} = & [0, \tau] \times \R^3_+ \times \R^3 - \mathscr{U} \\
		= & \{ (t,x,v) \in [0, \tau] \times \R^3_+ \times \R^3 : v \cdot n < 0 \textrm{ and } t_1 = t - t_b (x,v) > 0 \} \,.
	\end{aligned}
\end{equation}
Define a nonnegative number
\begin{equation}\label{t-star}
	t_\star =
	\left\{
	\begin{aligned}
		0 \,, \quad \textrm{if } (t,x,v) \in \mathscr{U} \,, \\
		t_1 \,, \quad \textrm{if } (t,x,v) \in \mathscr{V} \,.
	\end{aligned}
	\right.
\end{equation}
Observe that the trajectory over $s \in [t_\star, t]$ is
\begin{equation}\label{Trajectory}
	\begin{aligned}
		X_{cl} (s; t, x, v) = x - (t-s) v \,, \ V_{cl} (s; t, x, v) = v \,.
	\end{aligned}
\end{equation}

As in \cite{Caflish-1980-CPAM}, we define
\begin{align*}
\mathcal{L}_M g = - \tfrac{1}{\sqrt{\M_M}} \big\{ B(\M, \sqrt{\M_M} g) + B(\sqrt{\M_M} g, \M) \big\} = \big\{ \nu (\M) + K \big\} g \,,
\end{align*}
where $K g = K_1 g - K_2 g$ with
\begin{align*}
K_1 g = & \int_{\R^3 \times \mathbb{S}^2} b(\theta) | u -v |^{\gamma_0} \sqrt{\M_M (u)} \tfrac{\M (v)}{\sqrt{\M_M (v)}} g (u) \d \omega \d u \,, \\
K_2 g = & \int_{\R^3 \times \mathbb{S}^2} b(\theta) | u -v |^{\gamma_0} \M (u^\prime) \tfrac{\sqrt{\M_M (v^\prime)}}{\sqrt{\M_M (v)}} g(v^\prime) \d \omega \d u \\
& + \int_{\R^3 \times \mathbb{S}^2} b(\theta) | u -v |^{\gamma_0} \M (v^\prime) \tfrac{\sqrt{\M_M (u^\prime)}}{\sqrt{\M_M (v)}} g(u^\prime) \d \omega \d u \,.
\end{align*}
Consider a smooth cutoff function $ 0 \le \chi_m \le 1$ such that for $ m >0$, $\chi_m (s) \equiv 1$ for $s \le m$, and $\chi_m (s) \equiv 0$ for $s \ge 2m$. Then we introduce (see \cite{Guo-Jang-Jiang-2010-CPAM})
\begin{align*}
K^m g = & \int_{\R^3 \times \mathbb{S}^2} b(\theta) | u-v |^{\gamma_0} \chi_m ( |u-v| ) \sqrt{\M_M (u)} \tfrac{\M_ (v)}{\M_M (v)} g(u) \d \omega \d u \\
& - \int_{\R^3 \times \mathbb{S}^2} b(\theta) | u-v |^{\gamma_0} \chi_m ( |u-v| ) \M (u^\prime) \tfrac{\sqrt{\M_M (v^\prime)}}{\sqrt{\M_M (v)}} g(v^\prime) \d \omega \d u \\
& - \int_{\R^3 \times \mathbb{S}^2} b(\theta) | u-v |^{\gamma_0} \chi_m ( |u-v| ) \M (v^\prime) \tfrac{\sqrt{\M_M (u^\prime)}}{\sqrt{\M_M (v)}} g(u^\prime) \d \omega \d u \,,
\end{align*}
and further define $K^c g = K g - K^m g$.

\begin{lemma}[Lemma 2.3 of \cite{Guo-Jang-Jiang-2010-CPAM}]\label{Lemma-Bound-of-K}
	\begin{align}\label{K_m-Bound}
	| K^m g (v) | \le C m^{3 + \gamma_0} \nu (\M) \| g \|_\infty
	\end{align}
	and $K^c g(v) = \int_{\R^3} \Bbbk (v, v^\prime) g (v^\prime) \d v^\prime$, where the kernel $\Bbbk (v, v^\prime)$ satisfies
	\begin{align}\label{K_c-Kernal-Bound}
	\Bbbk (v, v^\prime)
	\le C_m \tfrac{\exp\big\{ - c |v-v^\prime|^2 \big\}}{|v-v^\prime| (1 + |v| + |v^\prime|)^{1-\gamma_0}}
	\end{align}
	for some $c > 0$.
\end{lemma}

\begin{proof}[Proof of Lemma \ref{Lemma-L-infty-Estimate}]
	The proof is similar to \cite{Guo-Jang-Jiang-2010-CPAM}. We will only elaborate on the proof ideas and give a sketch of the proof process. The details of calculation will be omitted for simplicity of presentation. We will especially focus on the details of the relationship between $\| h_{R,\eps}^\ell \|_\infty$ and $\| f_{R,\eps} \|_2$.
	
	We first define
	\begin{align*}
	K_{\ell} g \equiv \l v \r^{\ell} K \Big( \tfrac{g}{\l v \r^{\ell}} \Big) \,.
	\end{align*}
	From the remainder equation \eqref{Remainder-F} and the definition of $h^{\ell}_{R, \eps}$ given in \eqref{f_Reps-h_Reps}, we deduce that
	\begin{align}\label{h_R_eps}
	\no & \p_t h^{\ell}_{R, \eps} + v \cdot \nabla_x h^{\ell}_{R, \eps} + \tfrac{\nu (\M)}{\eps} h^{\ell}_{R, \eps} - \tfrac{1}{\eps} K_{\ell} h^{\ell}_{R, \eps} \\
	= & \tfrac{\sqrt{\eps}^2 \langle v \rangle^{\ell}}{\sqrt{\M}_M} B \Big( \tfrac{\sqrt{\M_M}}{\langle v \rangle^{\ell}} h^{\ell}_{R, \eps}, \tfrac{\sqrt{\M_M}}{\langle v \rangle^{\ell}} h^{\ell}_{R, \eps} \Big) + \tfrac{\langle v \rangle^\ell}{\sqrt{\M_M}} \big( R_\eps + R^b_\eps + R^{bb}_\eps \big) \\
	\no & + \sum_{i =k}^3 \sqrt{\eps}^{k-2} \tfrac{\langle v \rangle^{\ell}}{\sqrt{\M_M}} \Big\{ B \Big( F_k + F^b_k + F^{bb}_k, \tfrac{\sqrt{\M_M}}{\langle v \rangle^{\ell}} h^{\ell}_{R, \eps} \Big) + B \Big( \tfrac{\sqrt{\M_M}}{\langle v \rangle^{\ell}} h^{\ell}_{R, \eps}, F_k + F^b_k + F^{bb}_k \Big) \Big\}
	\end{align}
    with boundary condition on $\Sigma_-$
    \begin{equation}\label{IC-h}
    	\begin{aligned}
    		\gamma_- h^\ell_{R, \eps} = (1 - \alpha_\eps) h^\ell_{R, \eps} (t, x, R_x v) + \alpha_\eps M_w  \tfrac{\l v \r^\ell}{\sqrt{\M_M}} \int_{v^\prime \cdot n >0} (v^\prime \cdot n) \tfrac{\sqrt{\M_M}}{\l v^\prime \r^\ell} h^\ell_{R, \eps} (t, x, v^\prime) \d v^\prime \,.
    	\end{aligned}
    \end{equation}
	For any $(t, x, v)$, integrating \eqref{h_R_eps} along the backward trajectory \eqref{t-V-X-Formula} over $s \in [t_\star, t]$, one obtains that
	\begin{align}\label{h_R_eps-bound-0}
	\no & h^{\ell}_{R, \eps} (t, x, v) = \underbrace{\exp\big\{ - \tfrac{1}{\eps} \int_{t_\star}^t \nu (\phi) \d \phi \big\} h^{\ell}_{R, \eps} (t_\star, X_{cl} (t_\star), V_{cl} (t_\star) )}_{I_1} \\
	\no & + \underbrace{\int_{t_\star}^t \exp\big\{ - \tfrac{1}{\eps} \int_s^t \nu (\phi) \d \phi \big\} \big( \tfrac{1}{\eps} K_{\ell}^m h^{\ell}_{R, \eps} \big) (s, X_{cl} (s), V_{cl} (s)) \d s }_{I_2} \\
	\no & + \underbrace{ \int_{t_\star}^t \exp\big\{ - \tfrac{1}{\eps} \int_s^t \nu (\phi) \d \phi \big\} \big( \tfrac{1}{\eps} K_{\ell}^c h^{\ell}_{R, \eps} \big) (s, X_{cl} (s), V_{cl} (s)) \d s }_{I_3} \\
	\no & + \underbrace{ \int_{t_\star}^t \exp\big\{ - \tfrac{1}{\eps} \int_s^t \nu (\phi) \d \phi \big\} \Big( \sqrt{\eps} \tfrac{\langle v \rangle^{\ell}}{\sqrt{\M_M}} B \big( \tfrac{\sqrt{\M_M}}{\langle v \rangle^{\ell}} h^{\ell}_{R, \eps}, \tfrac{\sqrt{\M_M}}{\langle v \rangle^{\ell}} h^{\ell}_{R, \eps} \big) \Big) (s, X_{cl} (s), V_{cl} (s)) \d s }_{I_4} \\
	\no & \left.
	\begin{aligned}
	& + \sum_{k =1}^5 \sqrt{\eps}^{k-2} \int_{t_\star}^t \exp\big\{ - \tfrac{1}{\eps} \int_s^t \nu (\phi) \d \phi \big\} \Big\{ \tfrac{\langle v \rangle^{\ell}}{\sqrt{\M_M}} B \Big( F_k + F^b_k + F^{bb}_k, \tfrac{\sqrt{\M_M}}{\langle v \rangle^{\ell}} h^{\ell}_{R, \eps} \Big) \\
	& \qquad \qquad \qquad \quad + \tfrac{\langle v \rangle^{\ell}}{\sqrt{\M_M}} B \Big(  \tfrac{\sqrt{\M_M}}{\langle v \rangle^{\ell}} h^{\ell}_{R, \eps}, F_k + F^b_k + F^{bb}_k \Big) \Big\} (s, X_{cl} (s), V_{cl} (s)) \d s
	\end{aligned}
	\right\}{:= I_5} \\
	& + \underbrace{ \int_{t_\star}^t \exp\big\{ - \tfrac{1}{\eps} \int_s^t \nu (\phi) \d \phi \big\} \Big( \tfrac{\langle v \rangle^{\ell}}{\sqrt{\M_M}} \big( R_\eps + R^b_\eps + R^{bb}_\eps \big) \Big) \big(s, X_{cl} (s), V_{cl} (s)\big) \d s }_{I_6} \,,
	\end{align}
	where the simplified symbol has been employed:
	\begin{align*}
	\nu (\phi) = \nu (\M) \big( \phi, X_{cl} (\phi), V_{cl} (\phi) \big) \,.
	\end{align*}

    {\bf Case 1: $(t,x,v) \in \mathscr{U}$.}

    Notice that $t_\star = 0$. Then the term $I_1$ in \eqref{h_R_eps-bound-0} can be bounded by
	\begin{align}\label{h_R_eps-I1-bound}
	|I_1| \le C \eps \langle v \rangle^{-\gamma_0} \| h^{\ell}_{R, \eps} (0) \|_\infty \leq C \eps \| h^{\ell}_{R, \eps} (0) \|_\infty \,.
	\end{align}
	For small $m > 0$, one has
	\begin{align}\label{h_R_eps-I2-to-I6-bound}
	\no & |I_2| \le C m^{3+\gamma_0} \sup_{s \in [0, t]} \| h^{\ell}_{R, \eps} (s) \|_\infty \,, \ |I_4| \le C \sqrt{\eps}^3 \sup_{s \in [0, t]} \| h^{\ell}_{R, \eps} (s) \|_\infty^2 \,, \\
	& |I_5| \le C \sqrt{\eps} \sup_{s \in [0, t]} \| h^{\ell}_{R, \eps} (s) \|_\infty \,, \ |I_6| \le C \sqrt{\eps}^2 \,.
	\end{align}
	
	We then focus on the estimate of the term $I_3$ in \eqref{h_R_eps-bound-0}. Let $\Bbbk_{\ell} (v, v^\prime)$ be the corresponding kernel associated with $K^c_{\ell}$. From Lemma \ref{Lemma-Bound-of-K},
	\begin{align}\label{Bbbk-Bound}
	 | \Bbbk_\ell (v, v^\prime) | \le \tfrac{ C \langle v^\prime \rangle^{\ell} \exp\big\{ - c |v-v^\prime|^2 \big\} }{ |v-v^\prime| \langle v \rangle^{\ell} (1+ |v| + |v^\prime| )^{1-\gamma_0} } \le \tfrac{ C \exp\{ - \frac{3}{4} c |v-v^\prime|^2 \} }{ |v-v^\prime| (1+ |v| + |v^\prime| )^{1-\gamma_0} } \,.
	\end{align}
	We therefore bound $I_3$ in \eqref{h_R_eps-bound-0} by
	\begin{align*}
	|I_3| \le \tfrac{1}{\eps} \int_{t_\star}^t \exp\Big\{ - \tfrac{1}{\eps} \int_s^t \nu (\phi) \d \phi \Big\} \int_{\R^3} \big| \Bbbk_{\ell} (V_{cl} (s), v^\prime) h^{\ell}_{R, \eps} (s, X_{cl} (s), v^\prime) \big| \d v^\prime \d s \,.
	\end{align*}
	We now use \eqref{h_R_eps-bound-0} again to evaluate $h^{\ell}_{R, \eps}$ via replacing $(t, x, v)$ with $(s, X_{cl} (s), v^\prime)$. Together with \eqref{h_R_eps-I2-to-I6-bound}, the above can be further bounded by
	\begin{align}\label{h_R_eps-I3}
	\no & \left.
	\begin{aligned}
	|I_3| \le &\tfrac{1}{\eps} \int_0^t \int_{\R^3} \exp\Big\{ -\tfrac{1}{\eps} \int_s^t \nu (\phi) \d \phi - \tfrac{1}{\eps} \int_0^s \nu (v^\prime) (\phi) \d \phi \Big\} | \Bbbk_\ell (V_{cl} (s), v^\prime ) | \\
	& \qquad \qquad \qquad \qquad \qquad \times | h^{\ell}_{R, \eps} \big( 0, X_{cl} (0; s, X_{cl} (s), v^\prime), v^\prime \big) | \d v^\prime \d s
	\end{aligned}
	\right\}{:= I^1_3} \\
	\no & \left.
	\begin{aligned}
	& + \tfrac{1}{\eps^2} \int_0^t \exp\Big\{ -\tfrac{1}{\eps} \int_s^t \nu(\phi) \d \phi \Big\} \int_{\R^3} | \Bbbk_\ell (V_{cl} (s), v^\prime) | \int_0^s \exp\Big\{ -\tfrac{1}{\eps} \int_{s_1}^s \nu (\phi) \d \phi \Big\} \\
	& \qquad \qquad \times \big| K_{\ell}^m h^{\ell}_{R, \eps} \big( s_1, X_{cl} (s_1; s, X_{cl} (s), v^\prime), V_{cl} (s_1; s, X_{cl} (s), v^\prime) \big) \big| \d s_1 \d v^\prime \d s
	\end{aligned}
	\right\}{:= I^2_3} \\
	\no & \left.
	\begin{aligned}
	& + \tfrac{1}{\eps^2} \int_0^t \exp\Big\{ -\tfrac{1}{\eps} \int_s^t \nu (\phi) \d \phi \Big\} \int_{\R^3 \times \R^3} \big| \Bbbk_{\ell} \big( V_{cl} (s), v^\prime \big) \Bbbk_{\ell} (v^\prime, v^{\prime\prime}) \big| \\
	& \times \int_0^s \exp\Big\{ -\tfrac{1}{\eps} \int_{s_1}^s \nu (v^\prime) (\phi) \d \phi \Big\} \big| h^\ell_{R, \eps} \big( s_1, X_{cl} (s_1; s, X_{cl} (s), v^\prime), v^{\prime\prime} \big) \big| \d v^{\prime\prime} \d v^\prime \d s_1 \d s
	\end{aligned}
	\right\}{:= I^3_3} \\
	\no & + \underbrace{ \tfrac{C}{\eps} \int_0^t \exp\Big\{ -\tfrac{1}{\eps} \int_s^t \nu (\phi) \d \phi \Big\} \int_{\R^3} | \Bbbk_\ell (V_{cl} (s), v^\prime) | \d v^\prime \d s \cdot \sqrt{\eps}^3 \sup_{s \in [0, t]} \| h^{\ell}_{R, \eps} (s) \|_\infty^2 }_{I^4_3} \\
	\no & + \underbrace{ \tfrac{C}{\eps} \int_0^t \exp\Big\{ -\tfrac{1}{\eps} \int_s^t \nu (\phi) \d \phi \Big\} \int_{\R^3} | \Bbbk_\ell (V_{cl} (s), v^\prime) | \d v^\prime \d s \cdot \sqrt{\eps} \sup_{s \in [0, t]} \| h^{\ell}_{R, \eps} (s) \|_\infty }_{I^5_3} \\
	& + \underbrace{ \tfrac{C}{\eps} \int_0^t \exp\Big\{ -\tfrac{1}{\eps} \int_s^t \nu (\phi) \d \phi \Big\} \int_{\R^3} | \Bbbk_\ell (V_{cl} (s), v^\prime) | \d v^\prime \d s \cdot \sqrt{\eps}^2 }_{I^6_3} \,.
	\end{align}
	Except for the term $I_3^3$, the sum of all other terms can be bounded by
	\begin{align}\label{h_R_eps-I31-toI_36}
	\no I^1_3 + I^2_3 + I^4_3 + & I^5_3 + I^6_3 \\
	\le & C \big\{ \| h^{\ell}_{R, \eps} (0) \|_\infty + \sqrt{\eps}^3 \sup_{s \in [0,t]} \| h^{\ell}_{R, \eps} (s) \|_\infty^2 + \sqrt{\eps} \sup_{s \in [0,t]} \| h^{\ell}_{R, \eps} (s) \|_\infty + \sqrt{\eps}^2 \big\} \,.
	\end{align}
	It remains to control the quantity $I^3_3$ in \eqref{h_R_eps-I3}. As in \cite{Guo-Jang-Jiang-2010-CPAM}, we will finish our arguments by the following several cases. Let large $N_0 >0$ and small $\kappa_* > 0$ be undetermined.
	
	Case 1. If $|v| \ge N_0$,
	\begin{align}\label{I_3_3-Case1}
	  I^3_3 \le \tfrac{C}{N_0} \sup_{s \in [0,t]} \| h^{\ell}_{R, \eps} (s) \|_\infty \,.
	\end{align}
	
	Case 2. If either $|v| \le N_0$, $|v^\prime| \ge 2N_0$ or $|v^\prime| \le 2N_0$, $|v^{\prime \prime}| \ge 3N_0$,
	\begin{align}\label{I_3_3-Case2}
	I^3_3 \le C_\eta e^{- \frac{\eta}{8} N^2_0} \sup_{s \in [0,t]} \| h^{\ell}_{R, \eps} (s) \|_\infty
	\end{align}
	for some small $\eta > 0$.
	
	Case 3a. If $|v| \le N_0$, $|v^\prime| \le 2N_0$, $|v^{\prime \prime}| \le 3 N_0$, $s-s_1 \le \eps \kappa_*$,
	\begin{align}\label{I_3_3-Case3a-Bound}
	  I^3_3 \le C_{N_0} \kappa_* \sup_{s \in [0, t]} \| h^{\ell}_{R, \eps} (s) \|_\infty \,.
	\end{align}
	
	Case 3b. If $|v| \le N_0$, $|v^\prime| \le 2N_0$, $|v^{\prime \prime}| \le 3N_0$, $s - s_1 \ge \eps \kappa_*$, together with \eqref{nu-phi}, one has
	\begin{align*}
	I^3_3 \le & \tfrac{C}{\eps^2} \int_0^t \int_{ \big\{ |v^\prime| \le 2N_0, |v^{\prime \prime}| \le 3N_0 \big\} } \int_0^{s - \eps \kappa_*} \exp \Big\{ - \tfrac{C \langle v \rangle^{\gamma_0} (t-s)}{\eps} \Big\} \exp \Big\{ - \tfrac{C \langle v^\prime \rangle^{\gamma_0} (s - s_1)}{\eps} \Big\} \\
	& \times \big| \Bbbk_{\bar{\ell}} (V_{cl} (s), v^\prime) \Bbbk_{\ell} (v^\prime, v^{\prime \prime}) h^{\ell}_{R, \eps} \big(s_1; X_{cl} (s_1; s, X_{cl} (s), v^\prime), v^{\prime\prime} \big) \big| \d s_1 \d v^\prime \d v^{\prime \prime} \d s \,.
	\end{align*}
	We remark that the relation between $\| h_{R,\eps}^\ell \|_\infty$ and $\| f_{R,\eps} \|_2$ comes from this case. By \eqref{Bbbk-Bound}, $\Bbbk_{\ell} (V_{cl} (s), v^\prime)$ has a possible integrable singularity of $\frac{1}{|V_{cl} (s) - v^\prime|}$. We can choose $\Bbbk_{N_0} (V_{cl} (s), v^\prime)$ smooth with compact support such that
	\begin{align}\label{Bbbk_N-Bbbk_l}
	\sup_{|p| \le 3 N_0} \int_{|v^\prime| \le 3N_0} \big| \Bbbk_{N_0} (p, v^\prime) - \Bbbk_{\ell} (p, v^\prime) \big| \d v^\prime \le \tfrac{1}{N_0} \,.
	\end{align}
	Splitting
	\begin{align*}
	\Bbbk_{\ell} (V_{cl} (s), v^\prime) & \Bbbk_{\ell} (v^\prime, v^{\prime \prime}) = \big[ \Bbbk_{\ell} (V_{cl} (s), v^\prime) - \Bbbk_{N_0} (V_{cl} (s), v^\prime) \big] \Bbbk_{\ell} (v^\prime, v^{\prime \prime}) \\
	&  + \big[ \Bbbk_{\ell} (v^\prime, v^{\prime \prime}) - \Bbbk_{N_0} (v^\prime, v^{\prime \prime})   \big] \Bbbk_{N_0} (V_{cl} (s), v^\prime) + \Bbbk_{N_0} (V_{cl} (s), v^\prime) \Bbbk_{N_0} (v^\prime, v^{\prime \prime}) \,,
	\end{align*}
	we derive from \eqref{Bbbk_N-Bbbk_l} that
	\begin{align}\label{I_3_3-Case3b-Bound-0}
	 \no I^3_3 \le & \tfrac{C}{N_0} \sup_{s \in [0,t]} \| h^{\ell}_{R,\eps} (s)\|_\infty \\
	 & + \tfrac{C}{\eps^2} \int_0^t \int_{ \{ |v^\prime| \le 2N_0, |v^{\prime \prime}| \le 3N_0 \} } \int_0^{s- \kappa_* \eps} \exp \big\{ - \tfrac{C \langle v \rangle^{\gamma_0} (t-s)}{\eps} \big\} \exp \big\{ -\tfrac{C \langle v^\prime \rangle^{\gamma_0} (s-s_1)}{\eps} \big\} \\
	\no & \times \big| \Bbbk_{N_0} (V_{cl} (s), v^\prime) \Bbbk_{N_0} (v^\prime, v^{\prime \prime}) h^{\ell}_{R, \eps} \big( s_1, X_{cl} (s_1; s, X_{cl} (s), v^\prime), v^{\prime \prime} \big) \big| \d s_1 \d v^\prime \d v^{\prime \prime} \d s \,,
	\end{align}
	where the following bound has been used:
	\begin{equation*}
	  \sup_{|v^\prime| \le 2N_0} \int_{\{ |v^{\prime \prime}| \le 3N_0\}} | \Bbbk_{\ell} (v^\prime, v^{\prime \prime}) \d v^{\prime \prime} + \sup_{ |V_{cl}(s)| = |v| \le 2N_0} \int_{\{ |v^\prime| \le 2N_0 \}} | \Bbbk_{N_0} (V_{cl}(s), v^\prime) | \d v^\prime \leq C \,.
	\end{equation*}
	For the integration over $v^\prime$ in \eqref{I_3_3-Case3b-Bound-0}, we make a change of variable $v^\prime \to y := X_{cl} \big(s_1; s, X_{cl} (s), v^\prime \big)$. From the explicit formula \eqref{t-V-X-Formula}, one has $\frac{\p y}{\p v^\prime}
	= - (s - s_1) \mathrm{Diag} (1,1,1)$ for $s_1 \in [0, s - \eps \kappa_*]$. Therefore, it holds that
	\begin{align*}
	\Big| \det \Big( \frac{\p y}{\p v^\prime} \Big) (s_1) \Big|
	= (s- s_1)^s \ge (\eps \kappa_*)^3 >0 \ {\text{for}}\ s_1 \in [0, s - \kappa_* \eps] \,,
	\end{align*}
	which yields that for any $|v''| \le 3 N_0$
	\begin{align*}
	& \int_{|v^\prime| \le 2N_0} \big| h^{\ell}_{R, \eps} \big( s_1, X_{cl} (s_1; s, X_{cl} (s), v^\prime), v^{\prime \prime} \big) \big| \d v^\prime \\
	\le & C_{N_0} \Big( \int_{|v^\prime| \le 2N_0} \mathbbm{1}_{\R^3_+} \big( X_{cl} (s_1; s, X_{cl} (s), v^\prime) \big) \big| h^{\ell}_{R, \eps} \big( s_1, X_{cl} (s_1; s, X_{cl} (s), v^\prime) \big) \big|^2 \d v^\prime \Big)^{\frac{1}{2}}\\
	\le & \tfrac{C_{N_0}}{(\kappa_* \eps)^{\frac{3}{2}}} \Big( \int_{|y| \le 3 N_0 (s-s_1)} \big| h^{\ell}_{R, \eps} \big( s_1, y, v^{\prime \prime} ) \big|^2 \d y \Big)^{\frac{1}{2}} \\
	\le & \tfrac{C_{N_0} \big( (s-s_1)^{\frac{3}{2}} + 1 \big)}{ (\kappa_* \eps)^{\frac{3}{2}} } \Big( \int_{\R^3_+} \big| h^{\ell}_{R, \eps} \big( s_1, y, v^{\prime \prime} ) \big|^2 \d y \Big)^{\frac{1}{2}} \,.
	\end{align*}
	Together with the definition of $f_{R, \eps}$ and $h^{\ell}_{R, \eps}$ in \eqref{f_Reps-h_Reps}, the last term in \eqref{I_3_3-Case3b-Bound-0} can be bounded by $\tfrac{ C_{N_0, \kappa_*} }{ \eps^{\frac{3}{2}} } \sup_{s\in [0, t]} \| f_{R, \eps} (s) \|_2$. One thereby has
	\begin{align}\label{I_3_3-Case3b-Bound}
	I^3_3 \le \tfrac{C}{N_0} \sup_{s \in [0,t]} \| h^{\ell}_{R, \eps} (s) \|_\infty + \tfrac{ C_{N_0, \kappa_*} }{ \eps^{\frac{3}{2}} } \sup_{s\in [0, t]} \| f_{R, \eps} (s) \|_2 \,.
	\end{align}
	Collecting the above all estimates implies that
	\begin{equation}\label{CaseU}
		\begin{aligned}
			\| h^\ell_{R, \eps} \mathbbm{1}_{\mathscr{U}} (t) \|_\infty \leq C \eps \| h^\ell_{R, \eps} (0) \|_\infty + \aleph (h^\ell_{R, \eps}, f_{R, \eps}) \,.
		\end{aligned}
	\end{equation}
    where
    \begin{equation}
    	\begin{aligned}
    		\aleph (h^\ell_{R, \eps}, f_{R, \eps}) = & C (m^{3 + \gamma_0} + \sqrt{\eps} + \tfrac{1}{N_0} + \kappa_*) \sup_{s \in [0, t]} \| h^\ell_{R, \eps} (s) \|_\infty \\
    		& + C \sqrt{\eps}^3 \sup_{s \in [0, t]} \| h^\ell_{R, \eps} (s) \|_\infty^2 + \tfrac{C_{N_0. \kappa_*}}{\eps^\frac{3}{2}} \sup_{s \in [0, t]} \| f_{R, \eps} (s) \|_2 + C \sqrt{\eps}^2 \,.
    	\end{aligned}
    \end{equation}
	
	{\bf Case 2: $(t,x,v) \in \mathscr{V}$.}
	
	Note that $t_\star = t_1 > 0$, and $(X_{cl} (t_1), V_{cl} (t_1)) = (x_b (x,v), v) \in \Sigma_-$. Then, denoting by $x_b = x_b (x,v)$,
	\begin{equation*}
		\begin{aligned}
			I_1 = \exp\big\{ - \tfrac{1}{\eps} \int_{t_1}^t \nu (\phi) \d \phi \big\} h^{\ell}_{R, \eps} (t_1, x_b, v ) \,.
		\end{aligned}
	\end{equation*}
	Recalling the boundary condition \eqref{IC-h}, one has
	\begin{equation*}
		\begin{aligned}
			h^\ell_{R, \eps} (t_1, x_b, v ) = (1 - \alpha_\eps) h^\ell_{R, \eps} (t_1, x_b, R_{x_b} v) + \alpha_\eps M_w  \tfrac{\l v \r^\ell}{\sqrt{\M_M}} \int_{v^\prime \cdot n >0} (v^\prime \cdot n) \tfrac{\sqrt{\M_M}}{\l v^\prime \r^\ell} h^\ell_{R, \eps} (t_1, x_b, v^\prime) \d v^\prime \,.
		\end{aligned}
	\end{equation*}
	Thanks to \eqref{T_M} and $M_w / \M^0 = \frac{1}{\rho^0} \sqrt{\frac{2 \pi}{T^0}}$, one has $0 < M_w  \tfrac{\l v \r^\ell}{\sqrt{\M_M}} \int_{v^\prime \cdot n >0} (v^\prime \cdot n) \tfrac{\sqrt{\M_M}}{\l v^\prime \r^\ell} \d v^\prime \leq C$. Since $(t_1, x_b, R_{x_b} v), (t_1, x_b, v^\prime) \in \mathscr{U}$,
	\begin{equation*}
		\begin{aligned}
			|I_1| \leq (1 - \alpha_\eps) | h^\ell_{R, \eps} (t_1, x_b, R_{x_b} v) | + C \alpha_\eps \sup_{v' \in \mathcal{V}} | h^\ell_{R, \eps} (t_1, x_b, v^\prime) | \leq 2 \| h^\ell_{R, \eps} \mathbbm{1}_{\mathscr{U}} (t_1) \|_\infty
		\end{aligned}
	\end{equation*}
    for sufficiently small $\eps > 0$. Via employing the similar arguments in Case 1: $(t,x,v) \in \mathscr{U}$, the other terms except for $I_1$ can be bounded by
    \begin{equation*}
    	\begin{aligned}
    		I_2 + I_3 + I_4 + I_5 + I_6 \leq & 2 \| h^\ell_{R, \eps} \mathbbm{1}_{\mathscr{U}} (t_1) \|_\infty + \aleph (h^\ell_{R, \eps}, f_{R, \eps}) \,.
    	\end{aligned}
    \end{equation*}
    It thereby infers that
    \begin{equation}\label{CaseV}
    	\begin{aligned}
    		\| h^\ell_{R, \eps} \mathbbm{1}_{\mathscr{V}} (t) \|_\infty \leq 4 \| h^\ell_{R, \eps} \mathbbm{1}_{\mathscr{U}} (t_1) \|_\infty \aleph (h^\ell_{R, \eps}, f_{R, \eps}) \,.
    	\end{aligned}
    \end{equation}
	Note that $\mathscr{U} \cup \mathscr{V} = [0, \tau] \times \R^3_+ \times \R^3$. By choosing sufficiently small $\eps > 0$, $m > 0$, $\kappa_* > 0$ and large $N_0 > 0$, \eqref{CaseU} and \eqref{CaseV} conclude Lemma \ref{Lemma-L-infty-Estimate}.
\end{proof}

\section{Boundary conditions for equations of fluid variables: Proof of Lemma \ref{Lmm-Robin-BC}} \label{Sec_Robin-BC}
	
In this section, we will formally derive the boundary conditions for the linear hyperbolic system \eqref{Linear-Hyperbolic-Syst} and linear compressible Prandtl-type equations \eqref{Linear-Prandtl}, which, respectively, are slip boundary conditions and Robin-type boundary conditions, hence, prove Lemma \ref{Lmm-Robin-BC}.

\begin{proof}[Proof of Lemma \ref{Lmm-Robin-BC}]
	We will prove our conclusion from the solvability condition \eqref{Solva_Cond}, i.e.,
	\begin{equation*}
	\int_{\R^3} v_3 \mathbbm{f}_k (t, \bar{x}, v)
	\left(
	\begin{array}{c}
	1 \\
	(\bar{v} - \bar{\u}^0)\\
	|v- \u^0|^2
	\end{array}
	\right)
	\sqrt{\M^0} \d v=0\,,
	\end{equation*}
	where $\mathbbm{f}_k (t, \bar{x}, v)$ is given in \eqref{fk-KBL}. We split
	\begin{equation*}
	  \begin{aligned}
	    \mathbbm{f}_k (t, \bar{x}, v) = \hat{g}_k (t, \bar{x}, v) + \mathbbm{g}_k (t, \bar{x}, v) \,,
	  \end{aligned}
	\end{equation*}
	where, for $v_3 < 0$,
	\begin{equation*}
	  \begin{aligned}
	    \hat{g}_k (t, \bar{x}, v) = & (f_k + f^b_k + f^{bb}_{k,1})(t, \bar{x}, 0, \bar{v}, v_3) - (f_k + f^b_k + f^{bb}_{k,1}) (t, \bar{x}, 0, \bar{v}, -v_3) \,, \\
	    \mathbbm{g}_k (t, \bar{x}, v) = & \sqrt{2 \pi} \big\{ \l \gamma_+ (f_{k-1} + f^b_{k-1} + f^{bb}_{k-1}) \r_{\p \R^3_+} \sqrt{\M^0} - (f_{k-1} + f^b_{k-1} + f^{bb}_{k-1}) (t, \bar{x}, 0, \bar{v}, v_3) \big\} \,,
	  \end{aligned}
	\end{equation*}
	ans they both vanish for $v_3 > 0$. We remark that $\mathbbm{g}_k (t, \bar{x}, v) \equiv 0$ in the case of specular reflection boundary condition, see \cite{GHW-2020}.
	
	First,
	\begin{equation*}
	  \begin{aligned}
	    0 = & \int_{\R^3} v_3 \mathbbm{f}_k (t, \bar{x}, v) \sqrt{\M^0} \d v = \underbrace{ \int_{\R^3} v_3 \hat{g}_k (t, \bar{x}, v) \sqrt{\M^0} \d v }_{Q_1} + \underbrace{ \int_{\R^3} v_3 \mathbbm{g}_k (t, \bar{x}, v) \sqrt{\M^0} \d v }_{Q_2} \,.
	  \end{aligned}
	\end{equation*}
	Following Section 2.4 of \cite{GHW-2020},
	\begin{equation}
	  \begin{aligned}
	    Q_1 = & \rho^0 \big[ u_{k,3}^0 + u_{k,3}^{b,0} + T^0 ( \Psi_k^0 + 5 T^0 \Theta_k^0 ) \big] \\
	    = & \rho^0 \Big\{ u_{k,3}^0 + \int_0^{+ \infty} \tfrac{1}{\rho^0} \big[ \p_t \rho^b_{k-1} + \div_{\bar{x}} \big( \rho^0 \bar{u}^b_{k-1} + \rho^b_{k-1} \bar{\u}^0 \big)\big] \d \zeta + T^0 ( \Psi_k^0 + 5 T^0 \Theta_k^0 ) \Big\} \,.
	  \end{aligned}
	\end{equation}
	Since $M_w = \M^0$ and $\sqrt{2 \pi} \int_{\R^2} \int_{- \infty}^0 v_3 M_w (v) \d \bar{v} \d v_3 = - \rho^0 \sqrt{T^0}$, we derive from the direct calculation that
	\begin{equation}
	  \begin{aligned}
	    Q_2 = & \Big( \sqrt{2 \pi} \int_{\R^2} \int_{- \infty}^0 v_3 M_w (v) \d \bar{v} \d v_3 - 1 \Big) \\
	    & \qquad \qquad \times \sqrt{2 \pi} \int_{\R^2} \int_{-\infty}^0 v_3 (f_{k-1} + f^b_{k-1} + f^{bb}_{k-1}) (t, \bar{x}, 0, \bar{v}, v_3) \sqrt{\M^0} \d \bar{v} \d v_3 \\
	    = & - (\rho^0 \sqrt{T^0} + 1) \sqrt{2 \pi} \int_{\R^2} \int_{-\infty}^0 v_3 (f_{k-1} + f^b_{k-1} + f^{bb}_{k-1}) (t, \bar{x}, 0, \bar{v}, v_3) \sqrt{\M^0} \d \bar{v} \d v_3 \,.
	  \end{aligned}
	\end{equation}
	Consequently, the above three relations imply \eqref{BC-u_k3}, and \eqref{BC-u_13} will automatically hold by letting $k=1$.

Second, for $i=1,2$, one has
\begin{equation}\label{Y1Y2}
  \begin{aligned}
    0 = & \int_{\R^3} (v_i - \u^0_i) v_3 \mathbbm{f}_k (t, \bar{x}, v) \sqrt{\M^0} \d v \\
    = & \int_{\R^3} (v_i - \u^0_i) v_3 \hat{g}_k (t, \bar{x}, v) \sqrt{\M^0} \d v + \int_{\R^3} (v_i - \u^0_i) v_3 \mathbbm{g}_k (t, \bar{x}, v) \sqrt{\M^0} \d v = : Y_1 + Y_2 \,.
  \end{aligned}
\end{equation}
As shown in \cite{GHW-2020}, together with \eqref{BC-u_13}, one has
\begin{equation}\label{Y1}
  \begin{aligned}
    Y_1 = & \big[ - \mu (T^0) \p_\zeta u^b_{k-1,i} + \rho^0 \sqrt{T^0} (1 + \rho^0 \sqrt{T^0}) u^b_{k-1, i} \big] (t, \bar{x}, 0) + \rho^0 \big[ \big(u_{1,i}^0 + u^{b, 0}_{1,i} \big) u^{b, 0}_{k-1,3} \big] \\
    & + T^0 \l \A^0_{3i}, J^b_{k-2} + (\mathcal{I} - \mathcal{P}^0) f_k \r (t, \bar{x}, 0) + \rho^0 (T^0)^2 \big( \delta_{i1} \Phi_{k,1}^0 + \delta_{i2} \Phi_{k,2}^0 \big) \,.
  \end{aligned}
\end{equation}
Since $M_w = \M^0$ and $\sqrt{2 \pi} \int_{\R^2} \int_{- \infty}^0 (v_i - \u^0_i) v_3 M_w (v) \d \bar{v} \d v_3 = 0$ for $i = 1,2$, one has
\begin{equation*}
  \begin{aligned}
    Y_2 = & 2 \pi  \int_{\R^2} \int_{- \infty}^0 (v_i - \u^0_i) v_3 M_w (v) \d \bar{v} \d v_3 \int_{\R^2} \int_{- \infty}^0 (- v_3) (f_{k-1} + f^b_{k-1} + f^{bb}_{k-1}) \sqrt{\M^0} \d \bar{v} \d v_3 \\
    & - \sqrt{2 \pi} \int_{\R^2} \int_{- \infty}^0 (v_i - \u^0_i) v_3 (f_{k-1} + f^b_{k-1} + f^{bb}_{k-1}) \sqrt{\M^0} \d \bar{v} \d v_3 \\
    = & - \sqrt{2 \pi} \int_{\R^2} \int_{- \infty}^0 (v_i - \u^0_i) v_3 (f_{k-1} + f^b_{k-1} + f^{bb}_{k-1}) \sqrt{\M^0} \d \bar{v} \d v_3 \,.
  \end{aligned}
\end{equation*}
Noticing that for $i = 1, 2$,
\begin{equation*}
  \begin{aligned}
    \int_{\R^2} \int_{-\infty}^0 (v_i - \u^0_i) v_3 \P^0 (f_{k-1} + f_{k-1}^b) \sqrt{\M^0} \d \bar{v} \d v_3 = - \tfrac{\rho^0 \sqrt{T^0}}{\sqrt{2 \pi}} ( u_{k-1, i} + u_{k-1, i}^b ) \,,
  \end{aligned}
\end{equation*}
the $Y_2$ can be further computed as
\begin{equation}\label{Y2}
  \begin{aligned}
    Y_2 = & \rho^0 \sqrt{T^0} (u_{k-1, i}^0 + u_{k-1, i}^{b,0}) \\
    & - \sqrt{2 \pi} \int_{\R^2} \int_{- \infty}^0 (v_i - \u^0_i) v_3 [ (\I - \P^0)(f_{k-1} + f^b_{k-1}) + f^{bb}_{k-1}] \sqrt{\M^0} \d \bar{v} \d v_3 \,.
  \end{aligned}
\end{equation}
Then, plugging \eqref{Y1} and \eqref{Y2} into \eqref{Y1Y2} implies the first boundary condition of \eqref{Boudary_Equa}.

At the end,
\begin{equation}\label{Z1Z2}
  \begin{aligned}
    0 = & \int_{\R^3} |v-\u^0|^2 v_3 \mathbbm{f}_k (t, \bar{x}, v) \sqrt{\M^0} \d v \\
    = & \underbrace{ \int_{\R^3} |v-\u^0|^2 v_3 \hat{g}_k (t, \bar{x}, v) \sqrt{\M^0} \d v }_{Z_1} + \underbrace{ \int_{\R^3} |v-\u^0|^2 v_3 \mathbbm{g}_k (t, \bar{x}, v) \sqrt{\M^0} \d v }_{Z_2}
  \end{aligned}
\end{equation}
As computed in Section 2.4 of \cite{GHW-2020},
\begin{equation}\label{Z1}
  \begin{aligned}
    Z_1 = & - \kappa(T^0) \p_\zeta \theta^b_{k-1} (t, \bar{x}, 0) + 2 (T^0)^{\frac{3}{2}} \l \B^0_3, (\mathcal{I}-\mathcal{P}_0) f_k + J^b_{k-2} \r (t, \bar{x}, 0) \\
    & + \tfrac{5}{3} \rho^0  [(\theta_1 + \theta^b_1) u^b_{k-1,3}] (t, \bar{x}, 0) + 10 \rho^0 (T^0)^3 \Theta_k (t, \bar{x}, 0) \,.
  \end{aligned}
\end{equation}
Since $M_w = \M^0$ and
\begin{equation*}
  \begin{aligned}
    \int_{\R^2} \int_{-\infty}^0 v_3 |v - \u^0|^2 M_w (v) \d \bar{v} \d v_3 = - \tfrac{4 \rho^0 (T^0)^\frac{3}{2}}{\sqrt{2 \pi}} \,,
  \end{aligned}
\end{equation*}
one has
\begin{equation*}
  \begin{aligned}
    Z_2 = & 2 \pi \int_{\R^2} \int_{-\infty}^0 v_3 |v - \u^0|^2 M_w (v) \d \bar{v} \d v_3 \int_{\R^2} \int_{-\infty}^0 (- v_3) (f_{k-1} + f_{k-1}^b + f_{k-1}^{bb}) \sqrt{\M^0} \d \bar{v} \d v_3 \\
    & - \sqrt{2 \pi} \int_{\R^2} \int_{-\infty}^0 v_3 |v - \u^0|^2 (f_{k-1} + f_{k-1}^b + f_{k-1}^{bb}) \sqrt{\M^0} \d \bar{v} \d v_3 \\
    = & - \sqrt{2 \pi} \int_{\R^2} \int_{-\infty}^0 v_3 ( |v - \u^0|^2 - 4 \rho^0 (T^0)^\frac{3}{2} ) (f_{k-1} + f_{k-1}^b + f_{k-1}^{bb}) \sqrt{\M^0} \d \bar{v} \d v_3 \,.
  \end{aligned}
\end{equation*}
A direct calculation implies
\begin{equation*}
  \begin{aligned}
    & \int_{\R^2} \int_{-\infty}^0 v_3 ( |v - \u^0|^2 - 4 \rho^0 (T^0)^\frac{3}{2} ) \P^0  (f_{k-1} + f_{k-1}^b) \sqrt{\M^0} \d \bar{v} \d v_3 \\
    = & - \tfrac{\rho^0 \sqrt{T^0}}{\sqrt{2 \pi}} \big( 2 \rho^0 \sqrt{T^0} + \tfrac{\sqrt{2 \pi}}{3} \rho^0 + \tfrac{2}{3} \big) \theta^b_{k-1} + \rho^0 \big( \tfrac{1}{2} + \tfrac{3}{2} T^0 - 2 \rho^0 (T^0)^\frac{3}{2} \big) (u_{k-1, 3} + u^b_{k-1, 3}) \\
    & - \tfrac{\rho^0 \sqrt{T^0}}{\sqrt{2 \pi}} \big( \tfrac{2}{3} \rho^0 \sqrt{T^0} + \tfrac{\sqrt{2 \pi}}{3} \rho^0 + 2 \big) \theta_{k-1} + \tfrac{4 \sqrt{T^0}}{\sqrt{2 \pi}} (\rho^0 \sqrt{T^0} - 1) ( T^0 \rho_{k-1} + p^b_{k-1} ) \,.
  \end{aligned}
\end{equation*}
One therefore obtains
  \begin{align}\label{Z2}
    \no Z_2 = & \rho^0 \sqrt{T^0} \big( 2 \rho^0 \sqrt{T^0} + \tfrac{\sqrt{2 \pi}}{3} \rho^0 + \tfrac{2}{3} \big) \theta^b_{k-1} - 4 \sqrt{T^0} (\rho^0 \sqrt{T^0} - 1) ( T^0 \rho_{k-1} + p^b_{k-1} ) \\
    \no & + \rho^0 \sqrt{T^0} \big( \tfrac{2}{3} \rho^0 \sqrt{T^0} + \tfrac{\sqrt{2 \pi}}{3} \rho^0 + 2 \big) \theta_{k-1} \\
    & - \sqrt{2 \pi} \rho^0 \big( \tfrac{1}{2} + \tfrac{3}{2} T^0 - 2 \rho^0 (T^0)^\frac{3}{2} \big) (u_{k-1, 3} + u^b_{k-1, 3}) \\
    \no & - \sqrt{2 \pi} \int_{\R^2} \int_{-\infty}^0 v_3 ( |v - \u^0|^2 - 4 \rho^0 (T^0)^\frac{3}{2} ) [ (\I - \P^0) (f_{k-1} + f_{k-1}^b) + f_{k-1}^{bb}] \sqrt{\M^0} \d \bar{v} \d v_3 \,.
  \end{align}
Consequently, the second condition of \eqref{Boudary_Equa} is derived from using \eqref{BC-u_k3} and plugging \eqref{Z1}-\eqref{Z2} into \eqref{Z1Z2}, and then the proof of Lemma \ref{Lmm-Robin-BC} is finished.
\end{proof}

\section{Existence for the linear compressible Prandtl-type system \eqref{u-theta_Eq}: Proof of Lemma \ref{Prop-Prandtl}} \label{Sec_Prandtl}

In this section, we focus on the existence of the smooth solutions to the linear compressible Prandtl-type equations \eqref{u-theta_Eq} with Robin-type boundary conditions \eqref{u-theta_BC}, which derives from the Maxwell reflection boundary condition \eqref{MBC}. We remark that if the \eqref{MBC} is replaced by the specular reflection boundary condition, one can obtain Neumann boundary condition, see \cite{GHW-2020}.

Observe that the linear compressible Prandtl-type system \eqref{u-theta_Eq} is a degenerated parabolic system with nontrivial Robin-type boundary values \eqref{u-theta_BC}. First, we can introduce some explicit functions to zeroize the inhomogeneous boundary values (see \eqref{theta_BC}), which actually converts the boundary values into the form of external force term. We therefore obtain the system \eqref{Theta_Equa} with zero boundary values. We then construct a linear parabolic approximate system \eqref{Theta_Appro} while proving the existence of \eqref{Theta_Equa}. When deriving the uniform bounds of the approximate system in the space $L^\infty(0, \tau; \mathbb{H}^k_l(\R^3_+))$, we will employ the structures of equations to convert the higher normal $\zeta$-derivatives to the values of tangential $\bar{x}$-derivatives and time derivatives on $\{ \zeta = 0 \}$, see Lemma \ref{Lemma_Theta_BC_infty} below.

We first quote the following results to control some integral values on the boundary $\p \R^3_+$ by that in the interior $\R^3_+$. More precisely,
\begin{lemma}[Lemma 8.1 of \cite{GHW-2020}]\label{Lemma_Trace}
	Let $\Omega_b : = \{ (\bar{x}, x_3) : \bar{x} \in \R^2, x_3 \in [0, b) \}$ with $1 \leq b \leq \infty$. We assume $f, g \in H^1 (\Omega_b)$. There holds, for any $x_3 \in [0, b)$, that	
	\begin{equation}\label{Trace-1}
	\begin{aligned}
	\Big| \int_{\R^2} (f g) (\bar{x}, x_3) \d \bar{x} \Big| \leq \| \p_{x_3} (f, g) \|_{L^2 (\Omega_b)} \| (f, g) \|_{L^2(\Omega_b)} + \tfrac{1}{b} \| f \|_{L^2(\Omega_b)} \| g \|_{L^2(\Omega_b)} \,.
	\end{aligned}
	\end{equation}
	For $i = 1, 2$, there holds
	\begin{equation}\label{Trace-2}
	\begin{aligned}
	\Big| \int_{\R^2} (\p_{x_i} f  g) (\bar{x}, x_3) \d \bar{x} \Big| \leq \| \p_{x_3} (f, g) \|_{L^2 (\Omega_b)} & \| \p_{x_i} (f, g) \|_{L^2(\Omega_b)} + \tfrac{1}{b} \| \p_{x_i} f \|_{L^2(\Omega_b)} \| g \|_{L^2(\Omega_b)} \,.
	\end{aligned}
	\end{equation}
\end{lemma}
We remark that if $b = \infty$, $\Omega_\infty = \R^3_+$ and the $\frac{1}{b}$-terms in \eqref{Trace-1}-\eqref{Trace-2} will be automatically vanished. We will apply Lemma \ref{Lemma_Trace} by letting $b = \frac{3}{\sigma}$, which means $0 < \tfrac{1}{b} \leq \tfrac{1}{3}$ for $0 < \sigma \leq 1$ below.

\begin{proof}[Proof of Lemma \ref{Prop-Prandtl}]
We consider the system \eqref{u-theta_Eq}-\eqref{u-theta_IC}, namely,
\begin{equation}\label{theta_Equa}
 \left\{
 \begin{aligned}
  & \rho^0 ( \partial_t + \bar{\u}^0 \cdot \nabla_{\bar{x}} ) u_i + \rho^0 ( \p_{x_3} \u^0_3 \zeta + u^0_{1,3} ) \partial_\zeta u_i + \rho^0 u \cdot \nabla_{\bar{x}} \u^0_i + \tfrac{\partial_{x_i} p^0}{3 T^0} \theta = \mu (T^0) \partial_\zeta^2 u_i + f_i \,, \\
  & \rho^0 ( \p_t + \bar{\u}^0 \cdot \nabla_{\bar{x}} ) \theta + \rho^0 ( \p_{x_3} \u^0_3 \zeta + u^0_{1,3} ) \p_\zeta \theta + \tfrac{2}{3} \rho^0 \div_x \u^0 \theta = \tfrac{3}{5} \kappa(T^0) \p_{\zeta}^2 \theta + g\,,\\
  & (\p_\zeta \theta - R_\theta \theta)|_{\zeta =0} = a(t, \bar{x}) \,, \ \big( \partial_\zeta u_i - R_u u_i \big) \big|_{\zeta = 0} = b_i (t, \bar{x}) \,, \ i = 1,2 \,, \\
  & \lim_{\zeta \to \infty} ( u, \theta ) (t, \bar{x} , \zeta) = 0 \,,  \\
  & u (t, \bar{x}, \zeta) |_{t=0} = u_0 (\bar{x}, \zeta) \in \R^2 \,, \ \theta (t, \bar{x}, \zeta) |_{t=0} = \theta_0 (\bar{x}, \zeta) \,.
 \end{aligned}
 \right.
\end{equation}
We introduce the following functions $u_b (t, \bar{x}, \zeta) = (u_{b1}, u_{b2}) (t, \bar{x}, \zeta)$ and $\theta_a (t, \bar{x}, \zeta)$ as
\begin{equation*}
  \begin{aligned}
    u_{bi} (t, \bar{x}, \zeta) = ( \tfrac{1}{R_u} + 2 \zeta ) b_i (t, \bar{x}) \chi (\zeta) \,, \ i =1,2 \,, \ \theta_a (t, \bar{x}, \zeta) = \big( \tfrac{1}{R_\theta} + 2 \zeta \big) a(t,\bar{x}) \chi(\zeta)  \,,
  \end{aligned}
\end{equation*}
where
\begin{equation*}
 \chi (\zeta) =
 \left\{
  \begin{array}{l}
   1\,, \qquad 0 \le \zeta \le 1\,,\\
   0\,, \qquad \zeta \ge 2\,,\\
  \end{array}
 \right.
\end{equation*}
is a monotone cut-off function belonging to $C^\infty \big([0, \infty)\big)$. It is easy to verify that
\begin{equation}\label{theta_BC}
  \lim_{\zeta \to \infty} ( u_b, \theta_a) (t, \bar{x}, \zeta) = 0 \,,\ \big( \p_\zeta \theta_a - R_\theta \theta_a \big) |_{\zeta=0} = a (t,\bar{x})\,, \ (\p_{\zeta} u_{b} - R_u u_{b}) |_{\zeta=0} = b (t, \bar{x}) \,.
\end{equation}
Let $\mho = u - u_b \in \R^2$ and $\Theta = \theta - \theta_a$. Combining with \eqref{theta_Equa} implies
\begin{equation}\label{Theta_Equa}
 \left\{
  \begin{aligned}
   & \p_t \mho + \bar{\u}^0 \cdot \nabla_{\bar{x}} \mho + ( \p_{x_3} \u^0_3 \zeta + u_{1,3}^0 ) \p_{\zeta} \mho + \mho \cdot \nabla_{\bar{x}} \bar{\u}^0 + \tfrac{\nabla_{\bar{x}} p^0}{3 T^0} \Theta = \tilde{\mu} \p_{\zeta}^2 \mho + \tilde{f} \,, \\
   & \p_t \Theta + \bar{\u}^0 \cdot \nabla_{\bar{x}} \Theta + ( \p_{x_3} \u^0_3 \zeta + u^0_{1,3} ) \p_\zeta \Theta + \tfrac{2}{3} \div_x \u^0 \Theta = \tilde{\kappa} \p_{\zeta}^2 \Theta + \tilde{g}\,,\\
   & \lim_{\zeta \to \infty} ( \mho, \Theta ) (t, \bar{x}, \zeta) =0\,, \ ( \p_{\zeta} \mho - R_u \mho ) |_{\zeta=0} = 0 \,, \ (\p_\zeta \Theta - R_\theta \Theta) |_{\zeta=0} = 0\,, \\
   & ( \mho, \Theta ) (t, \bar{x}, \zeta) |_{t=0} = ( \mho, \Theta ) (0) : = \big( u_0 (\bar{x}, \zeta) - u_b (0, \bar{x}, \zeta) \,, \theta_0 (\bar{x}, \zeta) -\theta_a (0, \bar{x}, \zeta) \big) \,,
  \end{aligned}
 \right.
\end{equation}
where
\begin{equation}\label{tilde-f-g}
  \begin{aligned}
    \tilde{\mu} = & \tfrac{1}{\rho^0} \mu (T^0) \,, \ \, \tilde{f} = \tfrac{1}{\rho^0} f - \p_t u_b - \bar{\u}^0 \cdot \nabla_{\bar{x}} u_b - ( \p_{x_3} \u^0_3 \zeta + u^0_{1,3} ) \p_{\zeta} u_b - u_b \cdot \nabla_{\bar{x}} \bar{\u}^0 - \tfrac{\nabla_{\bar{x}} p^0}{3 \rho^0 T^0} \theta_a \,, \\
    \tilde{\kappa} = & \tfrac{3}{5 \rho^0} \kappa (T^0)\,,\ \tilde{g} = \tfrac{1}{\rho^0} g - \p_t \theta_a - \bar{\u}^0 \cdot \nabla_{\bar{x}} \theta_a - ( \p_{x_3} \u^0_3 \zeta + u^0_{1,3} ) \p_\zeta \theta_a - \tfrac{2}{3} \div_x \u^0 \theta_a + \tilde{\kappa} \p_{\zeta}^2 \theta_a\,.
  \end{aligned}
\end{equation}
From Proposition \ref{Proposition_Compressible_Euler} and \eqref{mu0-kappa0}, it is easy to see
\begin{equation}\label{LowBnd-2}
  \begin{aligned}
    \tilde{\mu} \geq \tfrac{2 \mu_0}{\rho_\#} : = \tilde{\mu}_0 > 0 \,, \ \tilde{\kappa} \geq \tfrac{2 \kappa_0}{\rho_\#} : = \tilde{\kappa}_0 > 0 \,.
  \end{aligned}
\end{equation}

To prove the existence of smooth solution to \eqref{Theta_Equa}, we first construct the following approximate system
\begin{equation}
	\left\{
	\begin{aligned}\label{Theta_Appro}
		& \p_t \mho + \bar{\u}^0 \cdot \nabla_{\bar{x}} \mho + ( \p_{x_3} \u^0_3 \zeta + u^0_{1,3} ) \chi_\sigma (\zeta) \p_\zeta \mho + \mho \cdot \nabla_{\bar{x}} \bar{\u}^0 + \tfrac{\nabla_{\bar{x}} p^0}{3 T^0} \Theta = \tilde{\mu} \p_{\zeta}^2 \mho + \lambda \Delta_{\bar{x}} \mho + \tilde{f}^\sigma \,, \\
		& \p_t \Theta + \bar{\u}^0 \cdot \nabla_{\bar{x}} \Theta + ( \p_{x_3} \u^0_3 \zeta + u^0_{1,3} ) \chi_\sigma (\zeta) \p_\zeta \Theta + \tfrac{2}{3} \div_x \u^0 \Theta = \tilde{\kappa} \p_\zeta^2 \Theta + \lambda \Delta_{\bar{x}} \Theta + \tilde{g}^\sigma\,,\\
		& ( \p_{\zeta} \mho - R_u \mho ) |_{\zeta=0} = 0 \,, \ (\p_\zeta \Theta - R_\theta \Theta)|_{\zeta=0} = 0\,,\quad ( \mho, \Theta ) (t, \bar{x}, \zeta)|_{\zeta =\frac{3}{\sigma}} = 0\,,\\
		& \mho (0, \bar{x}, \zeta) = ( u_0 - u_b ) \chi_\sigma (\zeta) \,, \ \Theta (0, \bar{x}, \zeta) = (\theta_0 - \theta_a) \chi_\sigma (\zeta)
	\end{aligned}
	\right.
\end{equation}
for $(t, \bar{x}, \zeta) \in [0, \tau] \times \R^2 \times [0, \tfrac{3}{\sigma}]$ and $0 < \sigma, \lambda \leq 1$, where $\chi_\sigma (\zeta) = \chi (\sigma \zeta)$, $\tilde{f}^\sigma = \tilde{f} \chi_\sigma (\zeta)$ and $\tilde{g}^\sigma = \tilde{g} \chi_\sigma (\zeta)$. One notices that the compatibility condition of initial data at $\zeta = \tfrac{3}{\sigma}$ is also satisfied due to the property of $\chi (\zeta)$. For the approximate problem \eqref{Theta_Appro}, we can use the standard linear parabolic theory to obtain the existence of smooth solution in Sobolev space provided that the initial data and $(\rho^0, \u^0, T^0)$ are suitably smooth. To prove the existence of the smooth solutions to \eqref{Theta_Equa}, we only need to obtain some uniform estimates of $( \mho, \Theta)$ associated with $\sigma$ and $\lambda$, then take the limit $\sigma, \lambda \to 0+$.

We will first estimate the uniform bounds on $\Theta$. The derivation of uniform bounds for $\mho$ is similar to that for $\Theta$.  For simplicity, we still employ the notation
\begin{equation*}
  \begin{aligned}
    \| f \|^2_{L^2_l} = \int_{\R^2} \int_0^\frac{3}{\sigma} (1 + \zeta)^l |f (\bar{x}, \zeta)|^2 \d \bar{x} \d \zeta
  \end{aligned}
\end{equation*}
in \eqref{L2l} and, correspondingly, employ the notations $\| f (t) \|^2_{\mathbb{H}^r_l (\R^3_+)}$ and $\| g \|_{\mathbb{H}^k_l (\R^3_+)}$ in \eqref{Hrl-t-3D} and \eqref{Hrl-3D}, respectively. Moreover, the following properties of $\chi_\sigma (\zeta)$ will be frequently used:
  \begin{align}\label{chi-property}
    & \zeta \chi_\sigma (\zeta) |_{\zeta = 0} = \zeta \chi_\sigma (\zeta) |_{\zeta = \frac{3}{\sigma}} = 0 \,, \ \big| \p_\zeta [ \zeta  (1+\zeta)^{l} \chi_\sigma (\zeta) ] \big| \leq C (1 + \zeta)^{l} \ \textrm{for any } l \in \mathbb{N}^* \,, \\
    \no & \big| \p_{\zeta}^n (\zeta \chi_\sigma (\zeta)) \big| \leq C (1 + \zeta) \ \textrm{ for any } n \geq 0 \,, \ [\p_{\zeta}^n, \zeta \chi_\sigma (\zeta)] f |_{\zeta=0} = n \p_{\zeta}^{n-1} f |_{\zeta=0} \ \textrm{ for any } n \geq 1 \,.
  \end{align}
Furthermore, by letting $k \geq 3$, $F^0 = F^0 (\rho^0, \u^0, T^0, \nabla_x (\rho^0, \u^0, T^0))$, and employing the standard Sobolev embedding arguments, one sees that for any $2 \gamma + |\beta| + n \leq k$
  \begin{align*}
    \big| \int_{\R^2} \int_0^\frac{3}{\sigma} & (1 + \zeta)^l | \p_t^\gamma \p_{\bar{x}}^\beta \p_{\zeta}^n ( F^0 \Theta ) |^2 \d \bar{x} \d \zeta \big| \leq \| F^0 \|^2_{L^\infty_{t,\bar{x}}} \| \p_t^\gamma \p_{\bar{x}}^\beta \p_{\zeta}^n \Theta \|^2_{L^2_l} \\
    & + \sum_{\substack{0 \neq |\beta'| + 2 \gamma' \leq |\beta| + 2 \gamma \\ \beta' \leq \beta , \gamma' \leq \gamma}} C_\beta^{\beta'} C_\gamma^{\gamma'} \| \p_t^{\gamma'} \p_{\bar{x}}^{\beta'} F^0 \|^2_{L^\infty_t L^4_{\bar{x}}} \| (1 + \zeta)^\frac{l}{2} \p_t^{\gamma - \gamma'} \p_{\bar{x}}^{\beta - \beta'} \p_{\zeta}^n \Theta \|^2_{L^4_{\bar{x}} L^2_\zeta} \\
    & \leq C ( E_{k+1} ) \| \Theta (t) \|^2_{\mathbb{H}^k_l (\R^3_+)} \leq C (E_{k+1}) \| \Theta (t) \|^2_{\mathbb{H}^k_l (\R^3_+)}
  \end{align*}
holds for all $\Theta \in \mathbb{H}^k_l (\R^3_+)$, which will be directly used to deal with the coefficients depending on $(\rho^0, \u^0, T^0)$ in what follows. Here the symbol $E_{k+1}$ is defined in \eqref{Ek}.

\vspace*{2mm}

{\bf Step 1. The zero-derivatives estimates: $L^2_{l_0}$-bounds.} For integer $l_0$ given in \eqref{Def-l_j} with $s = k$, multiplying the $\Theta$-equation of \eqref{Theta_Appro} with $(1+\zeta)^{l_0} \Theta$, and integrating the resultant over $[0, t] \times \R^2 \times [0, \tfrac{3}{\sigma}]$, one has
 \begin{align*}
 & \tfrac{1}{2} \|\Theta (t)\|_{L^2_{l_0}}^2 - \tfrac{1}{2} \|\Theta (0)\|_{L^2_{l_0}}^2 = - \int_0^t \int_{\R^2} \int_0^{\frac{3}{\sigma}} ( \p_{x_3} \u^0_3 \zeta + u^0_{1,3} ) \chi_\sigma (\zeta) \p_\zeta \Theta (1+\zeta)^{l_0} \Theta \d \bar{x} \d \zeta \d s \\
 & - \int_0^t \int_{\R^2} \int_0^{\frac{3}{\sigma}} \tfrac{2}{3} \div_x \u^0 \Theta (1+\zeta)^{l_0} \Theta \d \bar{x} \d \zeta \d s + \int_0^t \int_{\R^2} \int_0^{\frac{3}{\sigma}} \tilde{\kappa} \p_\zeta^2 \Theta (1+\zeta)^{l_0} \Theta \d \bar{x} \d \zeta \d s\\
 & + \lambda \int_0^t \int_{\R^2} \int_0^{\frac{3}{\sigma}} \Delta_{\bar{x}} \Theta (1+\zeta)^{l_0} \Theta \d \bar{x} \d \zeta \d s + \int_0^t \int_{\R^2} \int_0^{\frac{3}{\sigma}} \tilde{g}^\sigma (1+\zeta)^{l_0} \Theta \d \bar{x} \d \zeta \d s\\
 & =: I_1 + I_2 + I_3 + I_4 + I_5 \,,
 \end{align*}
where $u^0_{1,3} = \sqrt{T^0} (\rho^0 \sqrt{T^0} + 1) > 0$ is given in \eqref{BC-u_13}. Straightforward calculation gives us
\begin{align}\label{I1}
  \no I_1 = & - \int_0^t \int_{\R^2} ( \p_{x_3} \u^0_3 \zeta + u^0_{1,3} ) \chi_\sigma (\zeta) (1+\zeta)^{l_0} \tfrac{1}{2} \Theta^2 |_{\zeta =0}^{\zeta=\tfrac{3}{\sigma}} \d \bar{x} \d s \\
  \no & + \int_0^t \int_{\R^2} \int_0^{\frac{3}{\sigma}} \{ \p_{x_3} \u^0_3 \p_\zeta \big[ \zeta  (1+\zeta)^{l_0} \chi_\sigma (\zeta) \big] + u^0_{1,3} \p_\zeta \big[ (1+\zeta)^{l_0} \chi_\sigma (\zeta) \big] \} \tfrac{1}{2} \Theta^2 \d \bar{x} \d \zeta \d s \\
  \no \le & \tfrac{1}{2} \int_0^t \int_{\R^2} u^0_{1,3} \Theta^2 |_{\zeta = 0} \d \bar{x} \d s + C (E_{k+1}) \int_0^t \|\Theta (s)\|_{L^2_{l_0}}^2 \d s \\
  \leq & C (E_{k+1}) \int_0^t \| \Theta |_{\zeta=0} (s) \|^2_{L^2(\R^2)} \d s + C (E_{k+1}) \int_0^t \|\Theta (s)\|_{L^2_{l_0}}^2 \d s \\
  \no \leq & \tfrac{1}{4} \int_0^t \tilde{\kappa}_0 \| \partial_\zeta \Theta (s) \|^2_{L^2_{l_0}} \d s + C (E_{k+1}) \int_0^t \|\Theta (s)\|_{L^2_{l_0}}^2 \d s \,,
\end{align}
where the facts about $\chi_\sigma (\zeta)$ given in \eqref{chi-property} have been used, and the last inequality is derived from the inequality \eqref{Trace-1} in Lemma \ref{Lemma_Trace} and the Young's inequality. Furthermore,
\begin{align*}
  |I_2| \le C(E_{k+1}) \int_0^t \|\Theta(s)\|_{L^2_{l_0}}^2 \d s\,, \quad |I_5| \le C \int_0^t \|(\Theta, \tilde{g}^\sigma)(s)\|_{L^2_{l_0}}^2 \d s\,.
\end{align*}
Together with the boundary conditions in \eqref{Theta_Appro}, one can also compute
 \begin{align*}
  I_3 = & \int_0^t \int_{\R^2} \tilde{\kappa} (1+\zeta)^{l_0} \p_\zeta \Theta \Theta |_{\zeta=0}^{\zeta=\frac{3}{\sigma}} \d \bar{x} \d s - \int_0^t \int_{\R^2} \int_0^{\frac{3}{\sigma}} (1+\zeta)^{l_0} (\p_\zeta \Theta)^2 \d \bar{x} \d \zeta \d s \\
  & - \int_0^t \int_{\R^2} \int_0^{\frac{3}{\sigma}} l_0 \tilde{\kappa} (1+\zeta)^{l_0-1} \p_\zeta \Theta \Theta \d \bar{x} \d \zeta \d s \\
  \le & - \int_0^t \tilde{\kappa}_0 \|\p_\zeta \Theta(s)\|_{L^2_{l_0}}^2 \d s + \tfrac{1}{2} \int_0^t \tilde{\kappa}_0 \|\p_\zeta \Theta (s)\|_{L^2_{l_0}}^2 \d s\\
  & + C \int_0^t \|\Theta (s)\|_{L^2_{l_0}}^2 \d s - \int_0^t \int_{\R^2} \tilde{\kappa} R_\theta \Theta^2 |_{\zeta=0} \d \bar{x} \d s \\
  \le & - \tfrac{1}{2} \int_0^t \tilde{\kappa}_0 \|\p_\zeta \Theta (s)\|_{L^2_{l_0}}^2 \d s - \int_0^t \tilde{\kappa}_0 R^0_\theta \|\Theta|_{\zeta=0} (s)\|_{L^2(\R^2)}^2 \d s + C \int_0^t \|\Theta (s)\|_{L^2_{l_0}}^2 \d s \,,
 \end{align*}
where the lower bounds \eqref{LowBnd-1} and \eqref{LowBnd-2}, i.e., $ 0 < \tilde{\kappa}_0 \le \tilde{\kappa} $ and $ 0< R^0_\theta \le R_\theta $, are also used. For $I_4$, one has
\begin{equation*}
 I_4 = -\lambda \int_0^t \int_{\R^2} \int_0^{\frac{3}{\sigma}} (1+\zeta)^{l_0} |\nabla_{\bar{x}} \Theta|^2 \d \bar{x} \d \zeta \d s
 = \lambda \int_0^t \|\nabla_{\bar{x}} \Theta (s)\|_{L^2_{l_0}}^2 \d s\,.
\end{equation*}
Therefore, one has
\begin{equation}\label{Theta_L_2}
 \begin{aligned}
   \|\Theta (s)\|_{L^2_{l_0}}^2 + \tfrac{1}{2} \int_0^t \tilde{\kappa}_0 \|\p_\zeta \Theta (s)\|_{L^2_{l_0}}^2 + \lambda \|\nabla_{\tilde{x}} \Theta (s)\|_{L^2_{l_0}}^2 \d s + \int_0^t c_2'  \|\Theta|_{\zeta=0} (s)\|_{L^2(\R^2)}^2 \d s\\
   \le C \big( \|\Theta(0)\|_{L^2_{l_0}}^2 + \int_0^t \|\tilde{g}^{\sigma} (s)\|_{L^2_{l_0}}^2 \d s \big) + C(E_{k+1}) \int_0^t \|\Theta(s)\|_{L^2_{l_0}}^2 \d s\,,
 \end{aligned}
\end{equation}
where $ c_2' =\tilde{\kappa}_0 R^0_\theta > 0$.

\vspace*{2mm}

{\bf Step 2. The higher order $(t, \bar{x})$-derivatives estimates: $\mathbb{H}^r_{l,0} (\R^3_+)$-bounds.} For $2\gamma + |\beta| = r$ with $\gamma \in \mathbb{N}, \beta \in \mathbb{N}^2$ and $1\le r \le k$, applying $\p_t^\gamma \p_{\bar{x}}^\beta$ to the $\Theta$-equation of \eqref{Theta_Appro} infers that
\begin{equation}\label{Theta_High_Derivative_0}
 \begin{aligned}
  & \p_t \p_t^\gamma \p_{\bar{x}}^\beta \Theta + \bar{\u}^0 \cdot \nabla_{\bar{x}} \p_t^\gamma \p_{\bar{x}}^\beta \Theta + ( \p_{x_3} \u^0_3 \zeta + u^0_{1,3} ) \chi_\sigma (\zeta) \p_\zeta \p_t^\gamma \p_{\bar{x}}^\beta \Theta
  \\
  = & \p_t^\gamma \p_{\bar{x}}^\beta (\tilde{\kappa} \p_\zeta^2 \Theta) + \lambda \Delta_{\bar{x}} \p_t^\gamma \p_{\bar{x}}^\beta \Theta - [\p_t^\gamma \p_{\bar{x}}^\beta, \bar{\u}^0 \cdot \nabla_{\bar{x}}]\Theta - \zeta \chi_\sigma (\zeta) [\p_t^\gamma \p_{\bar{x}}^\beta, \p_{x_3} \u^0_3 \p_\zeta]\Theta \\
  & - \chi_\sigma (\zeta) [\p_t^\gamma \p_{\bar{x}}^\beta, u^0_{1,3} \p_{\zeta} ] \Theta - \p_t^\gamma \p_{\bar{x}}^\beta(\tfrac{2}{3} \div_{\bar{x}} \u^0 \Theta) + \p_t^\gamma \p_{\bar{x}}^\beta \tilde{g}^\sigma\,,
 \end{aligned}
\end{equation}
where $[X, Y] = XY-YX$ is the commutator operator. Multiplying \eqref{Theta_High_Derivative_0} with $(1+\zeta)^{l_r} \p_t^\gamma \p_{\bar{x}}^\beta \Theta$ and integrating over $[0,t] \times \R^2 \times [0,\tfrac{3}{\sigma}]$ reduce to
\begin{align}\label{Theta_High_Derivative_1}
  \tfrac{1}{2} \|\p_t^\gamma \p_{\bar{x}}^\beta \Theta(s) \|_{L^2_{l_r}}^2 - \tfrac{1}{2} \|\p_t^\gamma \p_{\bar{x}}^\beta \Theta(0)\|_{L^2_{l_r}}^2 = \sum_{k=1}^9  II_k \,,
\end{align}
where
\begin{align*}
 II_1 = & -\int_0^t \int_{\R^2} \int_0^{\frac{3}{\sigma}} ( \p_{x_3} \u^0_3 \zeta + u^0_{1,3} ) \chi_\sigma (\zeta) \p_\zeta \p_t^\gamma \p_{\bar{x}}^\beta \Theta \cdot (1+\zeta)^{l_r} \p_t^\gamma \p_{\bar{x}}^\beta \Theta \d \bar{x} \d \zeta \d s \,, \\
 II_2 = & \int_0^t \int_{\R^2} \int_0^{\frac{3}{\sigma}} \p_t^\gamma \p_{\bar{x}}^\beta (\tilde{\kappa} \p_\zeta^2 \Theta) (1+\zeta)^{l_r} \p_t^\gamma \p_{\bar{x}}^\beta \Theta \d \bar{x} \d \zeta \d s \,, \\
 II_3 = & \lambda \int_0^t \int_{\R^2} \int_0^{\frac{3}{\sigma}} \Delta_{\bar{x}} \p_t^\gamma \p_{\bar{x}}^\beta \Theta \cdot (1+\zeta)^{l_r} \p_t^\gamma \p_{\bar{x}}^\beta \Theta \d \bar{x} \d \zeta \d s \,, \\
 II_4 = & - \int_0^t \int_{\R^2} \int_0^{\frac{3}{\sigma}} \bar{\u}^0 \cdot \nabla_{\bar{x}} \p_t^\gamma \p_{\bar{x}}^\beta \Theta \cdot (1+\zeta)^{l_r} \p_t^\gamma \p_{\bar{x}}^\beta \Theta \d \bar{x} \d \zeta \d s \,, \\
 II_5 =  & - \int_0^t \int_{\R^2} \int_0^{\frac{3}{\sigma}} [\p_t^\gamma \p_{\bar{x}}^\beta, \bar{\u}^0 \cdot \nabla_{\bar{x}}] \Theta \cdot (1+\zeta)^{l_r} \p_t^\gamma \p_{\bar{x}}^\beta \Theta \d \bar{x} \d \zeta \d s \,, \\
 II_6 = & - \int_0^t \int_{\R^2} \int_0^{\frac{3}{\sigma}} \zeta \chi_\sigma (\zeta) [\p_t^\gamma \p_{\bar{x}}^\beta, \p_{x_3} \u^0_3 \p_\zeta] \Theta \cdot (1+\zeta)^{l_r} \p_t^\gamma \p_{\bar{x}}^\beta \Theta \d \bar{x} \d \zeta \d s \,, \\
 II_7 = & - \int_0^t \int_{\R^2} \int_0^{\frac{3}{\sigma}} \p_t^\gamma \p_{\bar{x}}^\beta (\tfrac{2}{3}\div_x \u^0 \Theta) \cdot (1+\zeta)^{l_r} \p_t^\gamma \p_{\bar{x}}^\beta \Theta \d \bar{x} \d \zeta \d s \,, \\
 II_8 = & \int_0^t \int_{\R^2} \int_0^{\frac{3}{\sigma}} \p_t^\gamma \p_{\bar{x}}^\beta \tilde{g}^\sigma \cdot (1+\zeta)^{l_r} \p_t^\gamma \p_{\bar{x}}^\beta \Theta \d \bar{x} \d \zeta \d s \,, \\
 II_9 = & \int_0^t \int_{\R^2} \int_0^{\frac{3}{\sigma}} \chi_\sigma (\zeta) [\p_t^\gamma \p_{\bar{x}}^\beta, u^0_{1,3} \p_{\zeta} ] \Theta \cdot (1+\zeta)^{l_r} \p_t^\gamma \p_{\bar{x}}^\beta \Theta \d \bar{x} \d \zeta \d s \,,
\end{align*}

We then estimate the terms in the RHS of \eqref{Theta_High_Derivative_1} one by one. Integrating by parts with respect to the variable $\zeta$ and employing the similar derivation of $I_1$ in \eqref{I1} before give us
\begin{equation}\label{II1}
 \begin{aligned}
   II_1 \le & \tfrac{1}{4} \int_0^t \tilde{\kappa}_0 \| \partial_\zeta \p_t^\gamma \p_{\bar{x}}^\beta \Theta (s) \|^2_{L^2_{l_r}} \d s + C(E_{k+1}) \int_0^t \|\p_t^\gamma \p_{\bar{x}}^\beta \Theta (s)\|_{L^2_{l_r}}^2 \d s\,.
 \end{aligned}
\end{equation}
For $II_2$, a direct calculation infers that
\begin{align*}
 II_2 = & \int_0^t \int_{\R^2} \p_t^\gamma \p_{\bar{x}}^\beta (\tilde{\kappa} \p_\zeta \Theta) (1+\zeta)^{l_r} \p_t^\gamma \p_{\bar{x}}^\beta \Theta |_{\zeta=0}^{\zeta = \frac{3}{\sigma}} \d \bar{x} \d s - \int_0^t \int_{\R^2} \int_0^{\frac{3}{\sigma}} \tilde{\kappa} (1+\zeta)^{l_r} (\p_t^\gamma \p_{\bar{x}}^\beta \Theta)^2 \d \bar{x} \d \zeta \d s\\
 & - \int_0^t \int_{\R^2} \int_0^{\frac{3}{\sigma}} (1+\zeta)^{l_r} [\p_t^\gamma \p_{\bar{x}}^\beta, \tilde{\kappa} \p_\zeta] \Theta \cdot \p_\zeta \p_t^\gamma \p_{\bar{x}}^\beta \Theta \d \bar{x} \d \zeta \d s = : II_{21} + II_{22} + II_{23} \,.
\end{align*}
The quantity $II_{21}$ is bounded by
  \begin{align*}
    & - \int_0^t \int_{\R^2} \tilde{\kappa} R_\theta (\p_t^\gamma \p_{\bar{x}}^\beta \Theta)^2 |_{\zeta=0} \d \bar{x} \d s - \int_0^t \int_{\R^2} [\p_t^\gamma \p_{\bar{x}}^\beta, \tilde{\kappa} R_\theta] \Theta \cdot \p_t^\gamma \p_{\bar{x}}^\beta \Theta |_{\zeta=0} \d \bar{x} \d s\\
    \le & - \tfrac{3}{4} \int_0^t \tilde{\kappa}_0 R^0_\theta \|\p_t^\gamma \p_{\bar{x}}^\beta \Theta |_{\zeta=0} (s)\|_{L^2(\R^2)}^2 \d s + C(E_{k+1}) \int_0^t \| \Theta |_{\zeta=0} (s) \|^2_{\mathbb{H}^{r-1} (\R^2)} \d s \,,
  \end{align*}
where the boundary values $\p_t^\gamma \p_{\bar{x}}^\beta (\tilde{\kappa} \p_{\zeta} \Theta) |_{\zeta=0} = \p_t^\gamma \p_{\bar{x}}^\beta (\tilde{\kappa} R_\theta \Theta) |_{\zeta=0}$ and $\p_t^\gamma \p_{\bar{x}}^\beta \Theta |_{\zeta=\frac{3}{\sigma}} = 0$ have been used. The quantity $II_{22}$ is bounded by
\begin{equation*}
  \begin{aligned}
    - \int_0^t \tilde{\kappa}_0 \|\p_\zeta \p_t^\gamma \p_{\bar{x}}^\beta \Theta (s)\|_{L^2_{l_r}}^2 \d s \,,
  \end{aligned}
\end{equation*}
and the quantity $II_{23}$ is bounded by
\begin{equation*}
  \begin{aligned}
    \tfrac{1}{4} \int_0^t \tilde{\kappa}_0 \|\p_\zeta \p_t^\gamma \p_{\bar{x}}^\beta \Theta (s)\|_{L^2_{l_r}}^2 \d s + C(E_{k+1}) \sum_{j=1}^{r-1} \sum_{2\tilde{\gamma} + |\tilde{\beta}| =j} \int_0^t \|\p_\zeta \p_t^{\tilde{\gamma}} \p_{\bar{x}}^{\tilde{\beta}} \Theta (s)\|_{L^2_{l_{j+1}}}^2 \d s \\
    \leq \tfrac{1}{4} \int_0^t \tilde{\kappa}_0 \|\p_\zeta \p_t^\gamma \p_{\bar{x}}^\beta \Theta (s)\|_{L^2_{l_r}}^2 \d s + C(E_{k+1}) \int_0^t \| \p_{\zeta} \Theta (s) \|^2_{\mathbb{H}^{r-1}_{l,0}(\R^3_+)} \d s \,,
  \end{aligned}
\end{equation*}
where we require $l_j \ge l_{j+1}$. The above estimates also require the lower bounds \eqref{LowBnd-1} and \eqref{LowBnd-2}. Consequently,
  \begin{align}\label{II2}
    \no II_2 \leq & -\tfrac{3}{4} \int_0^t \tilde{\kappa}_0 \|\p_\zeta \p_t^\gamma \p_{\bar{x}}^\beta \Theta (s)\|_{L^2_{l_r}}^2 \d s - \tfrac{3}{4} \int_0^t \tilde{\kappa}_0 R^0_\theta \|\p_t^\gamma \p_{\bar{x}}^\beta \Theta |_{\zeta=0} (s)\|_{L^2(\R^2)}^2 \d s \\
    & + C(E_{k+1}) \int_0^t \| \p_{\zeta} \Theta (s) \|^2_{\mathbb{H}^{r-1}_{l,0}(\R^3_+)} + \| \Theta |_{\zeta=0} (s) \|^2_{\mathbb{H}^{r-1} (\R^2)} \d s \,.
  \end{align}
Similarly, one also has
\begin{equation}\label{II3578}
 \begin{aligned}
  & II_3 = -\lambda \int_0^t \|\nabla_{\bar{x}} \p_t^\gamma \p_{\bar{x}}^\beta \Theta (s)\|_{L^2_{l_r}}^2 \d s\,,\ |II_5| \le C(E_{k+1}) \int_0^t  \| \Theta  (s)\|_{\mathbb{H}^r_{l,0}(\R^3_+)}^2 \d s\,,\\
  & |II_7| + |II_8| \le C(E_{k+1}) \int_0^t  \| \Theta  (s)\|_{\mathbb{H}^r_{l,0}(\R^3_+)}^2 \d s + C \int_0^t \|\p_t^\gamma \p_{\bar{x}}^\beta \tilde{g}^\sigma (s)\|_{L^2_{l_r}}^2 \d s\,,
 \end{aligned}
\end{equation}
and
\begin{equation}\label{II4}
 \begin{aligned}
  |II_4| = \big| \int_0^t \int_{\R^2} \int_0^{\frac{3}{\sigma}} \div_{\bar{x}} \bar{\u}^0 (1+\zeta)^{l_r} \tfrac{1}{2} (\p_t^\gamma \p_{\bar{x}}^\beta \Theta (s) )^2 \d \bar{x} \d \zeta \d s \big| \le C(E_{k+1}) \int_0^t \|\p_t^\gamma \p_{\bar{x}}^\beta \Theta (s)\|_{L^2_{l_r}}^2 \d s\,,
 \end{aligned}
\end{equation}
and
 \begin{align}\label{II6}
   \no |II_6| \le & C(E_{k+1}) \int_0^t \int_{\R^2} \int_0^{\frac{3}{\sigma}} (1+\zeta)^{l_r+1} \sum_{j=0}^{r-1} \sum_{2\tilde{\gamma} + |\tilde{\beta}| =j} |\p_\zeta \p_t^{\tilde{\gamma}} \p_{\bar{x}}^{\tilde{\beta}} \Theta| |\p_t^\gamma \p_{\bar{x}}^\beta \Theta| \d \bar{x} \d \zeta \d s\\
   \no \le & C(E_{k+1}) \sum_{j=0}^{r-1} \sum_{2\tilde{\gamma} + |\tilde{\beta}| =j} \Big( \int_0^t \int_{\R^2} \int_0^{\frac{3}{\sigma}} (1+\zeta)^{l_r +2} |\p_\zeta \p_t^{\tilde{\gamma}} \p_{\bar{x}}^{\tilde{\beta}} \Theta|^2 \bar{x} \d \zeta \d s \Big)^{\frac{1}{2}} \\
   \no & \times \Big( \int_0^t \|\p_t^\gamma \p_{\bar{x}}^\beta \Theta (s)\|_{L^2_{l_r}}^2 \d s \Big)^{\frac{1}{2}}\\
   \leq &  C(E_{k+1}) \int_0^t \|\p_t^\gamma \p_{\bar{x}}^\beta \Theta (s)\|_{L^2_{l_r}}^2 + \| \p_{\zeta} \Theta (s) \|^2_{\mathbb{H}^{r-1}_{l,0} (\R^3_+)} \d s \,.
 \end{align}
Here $l_{r-1} = l_r +2$ is required, which gives the definition of $l_j$ in \eqref{Def-l_j}. Furthermore, the quantity $II_9$ can easily bounded by
\begin{equation}\label{II9}
  \begin{aligned}
    II_9 \leq & C(E_{k+1}) \int_0^t \| \p_\zeta \Theta (s) \|^2_{\mathbb{H}^{r-1}_{l,0} (\R^3_+)} \d s + C(E_{k+1}) \int_0^t \| \p_t^\gamma \p_{\bar{x}}^\beta \Theta (s) \|^2_{L^2_{l_r}} \d s \,.
  \end{aligned}
\end{equation}
It thereby derived from substituting the estimates \eqref{II1}, \eqref{II2}, \eqref{II3578}, \eqref{II4}, \eqref{II6} and \eqref{II9} into \eqref{Theta_High_Derivative_1} that
\begin{equation}\label{Theta_High_Derivative_t_x}
 \begin{aligned}
  & \| \Theta(t)\|_{l,r,0}^2 + \int_0^t \tilde{\kappa}_0 \|\p_\zeta \Theta(s) \|_{l,r,0}^2 + \lambda \|\nabla_{\bar{x}} \Theta(s) \|_{l,r,0}^2 + \tilde{\kappa}_0 R^0_\theta \| \Theta |_{\zeta=0} (s)\|_{\Gamma, r}^2 \d s \\
  \le & C \Big( \| \Theta(0)\|_{l,r,0}^2 + \int_0^t \| \tilde{g}^\sigma (s)\|_{l,r,0}^2 \d s \Big) + C(E_{k+1}) \int_0^t \| \Theta (s)\|_{\mathbb{H}^r_{l,0}(\R^3_+)}^2 \d s \\
  & + C(E_{k+1}) \int_0^t \|\p_\zeta \Theta (s) \|_{\mathbb{H}^{r-1}_{l,0}(\R^3_+)}^2 + \| \Theta |_{\zeta=0} (s)\|_{\mathbb{H}^{r-1}(\R^2)}^2\d s
 \end{aligned}
\end{equation}
for $1\le r \le k$ ($k \ge 3$). Together with \eqref{Theta_L_2} and \eqref{Theta_High_Derivative_t_x}, we can apply the induction for $0 \leq l \leq k (k \geq 3)$ to imply
\begin{equation}\label{Theta_tx1x2_Dert}
  \begin{aligned}
    \| \Theta(t)\|_{\mathbb{H}^r_{l,0}(\R^3_+)}^2 + \int_0^t \|\p_\zeta \Theta(s) \|_{\mathbb{H}^r_{l,0}(\R^3_+)}^2 + \lambda \|\nabla_{\bar{x}} \Theta(s) \|_{\mathbb{H}^r_{l,0}(\R^3_+)}^2 + \| \Theta |_{\zeta=0} (s)\|_{\mathbb{H}^r(\R^2)}^2 \d s \\
    \le C(E_{k+1}) \Big( \| \Theta(0)\|_{\mathbb{H}^r_{l,0}(\R^3_+)}^2 + \int_0^t \| (\Theta, \tilde{g}^\sigma) (s)\|_{\mathbb{H}^r_{l,0}(\R^3_+)}^2 \d s \Big)
  \end{aligned}
\end{equation}
for any $0 \leq l \leq k (k \geq 3)$.

\vspace*{2mm}

{\bf Step 3. The mixed $(t, \bar{x}, \zeta)$-derivatives estimates: $\sum_{1 \leq n \leq r} \mathbb{H}^r_{l,n} (\R^3_+)$-bounds.} While estimating the higher order  $(t, \bar{x}, \zeta)$-derivatives, we shall subtly deal with boundary values of higher order $\zeta$-derivatives of $\Theta$. In order to overcome the issues, we first give the following lemma, which can be proved directly. Here we omit the details of proof for simplicity of presentation.
\begin{lemma}\label{Lemma_Theta_BC_infty}
	Assume $\Theta (t, \bar{x}, \zeta)$ is a smooth solution to the linear parabolic equation on $\Theta$ in \eqref{Theta_Appro} over $(t, \bar{x}, \zeta) \in [0,\tau] \times \R^2 \times [0,\tfrac{3}{\sigma}]$. Then the boundary values of higher oder $\zeta$-derivatives of $\Theta$ satisfy that for $m \ge 1$:
	\begin{equation}\label{Theta_BC_infty}
	\begin{aligned}
	\tilde{\kappa} \p_\zeta^{2m+1} \Theta |_{\zeta=\frac{3}{\sigma}} = (\p_t + \bar{\u}^0 \cdot \nabla_{\bar{x}} - \lambda \Delta_{\bar{x}} + \tfrac{2}{3} \div_x \u^0) \p_\zeta^{2m-1} \Theta |_{\zeta = \frac{3}{\sigma}} \,, \ \tilde{\kappa} \p_\zeta^{2m} \Theta |_{\zeta=\frac{3}{\sigma}} = 0
	\end{aligned}
	\end{equation}
	and
	\begin{equation}\label{Theta_Recursion}
	\begin{aligned}
	  \p_\zeta^{m+1} \Theta |_{\zeta=0} = & (\mathscr{L} + \tilde{\mathscr{L}}_{m-1}) \p_\zeta^{m-1} \Theta |_{\zeta=0} + \hat{\mathscr{L}} \p_\zeta^{m-1} \tilde{g}^\sigma |_{\zeta=0} \,,
	\end{aligned}
	\end{equation}
	where
	\begin{align}\label{L-scr}
	  \no & \tilde{\mathscr{L}}_{m-1} f = \tilde{\kappa}^{-1} \big ( \bar{\u}^0 \cdot \nabla_{\bar{x}}+ \tfrac{2}{3} \div_x \u^0 + (m-1) \p_{x_3} \u^0_3\big) f\,,\\
	  & \mathscr{L} f = \tilde{\kappa}^{-1} (\p_t - \lambda \Delta_{\bar{x}} ) f \,, \quad \hat{\mathscr{L}} f = - \tilde{\kappa}^{-1} f\,.
	\end{align}
	Moreover, if $m$ is odd, there holds
	\begin{align*}
	\p_\zeta^{m+1} \Theta |_{\zeta=0} =  \prod_{i=1}^{\frac{m+1}{2}}(\mathscr{L} + \tilde{\mathscr{L}}_{m+1-2i}) \Theta |_{\zeta=0} + \sum_{i=0}^{\frac{m-1}{2}} \prod_{j=1}^i (\mathscr{L} + \tilde{\mathscr{L}}_{m+1-2j}) \hat{\mathscr{L}} \p_\zeta^{m-1-2i} \tilde{g}^\sigma |_{\zeta=0}
	\end{align*}
	If $m$ is even, there holds
	\begin{align*}
	  \p_\zeta^{m+1} \Theta |_{\zeta=0} = \prod_{i=1}^\frac{m}{2} (\mathscr{L} + \tilde{\mathscr{L}}_{m+1-2i}) (R_\theta \Theta) |_{\zeta=0} + \sum_{i=0}^{\frac{m}{2}-1} \prod_{j=1}^i (\mathscr{L} + \tilde{\mathscr{L}}_{m+1-2j}) \hat{\mathscr{L}} \p_\zeta^{m-1-2i} \tilde{g}^\sigma |_{\zeta=0} \,.
	\end{align*}
\end{lemma}

One notices that
\begin{equation}\label{L-scr-tilde}
  \begin{aligned}
    \tilde{\mathscr{L}}_m \thicksim \tilde{\mathscr{L}} : = \tilde{\kappa}^{-1} \big ( \bar{\u}^0 \cdot \nabla_{\bar{x}}+ \tfrac{2}{3} \div_x \u^0 + \p_{x_3} \u^0_3\big)
  \end{aligned}
\end{equation}
while carrying the operators $\tilde{\mathscr{L}}_m$ in the following estimates. For simplicity of presentation, we will employ the relations on the boundary $\zeta = 0$
\begin{equation}\label{BC-Thicksim}
  \begin{aligned}
    \p_\zeta^{m+1} \Theta \thicksim
    \left\{
      \begin{aligned}
        & (\mathscr{L} + \tilde{\mathscr{L}})^{\frac{m+1}{2}} \Theta + \sum_{i=0}^{\frac{m-1}{2}} (\mathscr{L} + \tilde{\mathscr{L}})^i \hat{\mathscr{L}} \p_\zeta^{m-1-2i} \tilde{g}^\sigma  \ \textrm{for odd } m \,, \\
        & (\mathscr{L} + \tilde{\mathscr{L}})^\frac{m}{2} (R_\theta \Theta) + \sum_{i=0}^{\frac{m}{2}-1} (\mathscr{L} + \tilde{\mathscr{L}})^i \hat{\mathscr{L}} \p_\zeta^{m-1-2i} \tilde{g}^\sigma \ \textrm{for even } m \,,
      \end{aligned}
    \right.
  \end{aligned}
\end{equation}
in the rest of the paper. We then return to the proof of Proposition \ref{Prop-Prandtl}.

\vspace*{2mm}

{\bf Case 1. The first-order $\zeta$-derivatives: $\mathbb{H}^r_{l,1} (\R^3_+)$-bounds.} Applying $\p_t^\gamma \p_{\bar{x}}^\beta \p_\zeta$ to the $\Theta$-equation of \eqref{Theta_Appro} with $2 \gamma + |\beta| = r-1$, $1 \le r \le k$ ($k\ge3$) yields
\begin{equation}\label{Theta_t_x_zeta_1} \begin{aligned}
  & \p_t \p_t^\gamma \p_{\bar{x}}^\beta \p_\zeta \Theta + \bar{\u}^0 \cdot \nabla_{\bar{x}} \p_t^\gamma \p_{\bar{x}}^\beta \p_\zeta \Theta + ( \p_{x_3} \u^0_3 \zeta + u^0_{1,3} ) \chi_\sigma (\zeta) \p_t^\gamma \p_{\bar{x}}^\beta \p_\zeta^2 \Theta \\
  = & \p_t^\gamma \p_{\bar{x}}^\beta (\tilde{\kappa} \p_\zeta^3 \Theta) + \lambda \Delta_{\bar{x}} \p_t^\gamma \p_{\bar{x}}^\beta \p_\zeta \Theta - \zeta \chi_\sigma (\zeta) [\p_t^\gamma \p_{\bar{x}}^\beta, \p_{x_3} \u^0_3 \p_\zeta^2] \Theta \\
  & - \p_{x_3} \u^0_3 \p_\zeta (\zeta \chi_\sigma (\zeta)) \partial_t^\gamma \partial_{\bar{x}}^\beta \partial_\zeta \Theta - [\p_t^\gamma \p_{\bar{x}}^\beta, \bar{\u}^0 \cdot \nabla_{\bar{x}}] \p_\zeta \Theta - [\p_t^\gamma \p_{\bar{x}}^\beta \p_\zeta, u^0_{1,3} \chi_\sigma (\zeta) \p_\zeta] \Theta \\
  & - \p_\zeta (\zeta \chi_\sigma (\zeta)) [\p_t^\gamma \p_{\bar{x}}^\beta, \p_{x_3} \u^0_3 \p_\zeta] \Theta - \p_t^\gamma \p_{\bar{x}}^\beta \p_\zeta (\tfrac{2}{3} \div_x \u^0 \Theta) + \p_t^\gamma \p_{\bar{x}}^\beta \p_\zeta \tilde{g}^\sigma\,.
 \end{aligned}
\end{equation}
By Lemma \ref{Lemma_Theta_BC_infty},
\begin{equation}\label{Theta_BC_2_Order}
  \tilde{\kappa} \p_\zeta^2 \Theta |_{\zeta=0} = \mathscr{L}_{\Theta} (\Theta, \tilde{g}^\sigma) |_{\zeta=0}\,,\quad \Theta |_{\zeta =\frac{3}{\sigma}} = \p_\zeta^2 \Theta |_{\zeta=\frac{3}{\sigma}} = 0\,,
\end{equation}
where $
\mathscr{L}_{\Theta} (\Theta, \tilde{g}^\sigma) = \big( \p_t \Theta - \lambda \Delta_{\bar{x}} \Theta + \bar{\u}^0 \cdot \nabla_{\bar{x}} \Theta + \tfrac{2}{3} \div_x \u^0 \Theta - \tilde{g}^\sigma \big) $. Multiplying \eqref{Theta_t_x_zeta_1} with $(1+\zeta)^{l_r} \p_t^\gamma \p_{\bar{x}}^\beta \p_\zeta \Theta$, and integrating over $[0,t] \times \R^2 \times [0,\tfrac{3}{\sigma}]$, one has
\begin{align*}
  & \tfrac{1}{2} \|\p_t^\gamma \p_{\bar{x}}^\beta \p_\zeta \Theta (t)\|_{L^2_{l_r}}^2 - \tfrac{1}{2} \|\p_t^\gamma \p_{\bar{x}}^\beta \p_\zeta \Theta (0)\|_{L^2_{l_r}}^2 = \sum_{k=0}^{11} III_k \,,
 \end{align*}
 where
 \begin{align*}
  III_1 = & - \int_0^t \int_{\R^2} \int_0^{\frac{3}{\sigma}} \bar{\u}^0 \cdot \nabla_{\bar{x}} \p_t^\gamma \p_{\bar{x}}^\beta \p_\zeta \Theta \cdot (1+\zeta)^{l_r} \p_t^\gamma \p_{\bar{x}}^\beta \p_\zeta \Theta \d \bar{x} \d \zeta \d s \,, \\
  III_2 = & - \int_0^t \int_{\R^2} \int_0^{\frac{3}{\sigma}} ( \p_{x_3} \u^0_3 \zeta + u^0_{1,3} ) \chi_\sigma (\zeta) \p_t^\gamma \p_{\bar{x}}^\beta \p_\zeta^2 \Theta (1+\zeta)^{l_r} \p_t^\gamma \p_{\bar{x}}^\beta \p_\zeta^2 \Theta \d \bar{x} \d \zeta \d s \,, \\
  III_3 = & \int_0^t \int_{\R^2} \int_0^{\frac{3}{\sigma}} \p_t^\gamma \p_{\bar{x}}^\beta (\tilde{\kappa} \p_\zeta^3 \Theta) \cdot (1+\zeta)^{l_r} \p_t^\gamma \p_{\bar{x}}^\beta \p_\zeta \Theta \d \bar{x} \d \zeta \d s \,, \\
  III_4 = & \int_0^t \int_{\R^2} \int_0^{\frac{3}{\sigma}} \lambda \Delta_{\bar{x}} \p_t^\gamma \p_{\bar{x}}^\beta \p_\zeta \Theta \cdot (1+\zeta)^{l_r} \p_t^\gamma \p_{\bar{x}}^\beta \p_\zeta \Theta \d \bar{x} \d \zeta \d s \,, \\
  III_5 = & - \int_0^t \int_{\R^2} \int_0^{\frac{3}{\sigma}} \zeta \chi_\sigma (\zeta) [\p_t^\gamma \p_{\bar{x}}^\beta, \p_{x_3} \u^0_3 \p_\zeta^2] \Theta \cdot (1+\zeta)^{l_r} \p_t^\gamma \p_{\bar{x}}^\beta \p_\zeta \Theta \d \bar{x} \d \zeta \d s \,, \\
  III_6 = & - \int_0^t \int_{\R^2} \int_0^{\frac{3}{\sigma}} \p_{x_3} \u^0_3 \p_\zeta (\zeta \chi_\sigma (\zeta)) \p_t^\gamma \p_{\bar{x}}^\beta \p_\zeta \Theta \cdot (1+\zeta)^{l_r} \p_t^\gamma \p_{\bar{x}}^\beta \p_\zeta \Theta \d \bar{x} \d \zeta \d s \,, \\
  III_7 = & - \int_0^t \int_{\R^2} \int_0^{\frac{3}{\sigma}} [\p_t^\gamma \p_{\bar{x}}^\beta, \bar{\u}^0 \cdot \nabla_{\bar{x}}] \p_\zeta \Theta \cdot (1+\zeta)^{l_r} \p_t^\gamma \p_{\bar{x}}^\beta \p_\zeta \Theta \d \bar{x} \d \zeta \d s \,, \\
  III_8 = & - \int_0^t \int_{\R^2} \int_0^{\frac{3}{\sigma}} \p_\zeta (\zeta \chi_\sigma (\zeta)) [\p_t^\gamma \p_{\bar{x}}^\beta, \p_{x_3} \u^0_3 \p_\zeta] \Theta \cdot (1+\zeta)^{l_r} \p_t^\gamma \p_{\bar{x}}^\beta \p_\zeta \Theta \d \bar{x} \d \zeta \d s \,, \\
  III_9 = & - \int_0^t \int_{\R^2} \int_0^{\frac{3}{\sigma}} \p_t^\gamma \p_{\bar{x}}^\beta \p_\zeta \big( \tfrac{2}{3} \div_x \u^0 \Theta \big) \cdot (1+\zeta)^{l_r} \p_t^\gamma \p_{\bar{x}}^\beta \p_\zeta \Theta \d \bar{x} \d \zeta \d s \,, \\
  III_{10} = & \int_0^t \int_{\R^2} \int_0^{\frac{3}{\sigma}} \p_t^\gamma \p_{\bar{x}}^\beta \p_\zeta \tilde{g}^\sigma \cdot (1+\zeta)^{l_r} \p_t^\gamma \p_{\bar{x}}^\beta \p_\zeta \Theta \d \bar{x} \d \zeta \d s \,, \\
  III_{11} = & - \int_0^t \int_{\R^2} \int_0^{\frac{3}{\sigma}} [\p_t^\gamma \p_{\bar{x}}^\beta \p_\zeta, u^0_{1,3} \chi_\sigma (\zeta) \p_\zeta] \Theta \cdot (1+\zeta)^{l_r} \p_t^\gamma \p_{\bar{x}}^\beta \p_\zeta \Theta \d \bar{x} \d \zeta \d s \,.
\end{align*}
Integration by parts over $\bar{x} \in \R^2$, $\zeta \in [0, \frac{3}{\sigma}]$ and Sobolev theory yield
\begin{align}
  |III_1| \le & C(E_{k+1}) \int_0^t \|\p_t^\gamma \p_{\bar{x}}^\beta \p_\zeta \Theta (s)\|^2_{L^2_{l_r}} \d s\,, \label{III_1}\\
  III_2 \le & \tfrac{1}{4} \int_0^t \tilde{\kappa}_0 \| \partial_\zeta \p_t^\gamma \p_{\bar{x}}^\beta \p_\zeta \Theta (s) \|^2_{L^2_{l_r}} \d s + C(E_{k+1}) \int_0^t \|\p_t^\gamma \p_{\bar{x}}^\beta \p_\zeta \Theta (s)\|_{L^2_{l_r}}^2 \d s\,, \label{III_2}
\end{align}
where the properties of $\chi_\sigma (\zeta)$ in \eqref{chi-property} and the similar arguments of derivation $I_1$ in \eqref{I1} have been used. For the term $III_3$, it is decomposed as
\begin{align*}
  III_3 = & - \int_0^t \int_{\R^2} \int_0^{\frac{3}{\sigma}} \tilde{\kappa} (1+\zeta)^{l_r} ( \p_t^\gamma \p_{\bar{x}}^\beta \p_\zeta^2 \Theta )^2 + \tilde{\kappa} l_r (1+\zeta)^{l_r-1} \p_t^\gamma \p_{\bar{x}}^\beta \p_\zeta^2 \Theta \p_t^\gamma \p_{\bar{x}}^\beta \p_\zeta \Theta \d \bar{x} \d \zeta \d s \\
  & - \int_0^t \int_{\R^2} \int_0^{\frac{3}{\sigma}} [\p_t^\gamma \p_{\bar{x}}^\beta, \tilde{\kappa} \p_\zeta^2] \Theta \cdot [(1+\zeta)^{l_r} + l_r (1+\zeta)^{l_r -1}] \p_t^\gamma \p_{\bar{x}}^\beta \p_\zeta \Theta \\
  & + \int_0^t \int_{\R^2} (1+\zeta)^{l_r} \p_t^\gamma \p_{\bar{x}}^\beta (\tilde{\kappa} \p_\zeta^2 \Theta) \p_t^\gamma \p_{\bar{x}}^\beta \p_\zeta \Theta |_{\zeta=0}^{\zeta=\frac{3}{\sigma}} \d \bar{x} \d s \,.
\end{align*}
By the lower bounds \eqref{LowBnd-2}, H\"older and Young inequalities, the first term in the right-hand side of the decomposition $III_3$ can be bounded by
\begin{equation*}
  \begin{aligned}
    - \tfrac{3}{4} \int_0^t \tilde{\kappa}_0 \|\p_\zeta^2 \p_t^\gamma \p_{\bar{x}}^\beta \Theta (s)\|_{L^2_{l_r}}^2 \d s + C(E_{k+1}) \int_0^t \|\p_t^\gamma \p_{\bar{x}}^\beta \p_\zeta \Theta(s)\|_{L^2_{l_r}}^2 \d s \,,
  \end{aligned}
\end{equation*}
and the second term can be bounded by $ C(E_{k+1}) \int_0^t \|\p_\zeta \Theta(s) \|_{\mathbb{H}^{r-1}_{l,1}(\R^3_+)}^2 \d s $. We further denote the third term by $III_{3B}$. There therefore holds
\begin{equation}\label{III_3_3B}
  \begin{aligned}
    III_3 \leq & III_{3B} - \tfrac{3}{4} \int_0^t \tilde{\kappa}_0 \|\p_\zeta^2 \p_t^\gamma \p_{\bar{x}}^\beta \Theta (s)\|_{L^2_{l_r}}^2 \d s + C(E_{k+1}) \int_0^t \|\p_t^\gamma \p_{\bar{x}}^\beta \p_\zeta \Theta(s)\|_{L^2_{l_r}}^2 \d s \\
    & +  C(E_{k+1}) \int_0^t \|\p_\zeta \Theta(s) \|_{\mathbb{H}^{r-1}_{l,1}(\R^3_+)}^2 \d s \,.
  \end{aligned}
\end{equation}
Due to the boundary values of \eqref{Theta_Appro} and \eqref{Theta_BC_2_Order}, the quantity $III_{3B}$ reads
\begin{align}\label{III_3B_Split}
   \no III_{3B} = & \underbrace{-\int_0^t \int_{\R^2} \p_t^\gamma \p_{\bar{x}}^\beta (\p_t - \lambda \Delta_{\bar{x}}) \Theta \cdot \p_t^\gamma \p_{\bar{x}}^\beta (R_\theta \Theta)|_{\zeta=0} \d \bar{x} \d s}_{III_{3B}^1} \\
   \no & \underbrace{- \int_0^t \int_{\R^2} \p_t^\gamma \p_{\bar{x}}^\beta (\bar{\u}^0 \cdot \nabla_{\bar{x}} \Theta) \cdot \p_t^\gamma \p_{\bar{x}}^\beta (R_\theta \Theta)|_{\zeta=0} \d \bar{x} \d s}_{III_{3B}^2} \\
   & \underbrace{- \int_0^t \int_{\R^2} \p_t^\gamma \p_{\bar{x}}^\beta \big( \tfrac{2}{3} \div_x \u^0 \Theta \big) \cdot \p_t^\gamma \p_{\bar{x}}^\beta (R_\theta \Theta)|_{\zeta=0} \d \bar{x} \d s}_{III_{3B}^3} \\
   \no & + \underbrace{\int_0^t \int_{\R^2} \p_t^\gamma \p_{\bar{x}}^\beta \tilde{g}^\sigma \cdot \p_t^\gamma \p_{\bar{x}}^\beta (R_\theta \Theta)|_{\zeta=0} \d \bar{x} \d s}_{III_{3B}^4}\,.
\end{align}

We first deal with the quantity $III_{3B}^1$. We deduce from integration by parts over $[0,t] \times \R^2 $ that
\begin{align}\label{III_3B_1_Split}
  \no III_{3B}^1 & = - \int_{\R^2}\! \tfrac{1}{2} R_\theta \big(\p_t^\gamma \p_{\bar{x}}^\beta \Theta\big)^2 |_{\zeta=0} |_{s=0}^{s=t} \d \bar{x} - \int_0^t \int_{\R^2} \lambda R_\theta |\nabla_{\bar{x}} \p_t^\gamma \p_{\bar{x}}^\beta \Theta|^2 |_{\zeta=0} \d \bar{x} \d s\\
  \no & + \int_0^t\! \int_{\R^2} \!\tfrac{1}{2} (\p_t + \lambda \Delta_{\bar{x}} ) R_\theta (\p_t^\gamma \p_{\bar{x}}^\beta \Theta)^2 |_{\zeta=0} \d \bar{x} \d s \\
  \no & \underbrace{- \int_0^t \int_{\R^2} \p_t^\gamma \p_{\bar{x}}^\beta \p_t \Theta \cdot [\p_t^\gamma \p_{\bar{x}}^\beta, R_\theta] \Theta |_{\zeta=0} \d \bar{x} \d s}_{III_{3B}^{11}} + \underbrace{\lambda \int_0^t \int_{\R^2} \p_t^\gamma \p_{\bar{x}}^\beta \Delta_{\bar{x}} \Theta \cdot [\p_t^\gamma \p_{\bar{x}}^\beta, R_\theta] \Theta |_{\zeta=0} \d \bar{x} \d s}_{III_{3B}^{12}}\\
  \no & \le - \tfrac{R^0_\theta}{2} \|\p_t^\gamma \p_{\bar{x}}^\beta \Theta |_{\zeta=0} (t)\|_{L^2(\R^2)}^2 - \int_0^t \lambda R^0_\theta \|\nabla_{\bar{x}} \p_t^\gamma \p_{\bar{x}}^\beta \Theta |_{\zeta=0} (s) \|_{L^2(\R^2)}^2 \d s \\
  &  + C(E_{k+1}) \Big( \| \Theta (0)\|_{\mathbb{H}^{r}_l(\R^3_+)}^2 + \int_0^t \| \p_t^\gamma \p_{\bar{x}}^\beta \Theta|_{\zeta=0} (s) \|^2_{L^2(\R^2)} \d s \Big) + III_{3B}^{11} + III_{3B}^{12} \,,
\end{align}
where the lower bounds \eqref{LowBnd-1} and the inequality
 $$|- \int_{\R^2} \tfrac{1}{2} R_\theta (\p_t^\gamma \p_{\bar{x}}^\beta \Theta)^2 (0) |_{\zeta=0} \d \bar{x} | \le C(E_{k+1}) \| \Theta (0)\|_{\mathbb{H}^{r}_l(\R^3_+)}^2 $$ implied by \eqref{Trace-1} in Lemma \ref{Lemma_Trace} have been used in the last inequality. For the term $III_{3B}^{11}$, if $\beta =0, 2\gamma = r-1$,
\begin{align*}
  III_{3B}^{11} = & - \int_{\R^2} \p_t^\gamma \Theta \cdot [\p_t^\gamma, R_\theta] \Theta |_{\zeta=0} |_{s=0}^{s=t} \d \bar{x} - \int_0^t \int_{\R^2} \p_t^\gamma \Theta \cdot \p_t [\p_t^\gamma, R_\theta] \Theta |_{\zeta=0} \d \bar{x} \d s \\
  \le & \tfrac{R^0_\theta}{4} \|\p_t^\gamma \p_{\bar{x}}^\beta \Theta |_{\zeta=0} (t)\|_{L^2(\R^2)}^2 + C(E_{k+1}) \| \Theta |_{\zeta=0} (t) \|^2_{\mathbb{H}^{r-3}(\R^2)} \\
  & + C(E_{k+1}) \Big( \| \Theta (0) \|^2_{\mathbb{H}^r_l(\R^3_+)} + \int_0^t \| \Theta (s) \|^2_{\mathbb{H}^r_{l,0} (\R^3_+)} + \| \Theta (s) \|^2_{\mathbb{H}^r_{l,1} (\R^3_+)} \d s \Big) \,,
\end{align*}
where the lower bounds \eqref{LowBnd-1} and trace inequality \eqref{Trace-1} in Lemma \ref{Lemma_Trace} are used.

If $|\beta| \ge 2$, without loss of generality, we assume $\beta \ge 2 e_1 = 2(1,0)$, i.e., $\beta_1 \ge 2$. Then, \eqref{Trace-2} in Lemma \ref{Lemma_Trace} implies
\begin{align*}
  III_{3B}^{11} = & \int_0^t \int_{\R^2} \p_{x_1} \big( \p_t^\gamma \p_{\bar{x}}^{\beta - 2e_1} \p_t \Theta \big) \cdot \p_{x_1} \big\{ [\p_t^\gamma \p_{\bar{x}}^\beta, R_\theta] \Theta \big\} |_{\zeta=0} \d \bar{x} \d s\\
  \le & C(E_{k+1}) \int_0^t \| \Theta (s) \|^2_{\mathbb{H}^r_{l,0} (\R^3_+)} + \| \Theta (s) \|^2_{\mathbb{H}^r_{l,1} (\R^3_+)} \d s \,.
\end{align*}
If $|\beta|=1, 2 \gamma = r-2$, where $r$ is required to be even. Then,
\begin{align*}
  III_{3B}^{11} = & - \sum_{i=1}^2 \delta_r (\beta=e_i) \int_0^t \int_{\R^2} \p_t^\gamma \p_{\bar{x}}^{e_i} \p_t \Theta \p_{\bar{x}}^{e_i} R_\theta \p_t^{\gamma} \Theta |_{\zeta=0} \d \bar{x} \d s \\
  & + \int_0^t \int_{\R^2} \p_t^\gamma \p_{\bar{x}}^\beta \Theta \cdot \p_t \p_{\bar{x}}^\beta [ \p_t^{\gamma} , R_\theta ] \Theta |_{\zeta=0} \d \bar{x} \d s - \int_{\R^2} \p_t^\gamma \p_{\bar{x}}^\beta \Theta \cdot \p_{\bar{x}}^\beta [ \p_t^{\gamma}, R_\theta ] \Theta |_{\zeta=0} |_{s=0}^{s=t} \d \bar{x} \\
  \le & \sum_{i=1}^2 \delta_r (\beta=e_i) \int_0^t \int_{\R^2} \p_t^\gamma \p_t \Theta \p_{\bar{x}}^{e_i} ( \p_{\bar{x}}^{e_i} R_\theta \p_t^{\gamma} \Theta ) |_{\zeta=0} \d \bar{x} \d s \\
  & + C(E_{k+1}) \sum_{j=0}^{r-1} \sum_{2\tilde{\gamma} + |\tilde{\beta}|=j} \int_0^t \|\p_t^{\tilde{\gamma}} \p_{\bar{x}}^{\tilde{\beta}} \Theta |_{\zeta=0} (s)\|_{L^2 (\R^2)}^2 \d s \\
  & + \tfrac{R^0_\theta}{8} \|\p_t^\gamma \p_{\bar{x}}^\beta \Theta |_{\zeta=0} (t)\|_{L^2(\R^2)}^2 + C(E_{k+1}) \sum_{j=0}^{r-3} \sum_{2\tilde{\gamma} + |\tilde{\beta}|=j} \| \p_t^{\tilde{\gamma}} \p_{\bar{x}}^{\tilde{\beta}} \Theta |_{\zeta=0} (t)\|_{L^(\R^2)}^2\\
  & + C(E_{k+1}) \sum_{j=0}^{r-1} \sum_{2\tilde{\gamma} + |\tilde{\beta}|=j} \| \p_t^{\tilde{\gamma}} \p_{\bar{x}}^{\tilde{\beta}} \Theta |_{\zeta=0} (0)\|_{L^(\R^2)}^2 \\
  \le & \tfrac{R^0_\theta}{8} \|\p_t^\gamma \p_{\bar{x}}^\beta \Theta |_{\zeta=0} (t)\|_{L^2(\R^2)}^2 + C(E_{k+1}) \| \Theta |_{\zeta=0} (t) \|^2_{\mathbb{H}^{r-3}(\R^2)} \\
  & + \epsilon_0 \sum_{i=1}^2 \delta_r (\beta=e_i) \int_0^t \| \p_t^{\gamma+1} \Theta |_{\zeta=0} (s) \|^2_{L^2(\R^2)} \d s \\
  & + C(E_{k+1}) \Big( \| \Theta (0) \|^2_{\mathbb{H}^r_l(\R^3_+)} + \int_0^t \| \Theta (s) \|^2_{\mathbb{H}^r_{l,0}(\R^3_+)} + \| \Theta (s) \|^2_{\mathbb{H}^r_{l,1}(\R^3_+)} \d s \Big)
\end{align*}
for some small $\epsilon_0 > 0$ to be determined, where \eqref{Trace-1} is used and
\begin{equation}\label{delta-r-beta}
\delta_r (\beta=e_i) =
 \left\{
  \begin{array}{l}
    1\,, \qquad r\, {\text{is\ even\ and}}\, \beta = e_i\,,\\
    0\,,\qquad {\text{otherwise}}\,.
  \end{array}
 \right.
\end{equation}
Consequently,
\begin{align} \label{Es_III_3B_11}
  \no III_{3B}^{11} \le & \tfrac{3 R^0_\theta}{8} \|\p_t^\gamma \p_{\bar{x}}^\beta \Theta |_{\zeta=0} (t)\|_{L^2(\R^2)}^2 + C(E_{k+1}) \| \Theta |_{\zeta=0} (t) \|^2_{\mathbb{H}^{r-3}(\R^2)} \\
  & + \epsilon_0 \sum_{i=1}^2 \delta_r (\beta=e_i) \int_0^t \| \p_t^{\gamma+1} \Theta |_{\zeta=0} (s) \|^2_{L^2(\R^2)} \d s \\
  \no & + C(E_{k+1}) \Big( \| \Theta (0) \|^2_{\mathbb{H}^r_l(\R^3_+)} + \int_0^t \| \Theta (s) \|^2_{\mathbb{H}^r_{l,0}(\R^3_+)} + \| \Theta (s) \|^2_{\mathbb{H}^r_{l,1}(\R^3_+)} \d s \Big) \,.
\end{align}
For $III_{3B}^{12}$, as similarly in \eqref{Es_III_3B_11}, one has
\begin{align}\label{III_2B_12}
  III_{3B}^{12} & = -\lambda \int_0^t \int_{\R^2} \nabla_{\bar{x}} \p_t^\gamma \p_{\bar{x}}^\beta \Theta \cdot \nabla_{\bar{x}} [\p_t^\gamma \p_{\bar{x}}^\beta, R_\theta] \Theta |_{\zeta=0} \d \bar{x} \d s \\
  \no \le & \tfrac{1}{2} \lambda R^0_\theta \int_0^t \|\nabla_{\bar{x}} \p_t^\gamma \p_{\bar{x}}^\beta \Theta |_{\zeta=0} (s) \|_{L^2(\R^2)}^2 \d s + C(E_{k+1}) \int_0^t \| \Theta (s) \|^2_{\mathbb{H}^r_{l,0}(\R^3_+)} + \| \Theta (s) \|^2_{\mathbb{H}^r_{l,1}(\R^3_+)} \d s \,.
\end{align}
Therefore, from plugging \eqref{Es_III_3B_11} and \eqref{III_2B_12} into \eqref{III_3B_1_Split}, one has
\begin{align}\label{Es_III_3B_1}
 \no III_{3B}^1 \le & - \tfrac{R^0_\theta}{8} \|\p_t^\gamma \p_{\bar{x}}^\beta \Theta |_{\zeta=0} (t)\|_{L^2(\R^2)}^2 - \tfrac{1}{2} \lambda \int_0^t R^0_\theta \|\nabla_{\bar{x}} \p_t^\gamma \p_{\bar{x}}^\beta \Theta |_{\zeta=0} (s) \|_{L^2(\R^2)}^2 \d s\\
 \no & + \epsilon_0 \sum_{i=1}^2 \delta_r (\beta=e_i) \int_0^t \| \p_t^{\gamma+1} \Theta |_{\zeta=0} (s) \|^2_{L^2(\R^2)} \d s + C(E_{k+1}) \| \Theta |_{\zeta=0} (t) \|^2_{\mathbb{H}^{r-3}(\R^2)} \\
 & + C(E_{k+1}) \Big( \| \Theta (0) \|^2_{\mathbb{H}^r_l(\R^3_+)} + \int_0^t \| \Theta (s) \|^2_{\mathbb{H}^r_{l,0}(\R^3_+)} + \| \Theta (s) \|^2_{\mathbb{H}^r_{l,1}(\R^3_+)} \d s \Big) \,.
\end{align}

For the term $III_{3B}^{2}$, there holds
\begin{align}\label{III_3B_2}
  \no III_{3B}^2 = & -\int_0^t \int_{\R^2} [\p_t^\gamma \p_{\bar{x}}^\beta, \bar{\u}^0 \cdot \nabla_{\bar{x}} ]\Theta \cdot \p_t^\gamma \p_{\bar{x}}^\beta (R_\theta \Theta) |_{\zeta=0} \d \bar{x} \d s \\
  \no + \int_0^t \int_{\R^2} & \p_t^{\tilde{\gamma}} \p_{\bar{x}}^{\tilde{\beta}} \Theta \div_{\bar{x}} \big\{ \bar{\u}^0 [\p_t^\gamma \p_{\bar{x}}^\beta, R_\theta] \Theta \big\}|_{\zeta=0} \d \bar{x} \d s + \int_0^t \int_{\R^2} \div_{\bar{x}} (R_\theta \bar{\u}^0) \tfrac{1}{2} \big( \p_t^\gamma \p_{\bar{x}}^\beta \Theta\big)^2|_{\zeta=0} \d \bar{x} \d s\\
  \le & C(E_{k+1}) \int_0^t \| \Theta (s) \|^2_{\mathbb{H}^r_{l,0}(\R^3_+)} + \| \Theta (s) \|^2_{\mathbb{H}^r_{l,1}(\R^3_+)} \d s \,.
\end{align}
Moreover, the trace inequality \eqref{Trace-1} in Lemma \ref{Lemma_Trace} also implies
\begin{align}\label{III_3B_3}
  III_{3B}^3 \le & C(E_{k+1}) \int_0^t \| \Theta (s) \|^2_{\mathbb{H}^r_{l,0}(\R^3_+)} + \| \Theta (s) \|^2_{\mathbb{H}^r_{l,1}(\R^3_+)} \d s \,.
\end{align}
and
\begin{align}\label{III_3B_4}
  III_{3B}^4 \le C(E_{k+1}) \Big( \int_0^t \| ( \Theta, \tilde{g}^\sigma ) (s)\|_{\mathbb{H}^r_{l,0}(\R^3_+)}^2 + \| ( \Theta, \tilde{g}^\sigma ) (s)\|_{\mathbb{H}^r_{l,1}(\R^3_+)}^2 \d s\Big)\,.
\end{align}
It is thereby derived from substituting the bounds \eqref{Es_III_3B_1}-\eqref{III_3B_4} into \eqref{III_3B_Split}, and combining with \eqref{III_3_3B} that
\begin{align}\label{III_3}
 \no III_3 \le & - \tfrac{3}{4} \int_0^t \tilde{\kappa}_0 \|\p_\zeta \p_t^\gamma \p_{\bar{x}}^\beta \p_\zeta \Theta (s)\|_{L^2_{l_r}}^2 \d s +  C(E_{k+1}) \int_0^t \|\p_\zeta \Theta(s) \|_{\mathbb{H}^{r-1}_{l,1}(\R^3_+)}^2 \d s\\
 \no & - \tfrac{R^0_\theta}{8} \|\p_t^\gamma \p_{\bar{x}}^\beta \Theta|_{\zeta=0} (t)\|_{L^2(\R^2)}^2 - \tfrac{1}{2} \lambda R^0_\theta \int_0^t \|\nabla_{\bar{x}} \p_t^\gamma \p_{\bar{x}}^\beta \Theta|_{\zeta=0}(s)\|_{L^2(\R^2)}^2 \d s \\
 \no & + \epsilon_0 \sum_{i=1}^2 \delta_r (\beta=e_i) \int_0^t \| \p_t^{\gamma+1} \Theta |_{\zeta=0} (s) \|^2_{L^2(\R^2)} \d s + C(E_{k+1}) \| \Theta |_{\zeta=0} (t) \|^2_{\mathbb{H}^{r-3}(\R^2)} \\
 & + C(E_{k+1}) \Big( \| \Theta (0) \|^2_{\mathbb{H}^r_l(\R^3_+)} + \int_0^t \| ( \Theta, \tilde{g}^\sigma ) (s)\|_{\mathbb{H}^r_{l,0}(\R^3_+)}^2 + \| ( \Theta, \tilde{g}^\sigma ) (s)\|_{\mathbb{H}^r_{l,1}(\R^3_+)}^2 \d s\Big) \,.
\end{align}
Direct calculation gives us
\begin{align}\label{III_4}
  III_4 = -\lambda \int_0^t \|\nabla_{\bar{x}} \p_t^\gamma \p_{\bar{x}}^\beta \p_\zeta \Theta (s)\|_{L^2_{l_r}}^2 \d s\,.
\end{align}
For $III_5$ and $III_6$, one has
\begin{align}\label{III_5}
 \no III_5 \le & C(E_{k+1}) \sum_{j=0}^{r-2} \sum_{2\tilde{\gamma} + |\tilde{\beta}| = j} \int_0^t \int_{\R^2} \int_0^{\frac{3}{\sigma}} (1+\zeta)^{l_r+1} |\p_t^{\tilde{\gamma}} \p_{\bar{x}}^{\tilde{\beta}} \p_\zeta^2 \Theta| |\p_t^\gamma \p_{\bar{x}}^\beta \p_\zeta \Theta| \d \bar{x} \d \zeta \d s\\
 \no \le & C(E_{k+1}) \sum_{j=0}^{r-2} \sum_{2\tilde{\gamma} + |\tilde{\beta}| = j} \Big( \int_0^t \int_{\R^2} \int_0^{\frac{3}{\sigma}} (1+\zeta)^{l_r+2} |\p_\zeta \p_t^{\tilde{\gamma}} \p_{\bar{x}}^{\tilde{\beta}} \p_\zeta \Theta |^2 \d \bar{x} \d \zeta \d s \Big)^{\frac{1}{2}} \\
 \no & \times \Big( \int_0^t \int_{\R^2} \int_0^{\frac{3}{\sigma}} (1+\zeta)^{l_r} | \p_t^{\gamma} \p_{\bar{x}}^{\beta} \p_\zeta \Theta |^2 \d \bar{x} \d \zeta \d s \Big)^{\frac{1}{2}}\\
 \le &  C(E_{k+1}) \int_0^t \|\p_t^{\gamma} \p_{\bar{x}}^{\beta} \p_\zeta \Theta (s)\|_{L^2_{l_r}}^2 \d s + C(E_{k+1}) \int_0^t \| \p_{\zeta} \Theta (s) \|^2_{\mathbb{H}^{r-1}_{l,1}(\R^3_+)} \d s\,,
\end{align}
and
\begin{align}\label{III_6}
 III_6 \le & C(E_{k+1}) \int_0^t \|\p_t^{\gamma} \p_{\bar{x}}^{\beta} \p_\zeta \Theta(s)\|_{L^2_{l_r}}^2 \d s\,.
\end{align}
Similarly, one also has
\begin{align}\label{III_789}
   III_7 + III_8 + III_9 \le C(E_{k+1}) \int_0^t \| \Theta (s) \|^2_{\mathbb{H}^r_{l,1} (\R^3_+)} \d s \,,
\end{align}
and
\begin{align}\label{III_10}
 III_{10} + III_{11} \le C \int_0^t \|\p_t^{\gamma} \p_{\bar{x}}^{\beta} \p_\zeta (\tilde{g}^\sigma, \Theta) (s)\|_{L^2_{l_r}}^2 \d s + C(E_{k+1}) \int_0^t \| \p_\zeta \Theta (s) \|^2_{\mathbb{H}^{r-1}_{l,1} (\R^3_+)} \,.
\end{align}
In summary, it is derived from \eqref{III_1}-\eqref{III_2}, \eqref{III_3}-\eqref{III_10} and summing up for $2 \gamma + |\beta| = r - 1$ that
\begin{align}\label{Theta_1zeta}
   \no & \| \Theta(s)\|_{l,r,1}^2 + \int_0^t \tilde{\kappa}_0 \|\p_\zeta \Theta(s)\|_{l,r,1}^2 + \lambda \|\nabla_{\bar{x}} \Theta(s)\|_{l,r,1}^2 \d s \\
   \no & + \tfrac{R^0_\theta}{4} \| \Theta|_{\zeta=0} (t)\|_{\Gamma, r-1}^2 + \lambda R^0_\theta \int_0^t \|\nabla_{\bar{x}} \Theta |_{\zeta=0} (s) \|_{\Gamma, r-1}^2 \d s \\
   \no & \le C(E_{k+1}) \int_0^t \|\p_\zeta \Theta(s) \|_{\mathbb{H}^{r-1}_{l,1}(\R^3_+)}^2 \d s + C(E_{k+1}) \| \Theta |_{\zeta=0} (t) \|^2_{\mathbb{H}^{r-3}(\R^2)}\\
   & + \epsilon_0 \sum_{2 \gamma + |\beta| = r - 1} \sum_{i=1}^2 \delta_r (\beta=e_i) \int_0^t \| \p_t^{\gamma+1} \Theta |_{\zeta=0} (s) \|^2_{L^2(\R^2)} \d s  \\
   \no & + C(E_{k+1}) \Big( \| \Theta (0) \|^2_{\mathbb{H}^r_l(\R^3_+)} + \int_0^t \| ( \Theta, \tilde{g}^\sigma ) (s)\|_{\mathbb{H}^r_{l,0}(\R^3_+)}^2 + \| ( \Theta, \tilde{g}^\sigma ) (s)\|_{\mathbb{H}^r_{l,1}(\R^3_+)}^2 \d s\Big)
\end{align}
for all $1 \leq r \leq k (k \geq 3)$ and some undetermined small constant $\epsilon_0 > 0$. Together with \eqref{Theta_L_2} and \eqref{Theta_1zeta}, the induction for $1 \leq r \leq k$ implies that
  \begin{align}\label{Theta_1zeta_Dert}
    \no & \| \Theta(s)\|_{\mathbb{H}^r_{l,1}(\R^3_+)}^2 + \int_0^t \|\p_\zeta \Theta(s)\|_{\mathbb{H}^r_{l,1}(\R^3_+)}^2 + \lambda \|\nabla_{\bar{x}} \Theta(s)\|_{\mathbb{H}^r_{l,1}(\R^3_+)}^2 \d s \\
    \no & + \| \Theta|_{\zeta=0} (t)\|_{\mathbb{H}^{r-1}(\R^2)}^2 + \lambda \int_0^t \|\nabla_{\bar{x}} \Theta |_{\zeta=0} (s) \|_{\mathbb{H}^{r-1}(\R^2)}^2 \d s \\
    \no \le & C(E_{k+1}) \epsilon_0 \underbrace{ \sum_{j=0}^r \sum_{2 \gamma + |\beta| = j - 1} \sum_{i=1}^2 \delta_j (\beta=e_i) \int_0^t \| \p_t^{\gamma+1} \Theta |_{\zeta=0} (s) \|^2_{L^2(\R^2)} \d s }_{= : \mathscr{A}_{B,1}^r (\Theta) (t)}  \\
    & + C(E_{k+1}) \Big( \| \Theta (0) \|^2_{\mathbb{H}^r_l(\R^3_+)} + \int_0^t \| ( \Theta, \tilde{g}^\sigma ) (s)\|_{\mathbb{H}^r_{l,0}(\R^3_+)}^2 + \| ( \Theta, \tilde{g}^\sigma ) (s)\|_{\mathbb{H}^r_{l,1}(\R^3_+)}^2 \d s\Big)
  \end{align}
for some small $\epsilon_0 > 0$ to determined.

{\bf Case 2. $n$-order $\zeta$-derivatives: $\mathbb{H}^r_{l,n}(\R^3_+)$-bounds ($2 \leq n \leq r$).} For $2 \le n \le r$ with $2 \gamma + |\beta| = r-n$, applying $\p_t^\gamma \p_{\bar{x}}^\beta \p_\zeta^n$ to the $\Theta$-equation of \eqref{Theta_Appro} infers that
\begin{equation}\label{Theta_t_x_High_0}
 \begin{aligned}
  & \p_t \p_t^\gamma \p_{\bar{x}}^\beta \p_\zeta^n \Theta + \bar{\u}^0 \cdot \nabla_{\bar{x}} \p_t^\gamma \p_{\bar{x}}^\beta \p_\zeta^n \Theta + [\p_t^\gamma \p_{\bar{x}}^\beta, \bar{\u}^0 \cdot \nabla_{\bar{x}}] \p_\zeta^n \Theta + \p_{x_3} \u^0_3 [\p_\zeta^n, \zeta \chi_\sigma (\zeta)] \p_\zeta \p_t^\gamma \p_{\bar{x}}^\beta \Theta \\
  & + ( \p_{x_3} \u^0_3 \zeta + u^0_{1,3} ) \chi_\sigma (\zeta) \p_\zeta \p_t^\gamma \p_{\bar{x}}^\beta \p_\zeta^n \Theta + \p_\zeta^n \big\{ \zeta \chi_\sigma(\zeta) [\p_t^\gamma \p_{\bar{x}}^\beta, \p_{x_3} \u^0_3] \p_\zeta \Theta \big\} \\
  =& \p_t^\gamma \p_{\bar{x}}^\beta (\tilde{\kappa} \p_\zeta^{n+2} \Theta) + \lambda \Delta_{\bar{x}} \p_t^\gamma \p_{\bar{x}}^\beta \p_\zeta^n \Theta + \p_t^\gamma \p_{\bar{x}}^\beta \p_\zeta^n \tilde{g}^\sigma \\
  & - \p_t^\gamma \p_{\bar{x}}^\beta \p_\zeta^n (\tfrac{2}{3} \div_x \u^0 \Theta) - [ \p_t^\gamma \p_{\bar{x}}^\beta \p_\zeta^n, u^0_{1,3} \chi_\sigma (\zeta) \p_\zeta ] \Theta \,.
 \end{aligned}
\end{equation}
From taking inner product with $(1+\gamma)^{l_r} \p_\gamma^{\gamma} \p_{\bar{x}}^{\beta} \p_\zeta^n \Theta$ over $[0,t] \times \R^2 \times [0,\tfrac{3}{\sigma}]$ in \eqref{Theta_t_x_High_0}, one obtains
\begin{align*}
  & \tfrac{1}{2} \|\p_t^\gamma \p_{\bar{x}}^\beta \p_\zeta^n \Theta (t)\|_{L^2_{l_r}}^2 - \tfrac{1}{2} \|\p_t^\gamma \p_{\bar{x}}^\beta \p_\zeta^n \Theta (0)\|_{L^2_{l_r}}^2 = \sum_{k=1}^{10} IV_k \,,
\end{align*}
where
\begin{align*}
  IV_1 = & - \int_0^t \int_{\R^2} \int_0^{\frac{3}{\sigma}} \bar{\u}^0 \cdot \nabla_{\bar{x}} \p_t^\gamma \p_{\bar{x}}^\beta \p_\zeta^n \Theta \cdot (1+\zeta)^{l_r} \p_t^\gamma \p_{\bar{x}}^\beta \p_\zeta^n \Theta \d \bar{x} \d \zeta \d s \,, \\
  IV_2 = & - \int_0^t \int_{\R^2} \int_0^{\frac{3}{\sigma}} [\p_t^\gamma \p_{\bar{x}}^\beta, \bar{\u}^0 \cdot \nabla_{\bar{x}}] \p_\zeta^n \Theta \cdot  (1+\zeta)^{l_r} \p_t^\gamma \p_{\bar{x}}^\beta \p_\zeta^n \Theta \d \bar{x} \d \zeta \d s \,, \\
  IV_3 = & - \int_0^t \int_{\R^2} \int_0^{\frac{3}{\sigma}} ( \p_{x_3} \u^0_3 \zeta + u^0_{1,3} ) \chi_\sigma (\zeta) \p_\zeta \p_t^\gamma \p_{\bar{x}}^\beta \p_\zeta^n \Theta \cdot (1+\zeta)^{l_r} \p_t^\gamma \p_{\bar{x}}^\beta \p_\zeta^n \Theta \d \bar{x} \d \zeta \d s \,, \\
  IV_4 = & - \int_0^t \int_{\R^2} \int_0^{\frac{3}{\sigma}} \p_{x_3} \u^0_3 [\p_\zeta^n, \zeta \chi_\sigma (\zeta)] \p_\zeta \p_t^\gamma \p_{\bar{x}}^\beta \Theta \cdot (1+\zeta)^{l_r} \p_t^\gamma \p_{\bar{x}}^\beta \p_\zeta^n \Theta \d \bar{x} \d \zeta \d s \,, \\
  IV_5 = & - \int_0^t \int_{\R^2} \int_0^{\frac{3}{\sigma}} \p_t^\gamma \p_{\bar{x}}^\beta \p_\zeta^n (\tfrac{2}{3} \div_x \u^0 \Theta) \cdot (1+\zeta)^{l_r} \p_t^\gamma \p_{\bar{x}}^\beta \p_\zeta^n \Theta \d \bar{x} \d \zeta \d s \,, \\
  IV_6 = & - \int_0^t \int_{\R^2} \int_0^{\frac{3}{\sigma}} \p_\zeta^n \big\{ \zeta \chi_\sigma(\zeta) [\p_t^\gamma \p_{\bar{x}}^\beta, \p_{x_3} \u^0_3] \p_\zeta \Theta \big\} \cdot (1+\zeta)^{l_r} \p_t^\gamma \p_{\bar{x}}^\beta \p_\zeta^n \Theta \d \bar{x} \d \zeta \d s \,, \\
  IV_7 = & \int_0^t \int_{\R^2} \int_0^{\frac{3}{\sigma}}\p_t^\gamma \p_{\bar{x}}^\beta (\tilde{\kappa} \p_\zeta^{n+2} \Theta)  \cdot (1+\zeta)^{l_r} \p_t^\gamma \p_{\bar{x}}^\beta \p_\zeta^n \Theta \d \bar{x} \d \zeta \d s \,, \\
  IV_8 = & \int_0^t \int_{\R^2} \int_0^{\frac{3}{\sigma}} \lambda \Delta_{\bar{x}} \p_t^\gamma \p_{\bar{x}}^\beta \p_\zeta^n \Theta  \cdot (1+\zeta)^{l_r} \p_t^\gamma \p_{\bar{x}}^\beta \p_\zeta^n \Theta \d \bar{x} \d \zeta \d s \,, \\
  IV_9 = & \int_0^t \int_{\R^2} \int_0^{\frac{3}{\sigma}} \p_t^\gamma \p_{\bar{x}}^\beta \p_\zeta^n \tilde{g}^\sigma  \cdot (1+\zeta)^{l_r} \p_t^\gamma \p_{\bar{x}}^\beta \p_\zeta^n \Theta \d \bar{x} \d \zeta \d s \,, \\
  IV_{10} = & - \int_0^t \int_{\R^2} \int_0^{\frac{3}{\sigma}} [ \p_t^\gamma \p_{\bar{x}}^\beta \p_\zeta^n, u^0_{1,3} \chi_\sigma (\zeta) \p_\zeta ] \Theta \cdot (1+\zeta)^{l_r} \p_t^\gamma \p_{\bar{x}}^\beta \p_\zeta^n \Theta \d \bar{x} \d \zeta \d s \,.
\end{align*}
It is easy to estimate that
\begin{equation*}
 \begin{aligned}
  |IV_1| + |IV_2| \le C(E_{k+1}) \int_0^t \| \Theta (s) \|_{l,r,n}^2 \d s\,.
 \end{aligned}
\end{equation*}
Together with the properties of $\chi_\sigma (\zeta)$ in \eqref{chi-property},
\begin{equation*}
 \begin{aligned}
  IV_3 \le \tfrac{1}{4} \int_0^t \tilde{\kappa}_0 \| \partial_\zeta \p_t^\gamma \p_{\bar{x}}^\beta \p_\zeta^n \Theta (s) \|^2_{L^2_{l_r}} \d s + C(E_{k+1}) \int_0^t \|\p_t^\gamma \p_{\bar{x}}^\beta \p_\zeta^n \Theta (s)\|_{L^2_{l_r}}^2 \d s \,,
 \end{aligned}
\end{equation*}
and
\begin{equation*}
 \begin{aligned}
  |IV_4| \le & C(E_{k+1}) \sum_{\tilde{n}=0}^{n-1} \int_0^t (1+\zeta)^{l_r +1} |\p_t^\gamma \p_{\bar{x}}^\beta \p_\zeta^{\tilde{n}+1} \Theta| |\p_t^\gamma \p_{\bar{x}}^\beta \p_\zeta^n \Theta| \d s \\
  \le & C(E_{k+1}) \sum_{\tilde{n}=1}^n \int_0^t \| \Theta (s) \|^2_{\mathbb{H}^r_{l, \tilde{n}}(\R^3_+)} \d s \,.
 \end{aligned}
\end{equation*}
Moreover, we similarly have
\begin{align*}
 |IV_6| \le & C(E_{k+1}) \int_0^t \| \p_{\zeta} \Theta (s) \|^2_{\mathbb{H}^{r-1}_{l,n} (\R^3_+)} + C(E_{k+1}) \sum_{\tilde{n}=1}^n \int_0^t \| \Theta (s) \|^2_{\mathbb{H}^r_{l, \tilde{n}}(\R^3_+)} \d s \,,
\end{align*}
and
\begin{equation*}
\begin{aligned}
 & |IV_5| \le C(E_{k+1}) \int_0^t \| \Theta (s) \|^2_{\mathbb{H}^r_{l,n}(\R^3_+)} \d s \,, \quad IV_8 = -\lambda \int_0^t \|\nabla_{\bar{x}} \p_t^\gamma \p_{\bar{x}}^\beta \p_\zeta^n \Theta (s)\|_{L^2_{l_r}}^2 \d s\,, \\
 & |IV_9| + |IV_{10}| \le C \int_0^t \| \p_t^\gamma \p_{\bar{x}}^\beta \p_\zeta^n (\tilde{g}^\sigma, \Theta) (s)\|_{L^2_{l_r}}^2 \d s + C(E_{k+1}) \int_0^t \| \p_\zeta \Theta (s) \|^2_{\mathbb{H}^{r-1}_{l,n}(\R^3_+)} \d s \,.
\end{aligned}
\end{equation*}

It remains to control the term $IV_7$, which can be decomposed as follows:
\begin{align}\label{IV_7_Decomposed}
 \no IV_7 = & \underbrace{\int_0^t \int_{\R^2}  \p_t^\gamma \p_{\bar{x}}^\beta (\tilde{\kappa} \p_\zeta^{n+1} \Theta) (1+\zeta)^{l_r} \p_t^\gamma \p_{\bar{x}}^\beta \p_\zeta^n \Theta |_{\zeta=0}^{\zeta=\frac{3}{\sigma}} \d \bar{x} \d s}_{IV_{7B}} \\
 \no & - \int_0^t \int_{\R^2} \int_0^{\frac{3}{\sigma}} \tilde{\kappa} (1+\zeta)^{l_r} (\p_\zeta \p_t^\gamma \p_{\bar{x}}^\beta \p_\zeta^n \Theta)^2 \d \bar{x} \d \zeta \d s \\
 \no & - \int_0^t \int_{\R^2} \int_0^{\frac{3}{\sigma}} [\p_t^\gamma \p_{\bar{x}}^\beta, \tilde{\kappa} \p_\zeta^{n+1}] \Theta \cdot (1+\zeta)^{l_r} \p_\zeta \p_t^\gamma \p_{\bar{x}}^\beta \p_\zeta^n \Theta \d \bar{x} \d \zeta \d s \\
 & - \int_0^t \int_{\R^2} \int_0^{\frac{3}{\sigma}} \p_t^\gamma \p_{\bar{x}}^\beta (\tilde{\kappa} \p_\zeta^{n+1} \Theta) \cdot l_r (1+\zeta)^{l_r-1} \p_t^\gamma \p_{\bar{x}}^\beta \p_\zeta^n \Theta \d \bar{x} \d \zeta \d s \,.
\end{align}
Due to the lower bound \eqref{LowBnd-2}, the second term in the right-hand side of \eqref{IV_7_Decomposed} can be bounded by
\begin{equation*}
  \begin{aligned}
    - \int_0^t \tilde{\kappa}_0 \|\p_\zeta \p_t^\gamma \p_{\bar{x}}^\beta \p_\zeta^n \Theta\|_{L^2_{l_r}}^2 \d s \,,
  \end{aligned}
\end{equation*}
and the last two terms in that can be bounded by
\begin{equation*}
  \begin{aligned}
    \tfrac{1}{4} \int_0^t \tilde{\kappa}_0 \|\p_\zeta \p_t^\gamma \p_{\bar{x}}^\beta \p_\zeta^n \Theta\|_{L^2_{l_r}}^2 \d s + C(E_{k+1}) \int_0^t \| \p_{\zeta} \Theta (s) \|^2_{\mathbb{H}^{r-1}_{l,n}(\R^3_+)} + \| \Theta (s) \|^2_{\mathbb{H}^r_{l,n}(\R^3_+)} \d s \,.
  \end{aligned}
\end{equation*}
Together with the bound of $IV_{7B}$ in Lemma \ref{Lemma_IV7B} below, one has
  \begin{align*}
    IV_7 \le & - \tfrac{3}{4} \int_0^t \tilde{\kappa}_0 \|\p_\zeta \p_t^\gamma \p_{\bar{x}}^\beta \p_\zeta^n \Theta\|_{L^2_{l_r}}^2 \d s - \mathscr{E}_{B, n}^{\gamma, \beta} (\Theta) (t) + C(\tau, E_{k+1}) \sup_{t \in [0,\tau]} \| \tilde{g}^\sigma (t) \|^2_{\mathbb{H}^{r+1}_l (\R^3_+)} \\
    & + \epsilon_0 \mathscr{A}_{B,n} (\Theta) (t) + C(E_{k+1}) \int_0^t \| \p_{\zeta} \Theta (s) \|^2_{\mathbb{H}^{r-1}_{l,n}(\R^3_+)} + \| \Theta (s) \|^2_{\mathbb{H}^r_{l,n}(\R^3_+)} \d s \\
    & + C(E_{k+1}) \Big( \| \Theta (0) \|^2_{\mathbb{H}^r_l(\R^3_+)} + \sum_{p=0, 1} \| \Theta (t) \|^2_{\mathbb{H}^r_{l,p}(\R^3_+)} + \int_0^t \sum_{p=0, 1} \| \Theta (s) \|^2_{\mathbb{H}^r_{l,p}(\R^3_+)} \d s \Big)
  \end{align*}
  for some undetermined small constant $\epsilon_0 > 0$.

Collecting the above bounds of $IV_i (1 \le i \le 9)$ and summing up for $2 \gamma + |\beta| = r - n$ reduce to
\begin{equation*}
 \begin{aligned}
   & \| \Theta (t)\|_{l,r,n}^2 + \int_0^t \tilde{\kappa}_0 \|\p_\zeta \Theta (s)\|_{l,r,n}^2 + \lambda \|\nabla_{\bar{x}} \Theta (s)\|_{l,r,n}^2 \d s + \sum_{2\gamma + |\beta| + n =r} \mathscr{E}_{B, n}^{\gamma, \beta} (\Theta) (t) \\
   \le & C(E_{k+1}) \Big( \sum_{2\gamma +|\beta| + n = r} \epsilon_0 \mathscr{A}_{B,n} (\Theta) (t) + \sum_{p=0, 1} \| \Theta (t) \|^2_{\mathbb{H}^r_{l,p}(\R^3_+)} + \int_0^t \| \p_{\zeta} \Theta (s) \|^2_{\mathbb{H}^{r-1}_{l,n}(\R^3_+)} \Big) \\
   & + C(\tau, E_{k+1}) \Big( \| \Theta (0) \|^2_{\mathbb{H}^r_l (\R^3_+)} + \int_0^t \sum_{0 \leq \tilde{n} \leq n} \| \Theta (s) \|^2_{\mathbb{H}^r_{l,\tilde{n}}(\R^3_+)} \d s + \sup_{t \in [0,\tau]} \| \tilde{g}^\sigma (t) \|^2_{\mathbb{H}^{r+1}_l(\R^3_+)} \Big)
 \end{aligned}
\end{equation*}
for $2 \leq r \leq k (k \geq 3)$ and $2 \leq n \leq r$. Together with \eqref{Theta_L_2}, \eqref{Theta_tx1x2_Dert} and \eqref{Theta_1zeta_Dert}, one derives from the induction for $1 \leq r \leq k (k \geq 3)$ that
\begin{equation}\label{Theta_nzeta_Dert}
  \begin{aligned}
    & \| \Theta (t) \|^2_{\mathbb{H}^r_{l,n}(\R^3_+)} + \sum_{j=0}^r \sum_{2\gamma + |\beta| = j-n} \mathscr{E}_{B,n}^{\gamma, \beta} (\Theta) (t) \\
    & \qquad \qquad \qquad \qquad + \int_0^t \| \p_{\zeta} \Theta (s) \|^2_{\mathbb{H}^r_{l,n}(\R^3_+)} + \lambda \| \nabla_{\bar{x}} \Theta (s) \|^2_{\mathbb{H}^r_{l,n}(\R^3_+)} \d s \\
    \le & C(\tau, E_{k+1}) \Big( \epsilon_0 \mathscr{A}_{B,n}^r (\Theta) (t) + \epsilon_0 \mathscr{A}_{B,1}^r (\Theta) (t) + \| \Theta (0) \|^2_{\mathbb{H}^r_l(\R^3_+)} \Big) \\
    & + C(\tau, E_{k+1}) \Big( \sum_{0 \leq \tilde{n} \leq n} \int_0^t \| \Theta (s) \|^2_{\mathbb{H}^r_{l,\tilde{n}}(\R^3_+)} + \sup_{t \in [0,\tau]} \| \tilde{g}^\sigma (t) \|^2_{\mathbb{H}^{r+1}_l(\R^3_+)} \Big)
  \end{aligned}
\end{equation}
for $2 \leq r \leq k (k \geq 3)$, $2 \leq n \leq r$ and some small $\epsilon_0 > 0$ to be determined, where
\begin{equation}
  \begin{aligned}
    \mathscr{A}_{B,n}^r (\Theta) (t) = & \sum_{j=0}^r \sum_{2\gamma + |\beta| = j-n} \mathscr{A}_{B,n}^{\gamma, \beta} (\Theta) (t) \\
    = & \sum_{j=0}^r \sum_{2\gamma + |\beta| = j-n} \sum_{i=1}^2 \delta_j (\beta = e_i) \int_0^t \| \p_t^{\gamma + \frac{n+1}{2}} \Theta |_{\zeta=0} (s) \|^2_{L^2(\R^2)} \d s \,.
  \end{aligned}
\end{equation}
Finally, collecting the estimates \eqref{Theta_L_2}, \eqref{Theta_tx1x2_Dert}, \eqref{Theta_1zeta_Dert} and \eqref{Theta_nzeta_Dert} implies
  \begin{align}\label{Hr_l}
    \no & \| \Theta (t) \|^2_{\mathbb{H}^r_l(\R^3_+)} + \int_0^t \| \p_{\zeta} \Theta (s) \|^2_{\mathbb{H}^r_l(\R^3_+)} + \lambda \| \nabla_{\bar{x}} \Theta (s) \|^2_{\mathbb{H}^r_l(\R^3_+)} \d s \\
    \no & + \widetilde{\mathscr{E}}_{B,r} (\Theta) (t) + \lambda \int_0^t \| \nabla_{\bar{x}} \Theta |_{\zeta=0} (s) \|^2_{\mathbb{H}^{r-1}(\R^2)} \d s \\
    \no & \le C(\tau, E_{k+1}) \Big( \epsilon_0 \sum_{n=1}^r \mathscr{A}_{B,n}^r (\Theta) (t) + \| \Theta (0) \|^2_{\mathbb{H}^r_l(\R^3_+)} \Big) \\
    & + C(\tau, E_{k+1}) \Big( \int_0^t \| \Theta (s) \|^2_{\mathbb{H}^r_l(\R^3_+)} + \sup_{t \in [0,\tau]} \| \tilde{g}^\sigma (t) \|^2_{\mathbb{H}^{r+1}_l(\R^3_+)} \Big)
  \end{align}
for any $1 \leq r \leq k (k \geq 3)$, where
\begin{equation}
  \begin{aligned}
    \widetilde{\mathscr{E}}_{B,r} (\Theta) (t) = & \sum_{j=0}^r \sum_{n=2}^r \sum_{2\gamma + |\beta| = j-n} \mathscr{E}_{B,n}^{\gamma, \beta} (\Theta) (t) + \sum_{n=0,1} \delta_n \| \Theta |_{\zeta=0} (t) \|^2_{\mathbb{H}^{r-1}(\R^2)} \\
    & + \sum_{n=0,1} \int_0^t (1 - \delta_n) \| \Theta |_{\zeta=0} (s) \|^2_{\mathbb{H}^r(\R^2)} + \delta_n \lambda \| \nabla_{\bar{x}} \Theta |_{\zeta=0} (s) \|^2_{\mathbb{H}^{r-1}(\R^2)} \d s \,.
  \end{aligned}
\end{equation}
Here $\delta_n$ is defined in Lemma \ref{Lemma_IV7B}, i.e., $\delta_n = 0$ for even $n$ and $\delta_n = 1$ for odd $n$. In particular, we set $r = k \geq 3$ in \eqref{Hr_l}. We need to control the quantity $\sum_{n=1}^k \mathscr{A}_{B,n}^k (\Theta) (t)$ in the right-hand side of \eqref{Hr_l}, namely,
\begin{equation*}
  \begin{aligned}
    \sum_{n=1}^k \mathscr{A}_{B,n}^k (\Theta) (t) = \sum_{n=1}^k \sum_{j=0}^k \sum_{2\gamma + |\beta| = j-n} \sum_{i=1}^2 \delta_j (\beta = e_i) \int_0^t \| \p_t^{\gamma + \frac{n+1}{2}} \Theta |_{\zeta=0} (s) \|^2_{L^2(\R^2)} \d s \,.
  \end{aligned}
\end{equation*}
It is easily derived from the trace inequality \eqref{Trace-1} in Lemma \ref{Lemma_Trace} that
\begin{equation}
  \begin{aligned}
    \sum_{n=1}^k \sum_{j=0}^{k-1} \sum_{2\gamma + |\beta| = j-n} \sum_{i=1}^2 \delta_j (\beta = e_i) \int_0^t \| \p_t^{\gamma + \frac{n+1}{2}} \Theta |_{\zeta=0} (s) \|^2_{L^2(\R^2)} \d s \\
    \le C \int_0^t \| \Theta (s) \|^2_{\mathbb{H}^k_l(\R^3_+)} \d s \,.
  \end{aligned}
\end{equation}
We now consider the case $j=k$. Recall the definition $\delta_k (\beta = e_i)$ in \eqref{delta-r-beta}, i.e., $\delta_k (\beta = e_i) = 1$ for even $k$ and $\beta = e_i$, otherwise $\delta_k (\beta = e_i) = 0$. If $k$ is odd,
\begin{equation*}
  \begin{aligned}
    \sum_{n=1}^k \sum_{2\gamma + |\beta| = k-n} \sum_{i=1}^2 \delta_k (\beta = e_i) \int_0^t \| \p_t^{\gamma + \frac{n+1}{2}} \Theta |_{\zeta=0} (s) \|^2_{L^2(\R^2)} \d s = 0 \,.
  \end{aligned}
\end{equation*}
If $k$ is even,
\begin{equation*}
  \begin{aligned}
    \sum_{n=1}^k \sum_{2\gamma + |\beta| = k-n} \sum_{i=1}^2 \delta_k (\beta = e_i) \int_0^t \| \p_t^{\gamma + \frac{n+1}{2}} \Theta |_{\zeta=0} (s) \|^2_{L^2(\R^2)} \d s \\
    \leq C_k' \int_0^t \| \Theta |_{\zeta=0} (s) \|^2_{\mathbb{H}^r(\R^2)} \d s \le C_k \widetilde{\mathscr{E}}_{B,r} (\Theta) (t)
  \end{aligned}
\end{equation*}
for some constant $C_k > 0$. There therefore holds
\begin{equation}\label{epsilon0_term}
  \begin{aligned}
    C (\tau, E_{k+1}) \epsilon_0 \sum_{n=1}^k & \mathscr{A}_{B,n}^k (\Theta) (t) \\
    & \le \epsilon_0 C(\tau, E_{k+1}) C_k \widetilde{\mathscr{E}}_{B,r} (\Theta) (t) + C (\tau, E_{k+1}) \int_0^t \| \Theta (s) \|^2_{\mathbb{H}^k_l(\R^3_+)} \d s \,.
  \end{aligned}
\end{equation}
We now choose an $\epsilon_0 \in (0, 1)$ such that
\begin{equation*}
  \begin{aligned}
    0 < \epsilon_0 C(\tau, E_{k+1}) C_k \le \tfrac{1}{2} \,.
  \end{aligned}
\end{equation*}
By plugging \eqref{epsilon0_term} into the inequality \eqref{Hr_l} with $r = k$, one has
  \begin{align}\label{Uniform_Theta1}
    \no & \| \Theta (t) \|^2_{\mathbb{H}^k_l(\R^3_+)} + \int_0^t \| \p_{\zeta} \Theta (s) \|^2_{\mathbb{H}^k_l(\R^3_+)} + \lambda \| \nabla_{\bar{x}} \Theta (s) \|^2_{\mathbb{H}^k_l(\R^3_+)} \d s \\
    \no & + \tfrac{1}{2} \widetilde{\mathscr{E}}_{B,k} (\Theta) (t) + \lambda \int_0^t \| \nabla_{\bar{x}} \Theta |_{\zeta=0} (s) \|^2_{\mathbb{H}^{k-1}(\R^2)} \d s \\
    & \le C(\tau, E_{k+1}) \Big( \| \Theta (0) \|^2_{\mathbb{H}^k_l(\R^3_+)} + \int_0^t \| \Theta (s) \|^2_{\mathbb{H}^k_l(\R^3_+)} \d s + \sup_{t \in [0,\tau]} \| \tilde{g}^\sigma (t) \|^2_{\mathbb{H}^{k+1}_l(\R^3_+)} \Big) \,.
  \end{align}

By the analogous arguments of the bound \eqref{Uniform_Theta1} for $\Theta$, one can also obtain the uniform bounds for $\mho$
\begin{equation}\label{Uniform_mho1}
  \begin{aligned}
    & \| \mho (t) \|^2_{\mathbb{H}^k_l(\R^3_+)} + \int_0^t \| \p_{\zeta} \mho (s) \|^2_{\mathbb{H}^k_l(\R^3_+)} + \lambda \| \nabla_{\bar{x}} \mho (s) \|^2_{\mathbb{H}^k_l(\R^3_+)} \d s \\
    & + \tfrac{1}{2} \widetilde{\mathscr{E}}_{B,k} (\mho) (t) + \lambda \int_0^t \| \nabla_{\bar{x}} \mho |_{\zeta=0} (s) \|^2_{\mathbb{H}^{k-1}(\R^2)} \d s \\
    & \le C(\tau, E_{k+1}) \Big( \| \mho (0) \|^2_{\mathbb{H}^k_l(\R^3_+)} + \int_0^t \| ( \mho, \Theta) (s) \|^2_{\mathbb{H}^k_l(\R^3_+)} \d s + \sup_{t \in [0,\tau]} \| \tilde{f}^\sigma (t) \|^2_{\mathbb{H}^{k+1}_l(\R^3_+)} \Big) \,.
  \end{aligned}
\end{equation}
Here the quantity $\int_0^t \| ( \mho, \Theta) (s) \|^2_{\mathbb{H}^k_l(\R^3_+)} \d s$ is resulted from the term $\tfrac{\nabla_{\bar{x}} p^0}{3 T^0} \Theta$ in the $\mho$-equation of \eqref{Theta_Equa}. From \eqref{Uniform_Theta1} and \eqref{Uniform_mho1},
\begin{equation*}
  \begin{aligned}
    & \| ( \mho, \Theta) (t) \|^2_{\mathbb{H}^k_l(\R^3_+)} + \int_0^t \| \p_{\zeta} ( \mho, \Theta) (s) \|^2_{\mathbb{H}^k_l(\R^3_+)} + \lambda \| \nabla_{\bar{x}} ( \mho, \Theta) (s) \|^2_{\mathbb{H}^k_l(\R^3_+)} \d s \\
    & + \tfrac{1}{2} \widetilde{\mathscr{E}}_{B,k} (\mho, \Theta) (t) + \lambda \int_0^t \| \nabla_{\bar{x}} ( \mho, \Theta) |_{\zeta=0} (s) \|^2_{\mathbb{H}^{k-1}(\R^2)} \d s \\
    & \le C(\tau, E_{k+1}) \Big( \| ( \mho, \Theta) (0) \|^2_{\mathbb{H}^k_l(\R^3_+)} + \int_0^t \| ( \mho, \Theta) (s) \|^2_{\mathbb{H}^k_l(\R^3_+)} \d s + \sup_{t \in [0,\tau]} \| ( \tilde{f}^\sigma, \tilde{g}^\sigma ) (t) \|^2_{\mathbb{H}^{k+1}_l(\R^3_+)} \Big) \,,
  \end{aligned}
\end{equation*}
which implies by the Gr\"onwall inequality that
\begin{equation}\label{Uniform_Theta}
  \begin{aligned}
    & \| ( \mho, \Theta) (t) \|^2_{\mathbb{H}^k_l(\R^3_+)} + \int_0^t \| \p_{\zeta} ( \mho, \Theta) (s) \|^2_{\mathbb{H}^k_l(\R^3_+)} + \lambda \| \nabla_{\bar{x}} ( \mho, \Theta) (s) \|^2_{\mathbb{H}^k_l(\R^3_+)} \d s \\
    & + \widetilde{\mathscr{E}}_{B,k} (\mho, \Theta) (t) + \lambda \int_0^t \| \nabla_{\bar{x}} ( \mho, \Theta) |_{\zeta=0} (s) \|^2_{\mathbb{H}^{k-1}(\R^2)} \d s \\
    & \le C(\tau, E_{k+1}) \Big( \| ( \mho, \Theta) (0) \|^2_{\mathbb{H}^k_l(\R^3_+)} + \sup_{t \in [0,\tau]} \| ( \tilde{f}^\sigma, \tilde{g}^\sigma ) (t) \|^2_{\mathbb{H}^{k+1}_l(\R^3_+)} \Big) \,.
  \end{aligned}
\end{equation}
Based on the uniform estimates \eqref{Uniform_Theta}, we can first take the limit $\sigma \to 0+$, and then $\lambda \to 0+$. Then, by using \eqref{Theta_Appro} and \eqref{Uniform_Theta}, the bound \eqref{theta-bnd} holds, and Proposition \ref{Prop-Prandtl} is proved. Here the details are omitted for simplicity of presentation.
\end{proof}

Now we give the following lemma to control the boundary term $IV_{7B}$.

\begin{lemma}\label{Lemma_IV7B}
	Let $e_1 = (1,0), e_2 = (0,1) \in \mathbb{N}^2$. For $2 \leq n \leq r$ and $2 \gamma + |\beta| + n = r$, let $IV_{7B}$ be given in \eqref{IV_7_Decomposed}, i.e.,
	\begin{equation*}
	  \begin{aligned}
	    IV_{7B} = \int_0^t \int_{\R^2}  \p_t^\gamma \p_{\bar{x}}^\beta (\tilde{\kappa} \p_\zeta^{n+1} \Theta) (1+\zeta)^{l_r} \p_t^\gamma \p_{\bar{x}}^\beta \p_\zeta^n \Theta |_{\zeta=0}^{\zeta=\frac{3}{\sigma}} \d \bar{x} \d s \,.
	  \end{aligned}
	\end{equation*}
	There hold
	\begin{equation}\label{IV7B_Bnd}
	  \begin{aligned}
	    IV_{7B} \le & - \mathscr{E}_{B, n}^{\gamma, \beta} (\Theta) (t) + \epsilon_0 \mathscr{A}_{B,n}^{\gamma, \beta} (\Theta) (t) + C(\tau, E_{k+1}) \sup_{t \in [0,\tau]} \| \tilde{g}^\sigma (t) \|^2_{\mathbb{H}^{r+1}_l (\R^3_+)} \\
	    + & C(E_{k+1}) \Big( \| \Theta (0) \|^2_{\mathbb{H}^r_l(\R^3_+)} + \sum_{p=0, 1} \| \Theta (t) \|^2_{\mathbb{H}^r_{l,p}(\R^3_+)} + \int_0^t \sum_{p=0, 1} \| \Theta (s) \|^2_{\mathbb{H}^r_{l,p}(\R^3_+)} \d s \Big)
	  \end{aligned}
	\end{equation}
	for some small $\epsilon_0 > 0$ to be determined, where the symbols $\mathscr{A}_{B,n}^{\gamma, \beta} (\Theta) (t)$ and $\mathscr{E}_{B, n}^{\gamma, \beta} (\Theta) (t)$ are
	\begin{equation}
	  \begin{aligned}
	    \mathscr{A}_{B,n}^{\gamma, \beta} (\Theta) (t) = \sum_{i=1}^2 \delta_r (\beta = e_i) \int_0^t \| \p_t^{\gamma + \frac{n+1}{2}} \Theta |_{\zeta=0} (s) \|^2_{L^2(\R^2)} \d s \,,
	  \end{aligned}
	\end{equation}
	and
	\begin{equation}
	  \begin{aligned}
	    \mathscr{E}_{B, n}^{\gamma, \beta} (\Theta) (t) = & \tfrac{1}{2} \delta_n c_n \| \p_t^\gamma \p_{\bar{x}}^\beta (\p_t - \lambda \Delta_{\bar{x}})^\frac{n-1}{2} \Theta |_{\zeta=0} (t) \|^2_{L^2(\R^2)} \\
	    & + \lambda \delta_n c_n \int_0^t \| \nabla_{\bar{x}} \p_t^\gamma \p_{\bar{x}}^\beta (\p_t - \lambda \Delta_{\bar{x}})^\frac{n-1}{2} \Theta |_{\zeta=0} (s) \|^2_{L^2(\R^2)} \d s \\
	    & + \tfrac{1}{2} (1 - \delta_n) c_0 \int_0^t \p_t^\gamma \p_{\bar{x}}^\beta (\p_t - \lambda \Delta_{\bar{x}})^\frac{n}{2} \Theta |_{\zeta=0} (s) \|^2_{L^2(\R^2)} \d s
	  \end{aligned}
	\end{equation}
	for some constant $c_0, c_n > 0$. Here $\delta_n = 1$ for odd $n$ and $\delta_n = 0$ for even $n$. The symbol $\delta_r (\beta = e_i)$ is defined in \eqref{delta-r-beta}. Namely, its value is 1 if $r$ is even and $\beta = e_i$. Otherwise, its value will vanish.
\end{lemma}

\begin{remark}\label{Remark_IV7B}
	When $r$ is even and $|\beta| = 1$, the relation $2 \gamma + |\beta| + n = r$ means that $n$ must be odd. In other words, the quantity $\mathscr{A}_{B,n} (\Theta) (t)$ will be zero for the case of even $n$. Furthermore, in order to obtain the bound of $IV_{7B}$ in \eqref{IV7B_Bnd}, the source term $\tilde{g}^\sigma$ shall assume $(r+1)$-order derivatives (i.e., in $\mathbb{H}^{r+1}_l (\R^3_+)$). The details can be found in \eqref{IV7B-2} below. Here a first order time derivative is equivalently regarded as a second order $\bar{x}$-derivative.
\end{remark}

\begin{proof}
From Lemma \ref{Lemma_Theta_BC_infty}, we see that for $2 \le n \le r$, $\p_t^\gamma \p_{\bar{x}}^\beta (\tilde{\kappa} \p_\zeta^{n+1} \Theta) (1+\zeta)^{l_r} \p_t^\gamma \p_{\bar{x}}^\beta \p_\zeta^n \Theta |_{\zeta = \frac{3}{\sigma}} = 0$. One therefore has
\begin{equation}
  \begin{aligned}
    IV_{7B} = -\int_0^t \int_{\R^2} \p_t^\gamma \p_{\bar{x}}^\beta (\tilde{\kappa} \p_\zeta^{n+1} \Theta) \p_t^\gamma \p_{\bar{x}}^\beta \p_\zeta^n \Theta |_{\zeta =0} \d \bar{x} \d s \,.
  \end{aligned}
\end{equation}

\vspace*{2mm}

{\bf Case 1. $n \in [2, r]$ is even.} From Lemma \ref{Lemma_Theta_BC_infty} and remarks \eqref{BC-Thicksim}, we have $2 \gamma + |\beta| + n = r$ and
\begin{equation}
\begin{aligned}
IV_{7B} = & - \int_0^t \int_{\R^2} \tilde{\kappa} R_\theta \{ \p_t^\gamma \p_{\bar{x}}^\beta (\mathscr{L} + \tilde{\mathscr{L}})^\frac{n}{2} \Theta \}^2 |_{\zeta=0} \d \bar{x} \d s \\
& - \int_0^t \int_{\R^2} \p_t^\gamma \p_{\bar{x}}^\beta (\mathscr{L} + \tilde{\mathscr{L}})^\frac{n}{2} \Theta \cdot ( \mathscr{L}^\star_5 \Theta + \mathscr{L}^\star_6 \tilde{g}^\sigma + \tilde{\kappa} R_\theta \mathscr{L}^\star_7 \tilde{g}^\sigma ) |_{\zeta=0} \d \bar{x} \d s \\
& - \int_0^t \int_{\R^2} ( \mathscr{L}^\star_5 \Theta + \mathscr{L}^\star_6 \tilde{g}^\sigma ) \cdot \mathscr{L}^\star_7 \tilde{g}^\sigma |_{\zeta=0} \d \bar{x} \d s = : IV_{7B}^{e1} + IV_{7B}^{e2} + IV_{7B}^{e3} \,,
\end{aligned}
\end{equation}
where
\begin{align}
\mathscr{L}^\star_5 \Theta = & \p_t^\gamma \p_{\bar{x}}^\beta [ \tilde{\kappa} (\mathscr{L} + \tilde{\mathscr{L}})^\frac{n}{2} , R_\theta ] \Theta + [ \p_t^\gamma \p_{\bar{x}}^\beta , \tilde{\kappa} R_\theta ] (\mathscr{L} + \tilde{\mathscr{L}})^\frac{n}{2} \Theta \,, \\
\no \mathscr{L}^\star_6 \tilde{g}^\sigma = & \sum_{i=0}^{\frac{n}{2} - 1} \p_t^\gamma \p_{\bar{x}}^\beta \{ \tilde{\kappa} (\mathscr{L} + \tilde{\mathscr{L}})^i \hat{\mathscr{L}} \p_{\zeta}^{n-1-2i} \tilde{g}^\sigma \} \,, \ \mathscr{L}^\star_7 \tilde{g}^\sigma = \sum_{i=0}^{\frac{n}{2} - 1} \p_t^\gamma \p_{\bar{x}}^\beta (\mathscr{L} + \tilde{\mathscr{L}})^i \hat{\mathscr{L}} \p_{\zeta}^{n-2-2i} \tilde{g}^\sigma \,.
\end{align}
From the lower bounds \eqref{LowBnd-1} and \eqref{LowBnd-2}, we know $\tilde{\kappa} R_\theta \geq \tilde{\kappa}_0 R_\theta^0 := c_0' > 0$. Then the quantity $IV_{7B}^{e1}$ can be bounded by
\begin{equation}\label{IV7B-e1}
\begin{aligned}
IV_{7B}^{e1} \le - c_0' \int_0^t \| \p_t^\gamma \p_{\bar{x}}^\beta (\mathscr{L} + \tilde{\mathscr{L}})^\frac{n}{2} \Theta |_{\zeta=0} (s) \|^2_{L^2(\R^2)} \d s \,.
\end{aligned}
\end{equation}
One further observes that $\mathscr{L}^\star_5$ is an at most $(r-1)$-order differential operator, $\mathscr{L}^\star_6$ is $(r-1)$-order, and $\mathscr{L}^\star_7$ is  $(r-2)$-order. There therefore hold
\begin{equation}\label{IV7B-e2}
\begin{aligned}
IV_{7B}^{e2} \le & \tfrac{1}{2} c_0' \int_0^t \| \p_t^\gamma \p_{\bar{x}}^\beta (\mathscr{L} + \tilde{\mathscr{L}})^\frac{n}{2} \Theta |_{\zeta=0} (s) \|^2_{L^2(\R^2)} \d s \\
& + C(E_{k+1}) \int_0^t \sum_{p=0,1} \| \Theta (s) \|^2_{\mathbb{H}^r_{l,p}(\R^3_+)} + \sum_{0 \leq \tilde{n} \leq n} \| \tilde{g}^\sigma (s) \|^2_{\mathbb{H}^r_{l,\tilde{n}}(\R^3_+)} \d s
\end{aligned}
\end{equation}
and
\begin{equation}\label{IV7B-e3}
\begin{aligned}
IV_{7B}^{e3} \le C(E_{k+1}) \int_0^t \sum_{p=0,1} \| \Theta (s) \|^2_{\mathbb{H}^r_{l,p}(\R^3_+)} + \sum_{0 \leq \tilde{n} \leq n} \| \tilde{g}^\sigma (s) \|^2_{\mathbb{H}^r_{l,\tilde{n}}(\R^3_+)} \d s  \,,
\end{aligned}
\end{equation}
where the trace inequalities \eqref{Trace-1}-\eqref{Trace-2} in Lemma \ref{Lemma_Trace} have been utilized. Consequently, the bounds \eqref{IV7B-e1}, \eqref{IV7B-e2} and \eqref{IV7B-e3} tell us
\begin{align*}
 IV_{7B} \le & - \tfrac{1}{2} c_0' \int_0^t \| \p_t^\gamma \p_{\bar{x}}^\beta (\mathscr{L} + \tilde{\mathscr{L}})^\frac{n}{2} \Theta |_{\zeta=0} (s) \|^2_{L^2(\R^2)} \d s \\
 & + C(E_{k+1}) \int_0^t \sum_{p=0,1} \| \Theta (s) \|^2_{\mathbb{H}^r_{l,p}(\R^3_+)} + \sum_{0 \leq \tilde{n} \leq n} \| \tilde{g}^\sigma (s) \|^2_{\mathbb{H}^r_{l,\tilde{n}}(\R^3_+)} \d s \,,
\end{align*}
when $n \in [2, r]$ is even. Recalling the definitions of $\mathscr{L}$ and $\tilde{\mathscr{L}}$ in Lemma \ref{Lemma_Theta_BC_infty}, one easily computes
\begin{equation*}
  \begin{aligned}
    \p_t^\gamma \p_{\bar{x}}^\beta (\mathscr{L} + \tilde{\mathscr{L}})^\frac{n}{2} = \tilde{\kappa}^{-\frac{n}{2}} \p_t^\gamma \p_{\bar{x}}^\beta (\p_t - \lambda \Delta_{\bar{x}})^\frac{n}{2} + \mathscr{L}^\star_0 \,,
  \end{aligned}
\end{equation*}
where
\begin{equation*}
  \begin{aligned}
    \mathscr{L}^\star_0 = & [\p_t^\gamma \p_{\bar{x}}^\beta , \tilde{\kappa}^{-\frac{n}{2}}] (\p_t - \lambda \Delta_{\bar{x}})^\frac{n}{2} + \sum_{p=1}^\frac{n}{2} C_\frac{n}{2}^p \p_t^\gamma \p_{\bar{x}}^\beta \mathscr{L}^{\frac{n}{2} - p} \tilde{\mathscr{L}}^p + \p_t^\gamma \p_{\bar{x}}^\beta \big( \mathscr{L}^\frac{n}{2} - \tilde{\kappa}^{-\frac{n}{2}} (\p_t - \lambda \Delta_{\bar{x}})^\frac{n}{2} \big)
  \end{aligned}
\end{equation*}
is an $(r-1)$-order differential operator. Together with the trace inequality \eqref{Trace-1} in Lemma \ref{Lemma_Trace}, one has
\begin{equation*}
  \begin{aligned}
    c_0'' \| \p_t^\gamma \p_{\bar{x}}^\beta (\p_t - \lambda \Delta_{\bar{x}})^\frac{n}{2} \Theta |_{\zeta=0} (s) \|^2_{L^2(\R^2)} \le \| \tilde{\kappa}^{-\frac{n}{2}} \p_t^\gamma \p_{\bar{x}}^\beta (\p_t - \lambda \Delta_{\bar{x}})^\frac{n}{2} \Theta |_{\zeta=0} (s) \|^2_{L^2(\R^2)} \\
    \le 2 \| \p_t^\gamma \p_{\bar{x}}^\beta (\mathscr{L} + \tilde{\mathscr{L}})^\frac{n}{2} \Theta |_{\zeta=0} (s) \|^2_{L^2(\R^2)} + 2 \| \mathscr{L}^\star_0 \Theta |_{\zeta=0} (s) \|^2_{L^2(\R^2)} \\
    \le 2 \| \p_t^\gamma \p_{\bar{x}}^\beta (\mathscr{L} + \tilde{\mathscr{L}})^\frac{n}{2} \Theta |_{\zeta=0} (s) \|^2_{L^2(\R^2)} + C(E_{k+1}) \sum_{p=0,1} \| \Theta (s) \|^2_{\mathbb{H}^r_{l,p}(\R^3_+)}
  \end{aligned}
\end{equation*}
for some constant $c_0'' > 0$. We thereby know
\begin{equation}\label{IV7B_Even}
  \begin{aligned}
    IV_{7B} \le & - \tfrac{1}{2} c_0 \int_0^t \p_t^\gamma \p_{\bar{x}}^\beta (\p_t - \lambda \Delta_{\bar{x}})^\frac{n}{2} \Theta |_{\zeta=0} (s) \|^2_{L^2(\R^2)} \d s \\
    & + C(E_{k+1}) \int_0^t \sum_{p=0,1} \| \Theta (s) \|^2_{\mathbb{H}^r_{l,p}(\R^3_+)} + \sum_{0 \leq \tilde{n} \leq n} \| \tilde{g}^\sigma (s) \|^2_{\mathbb{H}^r_{l,\tilde{n}}(\R^3_+)} \d s \,,
  \end{aligned}
\end{equation}
where $c_0 = \frac{1}{2} c_0' c_0'' > 0$.

\vspace*{2mm}

{\bf Case 2. $n \in [2, r]$ is odd.} It is further derived from Lemma \ref{Lemma_Theta_BC_infty} and remarks \eqref{BC-Thicksim} that
  \begin{align}\label{IV_7B_Odd}
    \no IV_{7B} = & -\int_0^t \int_{\R^2} \p_t^\gamma \p_{\bar{x}}^\beta \big\{ \tilde{\kappa} (\mathscr{L} + \tilde{\mathscr{L}} )^\frac{n+1}{2} \Theta \big\} \cdot \p_t^\gamma \p_{\bar{x}}^\beta (\mathscr{L} + \tilde{\mathscr{L}})^\frac{n-1}{2} (R_\theta \Theta) |_{\zeta=0} \d \bar{x} \d s \\
    \no & - \sum_{j=0}^\frac{n-3}{2} \int_0^t \int_{\R^2} \p_t^\gamma \p_{\bar{x}}^\beta [ \tilde{\kappa} (\mathscr{L}+\tilde{\mathscr{L}})^\frac{n+1}{2} \Theta ] \cdot \p_t^\gamma \p_{\bar{x}}^\beta (\mathscr{L}+\tilde{\mathscr{L}})^j \hat{\mathscr{L}} \p_{\zeta}^{n-2-2j} \tilde{g}^\sigma |_{\zeta=0} \d \bar{x} \d s \\
    \no & - \sum_{i=0}^\frac{n-1}{2} \int_0^t \int_{\R^2} \p_t^\gamma \p_{\bar{x}}^\beta [ \tilde{\kappa} (\mathscr{L}+\tilde{\mathscr{L}})^i \hat{\mathscr{L}} \p_{\zeta}^{n-1-2i} \tilde{g}^\sigma ] \\
    \no & \qquad \qquad \qquad \qquad \qquad \qquad \times \p_t^\gamma \p_{\bar{x}}^\beta (\mathscr{L}+\tilde{\mathscr{L}})^\frac{n-1}{2} (R_\theta \Theta) |_{\zeta=0} \d \bar{x} \d s \\
    & - \sum_{i = 0}^\frac{n-1}{2}\sum_{j=0}^\frac{n-3}{2} \int_0^t \int_{\R^2} \p_t^\gamma \p_{\bar{x}}^\beta [ \tilde{\kappa} (\mathscr{L}+\tilde{\mathscr{L}})^i \hat{\mathscr{L}} \p_{\zeta}^{n-1-2i} \tilde{g}^\sigma ] \\
    \no & \qquad \qquad \qquad \qquad \qquad \qquad \times \p_t^\gamma \p_{\bar{x}}^\beta (\mathscr{L}+\tilde{\mathscr{L}})^j \hat{\mathscr{L}} \p_{\zeta}^{n-2-2j} \tilde{g}^\sigma |_{\zeta=0} \d \bar{x} \d s \\
    \no = : & IV_{7B}^1 + IV_{7B}^2 + IV_{7B}^3 + IV_{7B}^4 \,,
  \end{align}
  where the differential operators $\mathscr{L}$, $\hat{\mathscr{L}}$ and $\tilde{\mathscr{L}}$ are defined in \eqref{L-scr} and \eqref{L-scr-tilde}. While $2 \gamma + |\beta| + n = r$, one observes that the differential operators $\p_t^\gamma \p_{\bar{x}}^\beta [\tilde{\kappa} (\mathscr{L}+\tilde{\mathscr{L}})^i \hat{\mathscr{L}} \p_{\zeta}^{n-1-2i}]$ ($0 \leq i \leq \frac{n-1}{2}$) and $\p_t^\gamma \p_{\zeta}^\beta (\mathscr{L}+\tilde{\mathscr{L}})^\frac{n-1}{2}$ are all ($r-1$)-order, and $\p_t^\gamma \p_{\bar{x}}^\beta (\mathscr{L}+\tilde{\mathscr{L}})^j \hat{\mathscr{L}} \p_{\zeta}^{n-2-2j}$ ($0 \leq j \leq \frac{n-3}{2}$) are $(r-2)$-order differential operators. Consequently, the trace inequality \eqref{Trace-1} in Lemma \ref{Lemma_Trace} tells us
  \begin{equation}\label{IV_7B_3+4}
    \begin{aligned}
      | IV_{7B}^3 + IV_{7B}^4 | \le C(E_{k+1}) \int_0^t \| \Theta (s) \|^2_{\mathbb{H}^r_{l,0}(\R^3_+)} + \| \Theta (s) \|^2_{\mathbb{H}^r_{l,1}(\R^3_+)} + \sum_{\tilde{n}=0}^n \| \tilde{g}^\sigma (s) \|^2_{\mathbb{H}^r_{l, \tilde{n}}(\R^3_+)} \d s \,.
    \end{aligned}
  \end{equation}

  Next we control the quantity $IV_{7B}^2$. One first computes that
  \begin{equation}\label{Split-1}
    \begin{aligned}
      \p_t^\gamma \p_{\bar{x}}^\beta [\tilde{\kappa} (\mathscr{L}+\tilde{\mathscr{L}})^\frac{n+1}{2}] = & \tilde{\kappa}^{- \frac{n-1}{2}} \p_t^\gamma \p_{\bar{x}}^\beta (\p_t - \lambda \Delta_{\bar{x}})^\frac{n+1}{2} + \mathscr{L}^\star_2 \\
      = & \tilde{\kappa}^{- \frac{n-1}{2}} \p_t^{\gamma + \frac{n+1}{2}} \p_{\bar{x}}^\beta + \tilde{\kappa}^{- \frac{n-1}{2}} \Delta_{\bar{x}} \mathscr{L}^\star_1 + \mathscr{L}^\star_2 \,,
    \end{aligned}
  \end{equation}
  where
  \begin{equation}
    \begin{aligned}
      \mathscr{L}^\star_1 = & \sum_{j=1}^\frac{n+1}{2} C_\frac{n+1}{2}^j (- \lambda)^j \p_t^{\gamma + \frac{n+1}{2} - j} \p_{\bar{x}}^\beta \Delta_{\bar{x}}^{j-1} \,, \\
      \mathscr{L}^\star_2 = & [\p_t^\gamma \p_{\bar{x}}^\beta, \tilde{\kappa}^{- \frac{n-1}{2}}] (\p_t - \lambda \Delta_{\bar{x}})^\frac{n+1}{2} + \sum_{j=1}^\frac{n+1}{2} C_\frac{n+1}{2}^j \p_t^\gamma \p_{\bar{x}}^\beta ( \tilde{\kappa} \mathscr{L}^{\frac{n+1}{2}-j} \tilde{\mathscr{L}}^j ) \\
      & + \p_t^\gamma \p_{\bar{x}}^\beta \big\{ \tilde{\kappa} \mathscr{L}^\frac{n+1}{2} - \tilde{\kappa}^{- \frac{n-1}{2}} (\p_t - \lambda \Delta_{\bar{x}})^\frac{n+1}{2} \big\} \,.
    \end{aligned}
  \end{equation}
  We observe that $\mathscr{L}^\star_1$ is $(r-1)$-order differential operator, and $\mathscr{L}^\star_2$ is $r$-order differential operator, where the highest order derivatives must contain $\nabla_{\bar{x}}$. Consequently, the trace inequalities \eqref{Trace-1}-\eqref{Trace-2} in Lemma \ref{Lemma_Trace} imply that
    \begin{align}\label{IV7B-1}
      \no & - \sum_{j=0}^\frac{n-3}{2} \int_0^t \int_{\R^2} \Big( \tilde{\kappa}^{- \frac{n-1}{2}} \Delta_{\bar{x}} \mathscr{L}^\star_1 + \mathscr{L}^\star_2 \Big) \Theta \cdot \p_t^\gamma \p_{\bar{x}}^\beta (\mathscr{L}+\tilde{\mathscr{L}})^j \hat{\mathscr{L}} \p_{\zeta}^{n-2-2j} \tilde{g}^\sigma |_{\zeta=0} \d \bar{x} \d s \\
      \no = & \sum_{j=0}^\frac{n-3}{2} \int_0^t \int_{\R^2} \nabla_{\bar{x}} \mathscr{L}^\star_1 \Theta \cdot \nabla_{\bar{x}} \Big\{ \tilde{\kappa}^{- \frac{n-1}{2}} \p_t^\gamma \p_{\bar{x}}^\beta (\mathscr{L}+\tilde{\mathscr{L}})^j \hat{\mathscr{L}} \p_{\zeta}^{n-2-2j} \tilde{g}^\sigma \Big\} |_{\zeta=0} \d \bar{x} \d s \\
      \no & - \sum_{j=0}^\frac{n-3}{2} \int_0^t \int_{\R^2} \mathscr{L}^\star_2 \Theta \cdot \p_t^\gamma \p_{\bar{x}}^\beta (\mathscr{L}+\tilde{\mathscr{L}})^j \hat{\mathscr{L}} \p_{\zeta}^{n-2-2j} \tilde{g}^\sigma |_{\zeta=0} \d \bar{x} \d s \\
      \le & C(E_{k+1}) \int_0^t \| \Theta (s) \|^2_{\mathbb{H}^r_{l,0}(\R^3_+)} + \| \Theta (s) \|^2_{\mathbb{H}^r_{l,1}(\R^3_+)} + \sum_{\tilde{n}=0}^n \| \tilde{g}^\sigma (s) \|^2_{\mathbb{H}^r_{l, \tilde{n}}(\R^3_+)} \d s \,.
    \end{align}
  For the quantity $IV_{7B}^2$, it remains to estimate the term
  \begin{equation*}
    \begin{aligned}
      IV_{7B}^{21} (\beta) = : - \sum_{j=0}^\frac{n-3}{2} \int_0^t \int_{\R^2} \tilde{\kappa}^{- \frac{n-1}{2}} \p_t^{\gamma + \frac{n+1}{2}} \p_{\bar{x}}^\beta \Theta \cdot \p_t^\gamma \p_{\bar{x}}^\beta (\mathscr{L}+\tilde{\mathscr{L}})^j \hat{\mathscr{L}} \p_{\zeta}^{n-2-2j} \tilde{g}^\sigma |_{\zeta=0} \d \bar{x} \d s \,.
    \end{aligned}
  \end{equation*}
  If $|\beta| \geq 1$, similar arguments in \eqref{IV7B-1} reduce to
    \begin{align*}
      IV_{7B}^{21} (\beta) = & \sum_{j=0}^\frac{n-3}{2} \int_0^t \int_{\R^2} \p_t^{\gamma + \frac{n-1}{2}} \p_{\bar{x}}^{\beta} \Theta \cdot \p_t \big\{ \tilde{\kappa}^{- \frac{n-1}{2}} \p_t^\gamma \p_{\bar{x}}^\beta (\mathscr{L}+\tilde{\mathscr{L}})^j \hat{\mathscr{L}} \p_{\zeta}^{n-2-2j} \tilde{g}^\sigma \big\} |_{\zeta=0} \d \bar{x} \d s \\
      & - \sum_{j=0}^\frac{n-3}{2} \int_{\R^2} \p_t^{\gamma + \frac{n-1}{2}} \p_{\bar{x}}^{\beta} \Theta \cdot \tilde{\kappa}^{- \frac{n-1}{2}} \p_t^\gamma \p_{\bar{x}}^\beta (\mathscr{L}+\tilde{\mathscr{L}})^j \hat{\mathscr{L}} \p_{\zeta}^{n-2-2j} \tilde{g}^\sigma |_{\zeta=0} |_{s=0}^{s=t} \d \bar{x} \\
      \le & C(E_{k+1}) \int_0^t \| \Theta (s) \|^2_{\mathbb{H}^r_{l,0}(\R^3_+)} + \| \Theta (s) \|^2_{\mathbb{H}^r_{l,1}(\R^3_+)} + \sum_{\tilde{n}=0}^n \| \tilde{g}^\sigma (s) \|^2_{\mathbb{H}^r_{l, \tilde{n}}(\R^3_+)} \d s \\
      & + C(E_{k+1}) \big( \| \Theta (0) \|^2_{\mathbb{H}^r_l (\R^3_+)} + \sum_{p=0, 1} \| \Theta (t) \|^2_{\mathbb{H}^r_{l, p}(\R^3_+)} + \sup_{t \in [0,\tau]} \sum_{\tilde{n} = 0}^n \| \tilde{g}^\sigma (t) \|^2_{\mathbb{H}^r_{l,\tilde{n}}(\R^3_+)} \big) \,.
    \end{align*}
  If $\beta = 0$, we have $2 \gamma = r -n$. Here $r$ is also odd. Then $\p_t \big\{ \tilde{\kappa}^{- \frac{n-1}{2}} \p_t^\gamma (\mathscr{L}+\tilde{\mathscr{L}})^j \hat{\mathscr{L}} \p_{\zeta}^{n-2-2j} \big\}$ $(0 \leq j \leq \frac{n-3}{2})$ are $r$-order differential operators. It is therefore derived from the trace inequality \eqref{Trace-1} that
  \begin{equation}\label{g-one-more-order2}
    \begin{aligned}
      \| \p_t \big\{ \tilde{\kappa}^{- \frac{n-1}{2}} \p_t^\gamma (\mathscr{L}+\tilde{\mathscr{L}})^j \hat{\mathscr{L}} \p_{\zeta}^{n-2-2j} \big\} \tilde{g}^\sigma |_{\zeta=0} (t) \|^2_{L^2(\R^2)} \le C(E_{k+1}) \sum_{\tilde{n} = 0}^n \| \tilde{g}^\sigma (t) \|^2_{\mathbb{H}^{r+1}_{l, \tilde{n}}(\R^3_+)} \,.
    \end{aligned}
  \end{equation}
  In other words, the source term $\tilde{g}^\sigma (t, \bar{x}, \zeta)$ shall be in $\mathbb{H}^{r+1}_l(\R^3_+)$. Consequently,
    \begin{align*}
      IV_{7B}^{21} (\beta) = & \sum_{j=0}^\frac{n-3}{2} \int_0^t \int_{\R^2} \p_t^{\gamma + \frac{n-1}{2}} \Theta \cdot \p_t \big\{ \tilde{\kappa}^{- \frac{n-1}{2}} \p_t^\gamma (\mathscr{L}+\tilde{\mathscr{L}})^j \hat{\mathscr{L}} \p_{\zeta}^{n-2-2j} \tilde{g}^\sigma \big\} |_{\zeta=0} \d \bar{x} \d s \\
      & - \sum_{j=0}^\frac{n-3}{2} \int_{\R^2} \p_t^{\gamma + \frac{n-1}{2}} \Theta \cdot \tilde{\kappa}^{- \frac{n-1}{2}} \p_t^\gamma (\mathscr{L}+\tilde{\mathscr{L}})^j \hat{\mathscr{L}} \p_{\zeta}^{n-2-2j} \tilde{g}^\sigma |_{\zeta=0} |_{s=0}^{s=t} \d \bar{x} \\
      \le & C(E_{k+1}) \int_0^t \| \Theta (s) \|^2_{\mathbb{H}^r_{l,0}(\R^3_+)} + \| \Theta (s) \|^2_{\mathbb{H}^r_{l,1}(\R^3_+)} + \sum_{\tilde{n}=0}^n \| \tilde{g}^\sigma (s) \|^2_{\mathbb{H}^{r+1}_{l, \tilde{n}}(\R^3_+)} \d s \\
      + & C(E_{k+1}) \big( \| \Theta (0) \|^2_{\mathbb{H}^r_l (\R^3_+)} + \sum_{p=0, 1} \| \Theta (t) \|^2_{\mathbb{H}^r_{l, p}(\R^3_+)} + \sup_{t \in [0,\tau]} \sum_{\tilde{n} = 0}^n \| \tilde{g}^\sigma (t) \|^2_{\mathbb{H}^r_{l,\tilde{n}}(\R^3_+)} \big) \,.
    \end{align*}
  In summary, one sees that for all $\beta \in \mathbb{N}^2$ with $2 \gamma + |\beta| + n = r$
  \begin{equation}\label{IV7B-2}
    \begin{aligned}
      IV_{7B}^{21} (\beta) \le C(E_{k+1}) \Big( \sum_{p=0, 1} \| \Theta (t) \|^2_{\mathbb{H}^r_{l, p}(\R^3_+)} + \int_0^t \sum_{p=0, 1} \| \Theta (s) \|^2_{\mathbb{H}^r_{l, p}(\R^3_+)} \d s \qquad \qquad \\
      + \| \Theta (0) \|^2_{\mathbb{H}^r_l (\R^3_+)} + (1+t) \sup_{t \in [0,\tau]} \sum_{\tilde{n} = 0}^n \| \tilde{g}^\sigma (t) \|^2_{\mathbb{H}^{r+1}_{l,\tilde{n}}(\R^3_+)} \Big) \,.
    \end{aligned}
  \end{equation}
  Then, \eqref{IV7B-1}-\eqref{IV7B-2} reduce to
  \begin{equation}\label{IV_7B_2}
    \begin{aligned}
      IV_{7B}^2 \le C(E_{k+1}) \Big( \sum_{p=0, 1} \| \Theta (t) \|^2_{\mathbb{H}^r_{l, p}(\R^3_+)} + \int_0^t \sum_{p=0, 1} \| \Theta (s) \|^2_{\mathbb{H}^r_{l, p}(\R^3_+)} \d s \qquad \qquad \\
      + \| \Theta (0) \|^2_{\mathbb{H}^r_l (\R^3_+)} + (1+t) \sup_{t \in [0,\tau]} \sum_{\tilde{n} = 0}^n \| \tilde{g}^\sigma (t) \|^2_{\mathbb{H}^{r+1}_{l,\tilde{n}}(\R^3_+)} \Big) \,.
    \end{aligned}
  \end{equation}

  We next control the term $IV_{7B}^1$ in \eqref{IV_7B_Odd}. By direct calculation,
  \begin{equation}\label{Split-2}
    \begin{aligned}
      \p_t^\gamma \p_{\bar{x}}^\beta (\mathscr{L} & + \tilde{\mathscr{L}} )^\frac{n-1}{2} (R_\theta \, \cdot) \\
      = & R_\theta \tilde{\kappa}^{- \frac{n-1}{2}} \p_t^\gamma \p_{\bar{x}}^\beta (\p_t - \lambda \Delta_{\bar{x}})^\frac{n-1}{2} + [\p_{\bar{x}}^\beta, R_\theta \tilde{\kappa}^{- \frac{n-1}{2}}] \p_t^\gamma (\p_t - \lambda  \Delta_{\bar{x}})^\frac{n-1}{2} + \mathscr{L}^\star_3 \,,
    \end{aligned}
  \end{equation}
  where
  \begin{equation}
    \begin{aligned}
      \mathscr{L}^\star_3 = & \sum_{q=1}^\frac{n-1}{2} C_\frac{n-1}{2}^q \p_{\bar{x}}^\beta \big\{ R_\theta \p_t^\gamma \mathscr{L}^{\frac{n-1}{2} - q} \tilde{\mathscr{L}}^q \big\} + \p_{\bar{x}}^\beta [ \p_t^\gamma (\mathscr{L} + \tilde{\mathscr{L}} )^\frac{n-1}{2} , R_\theta ] \\
      & + \p_{\bar{x}}^\beta \Big\{ R_\theta \p_t^\gamma \big( \mathscr{L}^\frac{n-1}{2} - \tilde{\kappa}^{- \frac{n-1}{2}} (\p_t - \lambda \Delta_{\bar{x}})^\frac{n-1}{2} \big) + R_\theta [\p_t^\gamma, \tilde{\kappa}^{- \frac{n-1}{2}}] (\p_t - \lambda \Delta_{\bar{x}})^\frac{n-1}{2} \Big\}
    \end{aligned}
  \end{equation}
  is an $(r-2)$-order differential operator. Moreover, due to the definition of $\tilde{\mathscr{L}}$ in \eqref{L-scr-tilde}, the highest order ($(r-2)$-order) derivatives of $\mathscr{L}^\star_3$ must contain $\nabla_{\bar{x}}$. Together with \eqref{Split-1} and \eqref{Split-2}, we can decompose the term $IV_{7B}^1$ as
  \begin{equation}
    \begin{aligned}
      IV_{7B}^1 = & - \int_0^t \int_{\R^2} \mathscr{L}^\star_2 \Theta \cdot \p_t^\gamma \p_{\bar{x}}^\beta (\mathscr{L} + \tilde{\mathscr{L}})^\frac{n-1}{2} (R_\theta \Theta) |_{\zeta=0} \d \bar{x} \d s \\
      & - \int_0^t \int_{\R^2} \tilde{\kappa}^{- \frac{n-1}{2}} \p_t^\gamma \p_{\bar{x}}^\beta (\p_t - \lambda \Delta_{\bar{x}})^\frac{n+1}{2} \Theta \cdot \mathscr{L}^\star_3 \Theta |_{\zeta=0} \d \bar{x} \d s \\
      & - \int_0^t \int_{\R^2} \tilde{\kappa}^{- \frac{n-1}{2}} \p_t^\gamma \p_{\bar{x}}^\beta (\p_t - \lambda \Delta_{\bar{x}})^\frac{n+1}{2} \Theta \\
      & \qquad \qquad \qquad \qquad \times [\p_{\bar{x}}^\beta, R_\theta \tilde{\kappa}^{- \frac{n-1}{2}}] \p_t^\gamma (\p_t - \lambda  \Delta_{\bar{x}})^\frac{n-1}{2} \Theta |_{\zeta=0} \d \bar{x} \d s \\
      & - \int_0^t \int_{\R^2} R_\theta \tilde{\kappa}^{-n+1} \p_t^\gamma \p_{\bar{x}}^\beta (\p_t - \lambda \Delta_{\bar{x}})^\frac{n+1}{2} \Theta \cdot \p_t^\gamma \p_{\bar{x}}^\beta (\p_t - \lambda \Delta_{\bar{x}})^\frac{n-1}{2} \Theta |_{\zeta=0} \d \bar{x} \d s \\
      = : & IV_{7B}^{11} + IV_{7B}^{12} + IV_{7B}^{13} + IV_{7B}^{14} \,.
    \end{aligned}
  \end{equation}
  It is thereby implied by the trace inequalities \eqref{Trace-1}-\eqref{Trace-2} in Lemma \ref{Lemma_Trace} that
  \begin{equation}\label{IV7B_11}
    \begin{aligned}
      IV_{7B}^{11} \le & C(E_{k+1}) \int_0^t \| \Theta (s) \|^2_{\mathbb{H}^r_{l,0}(\R^3_+)} + \| \Theta (s) \|^2_{\mathbb{H}^r_{l,1}(\R^3_+)} \d s \,,
    \end{aligned}
  \end{equation}
  and
    \begin{align}\label{IV7B_12}
      \no IV_{7B}^{12} = & - \int_{\R^2} \tilde{\kappa}^{- \frac{n-1}{2}} \p_t^\gamma \p_{\bar{x}}^\beta (\p_t - \lambda \Delta_{\bar{x}})^\frac{n-1}{2} \Theta \cdot \mathscr{L}^\star_3 \Theta |_{\zeta=0} |_{s=0}^{s=t} \d \bar{x} \\
      \no & + \int_0^t \int_{\R^2} \p_t^\gamma \p_{\bar{x}}^\beta (\p_t - \lambda \Delta_{\bar{x}})^\frac{n-1}{2} \Theta \cdot (\p_t + \lambda \Delta_{\bar{x}}) \big\{ \tilde{\kappa}^{- \frac{n-1}{2}} \mathscr{L}^\star_3 \Theta \big\} |_{\zeta=0} \d \bar{x} \d s \\
      \le & C(E_{k+1}) \Big( \| \Theta (0) \|^2_{\mathbb{H}^r_l(\R^3_+)} + \sum_{p=0, 1} \| \Theta (t) \|^2_{\mathbb{H}^r_{l,p}(\R^3_+)} + \int_0^t \sum_{p=0, 1} \| \Theta (s) \|^2_{\mathbb{H}^r_{l,p}(\R^3_+)} \d s \Big) \,.
    \end{align}
  If $\beta = 0$, we have $IV_{7B}^{13} = 0$. If $| \beta | \geq 2$, we assume $\beta \geq 2 e_1 = 2 (1, 0)$. One then has
  \begin{equation*}
    \begin{aligned}
      IV_{7B}^{13} = & - \int_0^t \int_{\R^2} \p_t^\gamma \p_{\bar{x}}^{\beta - 2 e_1} (\p_t - \lambda \Delta_{\bar{x}})^\frac{n+1}{2} \Theta \\
      & \qquad \qquad \qquad \times \p_{\bar{x}}^{2 e_1} \big\{ \tilde{\kappa}^{- \frac{n-1}{2}} [\p_{\bar{x}}^\beta, R_\theta \tilde{\kappa}^{- \frac{n-1}{2}}] \p_t^\gamma (\p_t - \lambda  \Delta_{\bar{x}})^\frac{n-1}{2} \Theta \big\} |_{\zeta=0} \d \bar{x} \d s \\
      \le & C(E_{k+1}) \int_0^t \| \Theta (s) \|^2_{\mathbb{H}^r_{l,0}(\R^3_+)} + \| \Theta (s) \|^2_{\mathbb{H}^r_{l,1}(\R^3_+)} \d s \,.
    \end{aligned}
  \end{equation*}
  If $|\beta| = 1$, the facts that $2 \gamma + |\beta| + n = r$ and $n \in [2, r]$ is odd imply that $r$ is even. We then have $2 \gamma + n = r - 1$ and
    \begin{align*}
      IV_{7B}^{13} = & \sum_{i=1}^2 \delta_r (\beta = e_i) \int_0^t \int_{\R^2} \p_t^{\gamma + \frac{n+1}{2}} \Theta \\
      & \qquad \qquad \qquad \qquad \times \p_{\bar{x}}^{e_i} \big\{ \tilde{\kappa}^{- \frac{n-1}{2}} \p_{\bar{x}}^{e_i} (R_\theta \tilde{\kappa}^{- \frac{n-1}{2}}) \p_t^\gamma (\p_t - \lambda  \Delta_{\bar{x}})^\frac{n-1}{2} \Theta \big\} |_{\zeta=0} \d \bar{x} \d s \\
      & + \sum_{i=1}^2 \delta_r (\beta = e_i) \int_0^t \int_{\R^2} \sum_{q=1}^\frac{n+1}{2} C_\frac{n+1}{2}^q (- \lambda)^q \nabla_{\bar{x}} \p_t^{\gamma + \frac{n+1}{2} - q} \Delta_{\bar{x}}^{q-1} \Theta \\
      & \qquad \qquad \qquad \cdot \nabla_{\bar{x}} \p_{\bar{x}}^{e_i} \big\{ \tilde{\kappa}^{- \frac{n-1}{2}} \p_{\bar{x}}^{e_i} (R_\theta \tilde{\kappa}^{- \frac{n-1}{2}}) \p_t^\gamma (\p_t - \lambda  \Delta_{\bar{x}})^\frac{n-1}{2} \Theta \big\} |_{\zeta=0} \d \bar{x} \d s \\
      \le & \epsilon_0 \sum_{i=1}^2 \delta_r (\beta = e_i) \int_0^t \| \p_t^{\gamma + \frac{n+1}{2}} \Theta |_{\zeta=0} (s) \|^2_{L^2(\R^2)} \d s \\
      & + C(E_{k+1}) \int_0^t \| \Theta (s) \|^2_{\mathbb{H}^r_{l,0}(\R^3_+)} + \| \Theta (s) \|^2_{\mathbb{H}^r_{l,1}(\R^3_+)} \d s
    \end{align*}
  for some small $\epsilon_0 > 0$ to be determined, where the last inequality is derived from Lemma \ref{Lemma_Trace} and $\delta_r (\beta = e_i)$ is given in \eqref{delta-r-beta}. As a result,
  \begin{equation}\label{IV7B_13}
    \begin{aligned}
       IV_{7B}^{13} \le & \epsilon_0 \sum_{i=1}^2 \delta_r (\beta = e_i) \int_0^t \| \p_t^{\gamma + \frac{n+1}{2}} \Theta |_{\zeta=0} (s) \|^2_{L^2(\R^2)} \d s \\
       & + C(E_{k+1}) \int_0^t \| \Theta (s) \|^2_{\mathbb{H}^r_{l,0}(\R^3_+)} + \| \Theta (s) \|^2_{\mathbb{H}^r_{l,1}(\R^3_+)} \d s \,.
    \end{aligned}
  \end{equation}
  The lower bounds \eqref{LowBnd-1} and \eqref{LowBnd-2} yield that there is a $c_n > 0$ such that
  \begin{equation*}
    \begin{aligned}
      R_\theta \tilde{\kappa}^{-n+1} \geq c_n > 0 \,.
    \end{aligned}
  \end{equation*}
  Together with Lemma \ref{Lemma_Trace}, one obtains
    \begin{align}\label{IV7B_14}
      \no IV_{7B}^{14} = & - \int_{\R^2} \tfrac{1}{2} R_\theta \tilde{\kappa}^{-n+1} \{ \p_t^\gamma \p_{\bar{x}}^\beta (\p_t - \lambda \Delta_{\bar{x}})^\frac{n-1}{2} \Theta \}^2 |_{\zeta=0} \d \bar{x} |_{s=0}^{s=t} \\
      \no & - \int_0^t \int_{\R^2} \lambda R_\theta \tilde{\kappa}^{-n+1} \big| \nabla_{\bar{x}} \p_t^\gamma \p_{\bar{x}}^\beta (\p_t - \lambda \Delta_{\bar{x}})^\frac{n-1}{2} \Theta \big|^2 |_{\zeta=0} \d \bar{x} \d s \\
      \no & + \int_0^t \int_{\R^2} \tfrac{1}{2} (\p_t + \lambda \Delta_{\bar{x}}) (R_\theta \tilde{\kappa}^{-n+1}) \{ \p_t^\gamma \p_{\bar{x}}^\beta (\p_t - \lambda \Delta_{\bar{x}})^\frac{n-1}{2} \Theta \}^2 |_{\zeta=0} \d \bar{x} \d s \\
      \no \le & - \tfrac{1}{2} c_n \| \p_t^\gamma \p_{\bar{x}}^\beta (\p_t - \lambda \Delta_{\bar{x}})^\frac{n-1}{2} \Theta |_{\zeta=0} (t) \|^2_{L^2(\R^2)} \\
      \no & - \lambda c_n \int_0^t \| \nabla_{\bar{x}} \p_t^\gamma \p_{\bar{x}}^\beta (\p_t - \lambda \Delta_{\bar{x}})^\frac{n-1}{2} \Theta |_{\zeta=0} (s) \|^2_{L^2(\R^2)} \d s \\
      & + C(E_{k+1}) \Big( \| \Theta (0) \|^2_{\mathbb{H}^r_l(\R^3_+)} + \int_0^t \sum_{p=0, 1} \| \Theta (s) \|^2_{\mathbb{H}^r_{l,p}(\R^3_+)} \d s \Big) \,.
    \end{align}
  From \eqref{IV7B_11}, \eqref{IV7B_12}, \eqref{IV7B_13} and \eqref{IV7B_14}, one thereby knows
    \begin{align}\label{IV_7B-1}
      \no IV_{7B}^1 \le & - \tfrac{1}{2} c_n \| \p_t^\gamma \p_{\bar{x}}^\beta (\p_t - \lambda \Delta_{\bar{x}})^\frac{n-1}{2} \Theta |_{\zeta=0} (t) \|^2_{L^2(\R^2)} \\
      \no & - \lambda c_n \int_0^t \| \nabla_{\bar{x}} \p_t^\gamma \p_{\bar{x}}^\beta (\p_t - \lambda \Delta_{\bar{x}})^\frac{n-1}{2} \Theta |_{\zeta=0} (s) \|^2_{L^2(\R^2)} \d s \\
      & + \epsilon_0 \sum_{i=1}^2 \delta_r (\beta = e_i) \int_0^t \| \p_t^{\gamma + \frac{n+1}{2}} \Theta |_{\zeta=0} (s) \|^2_{L^2(\R^2)} \d s \\
      \no & + C(E_{k+1}) \Big( \| \Theta (0) \|^2_{\mathbb{H}^r_l(\R^3_+)} + \sum_{p=0, 1} \| \Theta (t) \|^2_{\mathbb{H}^r_{l,p}(\R^3_+)} + \int_0^t \sum_{p=0, 1} \| \Theta (s) \|^2_{\mathbb{H}^r_{l,p}(\R^3_+)} \d s \Big) \,.
    \end{align}
  Consequently, if $n \in [2,r]$ is odd,
    \begin{align}\label{IV7B_Odd}
      \no IV_{7B} & \le - \tfrac{1}{2} c_n \| \p_t^\gamma \p_{\bar{x}}^\beta (\p_t - \lambda \Delta_{\bar{x}})^\frac{n-1}{2} \Theta |_{\zeta=0} (t) \|^2_{L^2(\R^2)} \\
      \no & - \lambda c_n \int_0^t \| \nabla_{\bar{x}} \p_t^\gamma \p_{\bar{x}}^\beta (\p_t - \lambda \Delta_{\bar{x}})^\frac{n-1}{2} \Theta |_{\zeta=0} (s) \|^2_{L^2(\R^2)} \d s \\
      \no & + \epsilon_0 \sum_{i=1}^2 \delta_r (\beta = e_i) \int_0^t \| \p_t^{\gamma + \frac{n+1}{2}} \Theta |_{\zeta=0} (s) \|^2_{L^2(\R^2)} \d s \\
      & + C(E_{k+1}) (1+t) \sup_{t \in [0,\tau]} \sum_{\tilde{n} = 0}^n \| \tilde{g}^\sigma (t) \|^2_{\mathbb{H}^{r+1}_{l, \tilde{n}}(\R^3_+)} \\
      \no & + C(E_{k+1}) \Big( \| \Theta (0) \|^2_{\mathbb{H}^r_l(\R^3_+)} + \sum_{p=0, 1} \| \Theta (t) \|^2_{\mathbb{H}^r_{l,p}(\R^3_+)} + \int_0^t \sum_{p=0, 1} \| \Theta (s) \|^2_{\mathbb{H}^r_{l,p}(\R^3_+)} \d s \Big)
    \end{align}
    for some small $\epsilon_0 > 0$ to be determined. Finally, the bounds \eqref{IV7B_Even} and \eqref{IV7B_Odd} finish the proof of Lemma \ref{Lemma_IV7B}.
\end{proof}


\section*{Acknowledgement}
This work is supported by the grants from the National Natural Foundation of
China under contract Nos. 11971360 and 11731008. This work is also supported by the Strategic Priority Research Program of Chinese Academy of Sciences, Grant No. XDA25010404. Prof. Xiongfeng Yang and his student Tianfang Wu carefully read our first draft and pointed out a few mistakes. We take this opportunity to express appreciation to them.  

\bigskip

\bibliography{reference}

\end{document}